\documentclass[12pt,a4paper]{amsart}
\usepackage{amsfonts}
\usepackage{amsthm}
\usepackage{amsmath}
\usepackage{mathtools}
\usepackage{amssymb}
\usepackage{amscd}
\usepackage[latin2]{inputenc}
\usepackage{t1enc}
\usepackage[mathscr]{eucal}
\usepackage{indentfirst}
\usepackage{graphicx}
\usepackage{graphics}
\usepackage{dirtytalk}
\usepackage{pict2e}
\usepackage{epic}
\usepackage{verbatim}
\usepackage{color,soul}
\usepackage{etoolbox}
\usepackage{enumitem}
\numberwithin{equation}{section}
\usepackage[margin=2.9cm]{geometry}
\usepackage{epstopdf}
\usepackage{tikz-cd}
\usepackage[all]{xy}
 
\usepackage{hyperref}
\graphicspath{ {figures/} }
\usepackage{array}
\theoremstyle{plain}
\newtheorem{Th}{Theorem}[section]
\newtheorem{Lemma}[Th]{Lemma}

\newtheorem{Prop}[Th]{Proposition}
\newtheorem{Claim}[Th]{Claim}

 \theoremstyle{definition}
\newtheorem{Def}[Th]{Definition}

\newtheorem{Rem}[Th]{Remark}
\newtheorem{?}[Th]{Problem}
\newtheorem{Ex}[Th]{Example}

\DeclareMathOperator{\sign}{sign}
\newcommand{\interior}[1]{%
  {\kern0pt#1}^{\mathrm{o}}%
}
\newcommand{\R}{\mathbb{R}}
\newcommand{\C}{\mathbb{C}}
\newcommand{\Z}{\mathbb{Z}}
\newcommand{\D}{\mathcal{D}}
\newcommand{\Q}{\mathbb{Q}}
\newcommand{\bd}{\partial}

\newcommand{\Li}{\mathcal{L}}
\begin{document}

\title[The Contact Banach-Mazur distance]{The Contact Banach-Mazur distance \\ and large scale geometry of overtwisted contact forms}

\author[]{Thomas Melistas}

\address{Department of Mathematics, University of Georgia, Athens, GA 30602} 

\email{thomas.melistas@uga.edu}

 %\subjclass[2010]{Primary: 05C??. Secondary: 05C??}

 %\keywords{sample paper} 

\begin{abstract} In the symplectic realm, a distance between open starshaped domains in Liouville manifolds was recently defined. This is the symplectic Banach-Mazur distance. It was proposed by Ostrover and Polterovich and developed by Ostrover, Polterovich, Usher, Gutt, Zhang and Stojisavljevi\'c. The natural question is, can an analogous distance in the contact realm be defined? One idea is to define the distance on contact hypersurfaces of Liouville manifolds and another one on contact forms supporting isomorphic contact structures. Rosen and Zhang recently defined such a distance working with manifolds that are prequantizations of Liouville manifolds. They also considered a distance on contact forms supporting the same contact structure on a contact manifold $Y$. This allowed them to view the space of contact forms supporting isomorphic contact structures on a manifold $Y$ as a pseudometric space, study its properties, and derive interesting results. In this work, we do something similar, yet the distance we define is less restrictive. Moreover, viewing contact homology algebra as a persistence module, focusing purely on the overtwisted case and exploiting the fact that the contact homology of overtwisted contact structures vanishes, allows us to bi-Lipschitz embed part of the 2-dimensional Euclidean space into the space of overtwisted contact forms supporting a given contact structure on a smooth closed manifold $Y$.
\end{abstract}

\maketitle

\tableofcontents

\listoffigures

\section{Introduction} 

Let $(Y,\xi)$ be a closed, co-oriented contact manifold of dimension $2n-1$. A consequence of Gray's stability theorem is that the space of contact structures up to diffeomorphism $\Xi(Y)/Diff(Y)$ on an odd dimensional closed manifold Y is discrete. In other words, there are no non-trivial deformations of a contact structure. The elements of this space are contactomorphism classes of contact structures defined on the smooth manifold $Y$. Although from a topological point of view there is no difference for the contact structure up to isotopy, the dynamics depend highly on the particular 1-form co-orienting the manifold $Y$, hence they can be vastly different. Looking at everything from the dynamics perspective, allows to ask the questions like how \say{large} is a class in $\Xi(Y)/Diff(Y)$, or in other words how far apart are two representatives of the same contactomorphism class. The idea on how to measure their distance comes from an analogy with the case of symplectic manifolds. In particular, for open Liouville domains the idea, which is inspired by convex geometry and the Banach-Mazur distance, is to look at the optimal way to interleave them, see \cite{2018arXiv181100734U}. \\ \par

 Interleaving is usually achieved by rescaling, so the way to rescale is a key issue that needs to be addressed when attempting to define a Banach-Mazur distance. Some very interesting attempts have already been made. One can work as in section 1.2.1 of \cite{2020arXiv200105094R} where Rosen and Zhang work with the case of fiberwise star-shaped domains $U$ in the contactization of a Liouville manifold $W$, namely $W \times S^1$ and with subdomains of contact manifolds that are boundaries of Liouville domains. Otherwise, one can look at their setup in section 1.2.2,  in which they define a distance between contact forms supporting isomorphic contact structures. This distance can in turn be used to define a distance between closed Liouville fillable contact manifolds.  In subsection \ref{JD}, we briefly recall part of their work, mainly focusing on the distance between contact forms as it is most relevant here and compare results to ours.   \\ \par 
 
 The most natural space in which two contactomorphic contact manifolds should be interleaved appears to be their common symplectization. This idea stems from the fact that Liouville manifolds decompose as the union of the symplectization of a hypersurface of restricted contact type $Y$ (which can be viewed as the boundary of the corresponding Liouville domain bounded by $Y$) and their core or skeleton. To be more specific, the symplectization of the boundary $Y$ of a Liouville domain $W$ sits naturally in the completion $\widehat{W}$ of the Liouville domain. The key advantage is that the symplectization also works in the case of absence of a core, namely the overtwisted or more generally in the non-fillable case. The main difference between a fillable and a non-fillable $Y$ appears to be the existence of a core, so this indicates that the distance may be defined when restricting to the same contactomorphism class in both the fillable and the non-fillable case. In this work, we approach the problem similarly to the second setup of Rosen and Zhang \cite{2020arXiv200105094R}, i.e. the case of contact forms supporting isomorphic contact structures, yet we allow more flexibility using appropriate embeddings in the symplectization, which we call cs-embeddings (because a Contact manifold is embedded into its Symplectization), resulting in a distance which does not only depend on conformal factors. The distance that we use here is the contact Banach-Mazur distance which is defined in subsection \ref{ImportantDefinitions}. We denote it by $d_{CBM}$. \\ \par
 
 The main result of this work is a mix of quantitative, dynamical and topological nature. Let $\mathcal{C}_{ot}^{Y,\xi}$ denote the space of contact forms on the closed manifold $Y$, supporting the co-oriented overtwisted contact structure $\xi$, which are positive with respect to the co-orientation. Let also $\mathbb{H}$  be the lower half-space $\mathbb{H}$ in $\R^2$ and $d_{\infty}$ denote the metric induced from the norm $||\cdot||_{\infty}$ in $\R^2$.
 
 \begin{Th}[Main Theorem]
 There exists a bi-Lipschitz embedding $F: (\mathbb{H},d_{\infty}) \rightarrow (\mathcal{C}_{ot}^{Y,\xi},d_{CBM})$.
 \end{Th}
 
 The core of the proof of the main theorem is to be able to control the action level for which the identity becomes an exact element in the filtered contact homology algebra. In 3 dimensions, our goal will be to modify a construction by Wendl in \cite{Wend}, so as to be able to know precisely what is the action level for which the unit in the contact homology algebra $CH(Y,\lambda_{ot})$ of an overtwisted contact manifold becomes exact. This is the subject of section \ref{Wendlwork}. In higher dimensions, we follow Bourgeois and Van Koert's approach from \cite{MR2646902} which uses the characterization of overtwisted contact manifolds as negatively stabilized open books. This is discussed in \ref{OBA}. As will be explained, in all dimensions, the action level for which the unit in the contact homology algebra becomes exact corresponds to the right endpoint of the largest finite bar in the barcode. \\ \par

 The importance of this control is also justified by the following well known observation. We know that the vanishing level of the class of the unit controls all other vanishing levels just by using Leibniz rule. This can be seen as follows. If $y$ represents a class in the contact homology algebra, then we need to find an element that maps to $y$ under $\partial$, i.e. show that any element is exact. If $x$ is the orbit bounding the pseudoholomorphic plane, then $\partial x=1$. Thus, using Leibniz rule we see that $\partial(xy)=(\partial x)y\pm x\partial y=y$ (note that $y$ is closed as it represents a class). If $y$ has action $\mathcal{A}(y)$, then $x\cdot y$ has action $\mathcal{A}(x)+\mathcal{A}(y)$, hence the vanishing level of the class $[y]$ is at most $\mathcal{A}(x)+\mathcal{A}(y)$, which shows that the length of its corresponding finite bar is $\mathcal{A}(x)+\mathcal{A}(y)-\mathcal{A}(y)=\mathcal{A}(x)=l$, hence in the case of contact homology algebra, the bar of the unit is the longest and most essential one. This is essentially another application of the argument used to show the vanishing of contact homology of overtwisted structures, if we have proven the existence of a unique pseudoholomorphic plane bounded by a Reeb orbit.
 
 \subsection{Organization}
 In subsection \ref{ImportantDefinitions} we provide the main definitions needed to study this work. In subsection \ref{MainResults} we describe our main results more thoroughly. In short, these amount to defining the contact Banach-Mazur pseudodistance between contact forms and using it to show that the lower half-space in $\R^2$ bi-Lipschitz embeds in the space of overtwisted contact forms supporting a given overtwisted contact structure. In subsection \ref{JD} we recall the relevant definitions from \cite{2020arXiv200105094R} where a different and more restrictive flavor of the distance is defined and compare their results to ours. In particular, using symplectic folding we exhibit the main differences between our definitions. Section \ref{Contact homology} is devoted to recalling the construction of (filtered) contact homology which viewed as a persistence module gave us the idea to produce the bi-Lipschitz embedding. Section \ref{pf} provides the proof of the main bi-Lipschitz embedding theorem, theorem \ref{quasiisometricembedding}. Section \ref{Higherdim} addresses the question of extending the result of the previous section to higher dimensions. Finally, section \ref{Remarkonhigherdegrees} describes the situation if one would like to attempt to bi-Lipschitz embed $\R^n$ into the space of overtwisted contact forms supporting isomorphic contact structures for $n>2$. There is no definite answer provided there.\\ \par

\textbf{Acknowledgements.}
I am very grateful to my advisor, Michael Usher, for his guidance, patience and support and for teaching me so many interesting things during this work. I would also like to thank Leonid Polterovich and Jun Zhang for helpful discussions and insightful comments. This work was partially supported by the NSF through the grant DMS-1509213.

\subsection{Definitions} \label{ImportantDefinitions} \text{} \\ \par

Throughout this paper, unless otherwise stated, $Y$ will be a $(2n-1)$-dimensional closed manifold with a co-oriented contact structure $\xi$. The contact Banach-Mazur distance is a distance between 2 co-orientation compatible contact forms on $Y$ having the same kernel, the contact hyperplane field $\xi$. Fixing a hypersurface of restricted contact type $Y_0$ (see definition \ref{restr}) in a Liouville manifold $W$, the distance can also be defined between contact hypersurfaces of restricted contact type that are in the image of the Liouville flow starting at $Y_0$ and flowing for either positive or negative time, not necessarily uniformly. \\ \par

In what follows, we are going to need the notion of the symplectization of a contact manifold $(Y,\xi)$. There are two versions of the definition. One of them does not require the choice of a contact form in order to be defined and it is a special line subbundle of the cotangent bundle of $Y$. The other one involves the choice of a co-orienting contact form for $Y$. This choice of global contact form for $Y$ yields a splitting of the symplectization $SY$ as a trivial principal $\mathbb{R}_+$-bundle, $\mathbb{R}_+ \times Y$. The advantage of the first one is obvious, while the advantage of the second one is of course the convenience of being able to perform hands on computations. We start with the latter one.

\begin{Def}
We define the symplectization of $(Y,\xi=ker(\lambda))$ to be the manifold $(S_\lambda Y=M=\mathbb{R}_+ \times Y,\omega=d(r\lambda))$, where $r$ is the real positive coordinate.
\end{Def}

One easily checks that $(S_\lambda Y,\omega)$ is a symplectic manifold. The fact that $\omega$ is closed is immediate since it is exact. The non-degeneracy is equivalent to the contact condition for $\lambda$. Note that implicitly in this definition we chose a global form $\lambda$ for $Y$. The choice-free definition of the symplectization is as follows. We fix a co-orientation for $\xi$.

\begin{Def}\label{ChoiceFree}
We define the symplectization of $(Y,\xi)$ to be $M=SY=\displaystyle\bigcup_{y \in Y}S_yY$, where $S_yY:=\Bigg\{ \beta \in T_y^*Y-\{0\} \Bigg| \begin{aligned}
   & \hspace{10pt}  ker(\beta)=\xi_y \text{ and } \beta>0 \cr & \hspace{10pt} \text{ on vectors positively transverse to } \xi_y
   \end{aligned} \Bigg\}$
   
\end{Def} \vspace{7pt} \par

The symplectization of $Y$ is a submanifold of its cotangent bundle $T^*Y$. As it is known, $T^*Y$ comes naturally equipped with the canonical, or tautological, or Liouville 1-form $\theta$ and it turns out that $d\theta$ is a symplectic form when restricted to $SY$. Thus, $(SY,d\theta)$ is a symplectic manifold. There is an identification between the two versions of the symplectization which sends $d(r \lambda) \rightarrow d\theta$. We will primarily work with the latter formulation of the definition as it requires no reference to a contact form $\lambda$, yet when concrete calculations are needed we will work using the former one.\\ \par

We denote by $L=L_\theta$ the Liouville vector field of the symplectization $SY$, i.e. the unique vector field on $(SY,\omega)$ satisfying $i_L\omega=i_L d\theta=\theta$. One can see that $(Y,\alpha)$ sits in a standard way as a contact hypersurface $Y_\alpha$ inside $SY$. This is understood as follows. $Y$ gets identified with the graph of its contact form $\alpha$ inside $SY$ which is viewed as a subbundle of $T^*Y$. It is important to remember that we will denote its identification by $Y_\alpha$. Different choices of a contact form with kernel $\xi$ yield different embeddings of $(Y,\xi)$ into $SY$ and different splittings of $SY$.
\\ \par

When we need to focus on the contact dynamics instead of just the contact structure itself, we use the notion of a strict contact manifold. A strict contact manifold is a closed manifold $Y$ equipped with a co-orienting contact form $\alpha$. The term is not standard in the literature, yet it is very useful here as we work with contact forms and not only contact structures.

\begin{Def}\label{contactembedding}
  By a cs-embedding of a strict contact manifold $(X,\alpha)$ to $(SY,d\theta)$ we mean an embedding $\phi:(X,\alpha) \rightarrow (SY,d\theta)$ with $\phi^*(\theta+\eta)=\alpha$, where $\eta$ is an exact, compactly supported 1-form on $SY$.
\end{Def}
Figure \ref{fig:contactemb} illustrates this definition.

\begin{figure}[h!]
    \centering
    \includegraphics[scale=1]{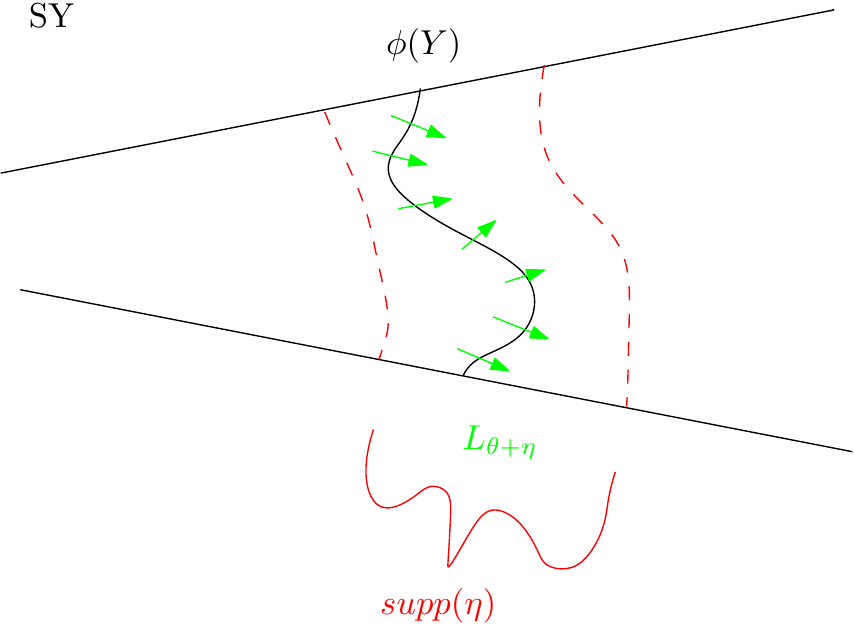}
    \caption{A cs-embedding}
    \label{fig:contactemb}
\end{figure}
 
 This definition is equivalent to the statement that $X$ embeds as a hypersurface transverse to the Liouville vector field defined by $i_L\omega=\theta+\eta$. Note that if $\eta=0$, then the Liouville vector field is simply $L=p\partial_p$, where $p$ are the cotangent fiber coordinates. So, the relaxed condition allows our embeddings to be transverse to some Liouville vector field dictated by $\theta + \eta$ and not just the standard one.\\ \par  
 
We denote the Liouville flow for time t by $L^{t}$.
Note that under this flow, the contact form $\lambda$ used to decompose $SY$ gets multiplied by $e^{\ln(r)}=r$. This is because the flow of the Liouville vector field $L$ conformally expands volume since by definition $\Li_L\omega=\omega$, where $\Li$ denotes the Lie derivative operation. The relationship between $t$ and $r$ is $t=\ln(r)$.\\ \par

A way to produce such cs-embeddings is to postcompose the standard embedding induced by $\alpha$, by a compactly supported symplectic isotopy $\Phi_t$.

\begin{Rem}\label{HamRem}
The symplectic isotopy is automatically Hamiltonian. The flux determines whether the symplectic isotopy is a Hamiltonian.
Looking at \cite{McDuff-Salamon}, section 10.2, the flux homomorphism corresponds to a homomorphism $\pi_1(SY)\rightarrow \R$, defined by
$$\gamma \mapsto \displaystyle\int_0^1\int^1_0\omega(X_t(\gamma(s)),\dot{\gamma}(s))dsdt, \quad \gamma: \R/\Z \rightarrow SY$$
where $X_t$ the vector field generating the isotopy. Lemma 10.2.1 in \cite{McDuff-Salamon} states that the right hand side above depends only on the homotopy class of $\gamma$ and the homotopy class of $\Phi_t$ with fixed endpoints. The value of Flux($\{\Phi_t\}$) on the loop $\gamma$ is the area swept by the loop under the symplectic isotopy $\Phi_t$. Any loop in the compact support of $\Phi_t$ is homotopic to one outside of the support of the symplectic isotopy $\Phi_t$. Thus, the flux of any loop is equal to zero and hence $\Phi_t$ is Hamiltonian.
\end{Rem}

As stated in the first paragraph of this section, the following definition will also be useful. 

\begin{Def}\label{restr}
 Let $(W,\theta)$ be an exact symplectic manifold. A codimension-one smooth submanifold $Y \subset W$ is said to be a restricted contact type hypersurface of $(W,\theta)$ if the
Liouville vector field $L$ is transverse to $Y$, i.e. $\forall y \in Y$ we have $L_y \notin T_yY$.
\end{Def}
One application of this definition will be, in the case that $Y$ is fillable, to relate $d_{CBM}$ with $d_c/d_{SBM}$ as we can view $Y$ as the contact type boundary of a starshaped domain.\par
We will define a partial order to the set of co-orientation compatible contact forms having kernel $\xi$. First, we provide some preliminary definitions.

\begin{Def}
 Let $Y_\beta$ be the standard embedding of $(Y,\beta)$ in $SY$ as the image of the form $\beta$ in $T^*Y$. Then we define $W(\beta)=\{p \in SY \mid 0 < p(v) \leq \beta(v), \forall v \in TY \text{  such that  } \beta(v)>0 \}$.
\end{Def}

If we choose a contact form, namely a splitting for $SY$, the above definition turns into the following one which is more suitable for calculations.

\begin{Def}
 Let $Y_\beta$ be the standard embedding of $(Y,\beta)$ in $(S_\beta Y=\mathbb{R}_+\times Y,d(r\beta))$. Then we define $W(\beta)=\{(s,y) \in SY \mid s \leq 1\}$.
\end{Def}

The partial order is defined as follows.

\begin{Def}\label{dcbmdefn}
$\alpha \prec \beta$ iff there is a cs-embedding in the sense of definition \ref{contactembedding}, $\phi:(Y,\alpha) \rightarrow SY$ such that $\phi(Y)\subset W(\beta)$.
\end{Def}

\begin{Rem}
Note that later we will use the notation $\preceq$ for another partial order, so a warning should be given here. $\preceq$ will be referring to Rosen-Zhang partial order.
\end{Rem}

Recall that an example of a cs-embedding is produced by postcomposing the standard embedding by symplectic isotopies. The case when this isotopy has empty support corresponds to the partial order $\preceq$ as the following example shows.

\begin{Ex}
If already $Y_\alpha \subset W(\beta)$, then we can take the support of the isotopy to be $\emptyset$. One such example is when there are contactomorphisms such that $\phi^*(\alpha)=h_1(y)\lambda$, $\psi^*(\beta)=h_2(y)\lambda$ and $h_1(y)\leq h_2(y)$, i.e. using notation that will be made more precise in section \ref{JD}, $\alpha \preceq \beta$. The obvious obstruction in that setting is the volume of $Y_\alpha$ being larger than the volume of $Y_\beta$.
\end{Ex}

\begin{Def}
Let $(Y,\alpha)$, $(Y,\beta)$ be two contact manifolds in the same contactomorphism class and $(SY,d\theta)$ their common symplectization. We define the contact Banach-Mazur distance between $\alpha$ and $\beta$ to be 
$$
d_{CBM}(\alpha,\beta):=\inf\{\ln{C} \in [0,\infty) \mid \alpha \prec C\cdot \beta, \beta \prec C\cdot \alpha \}
$$
\end{Def}

\vspace{10pt}
In view of definition \ref{ChoiceFree}, it is obvious that if $(Y,\alpha)$ is contactomorphic to $(Y,\beta)$ then they have the same symplectization. So, the reference to the symplectization in the definition above is not ambiguous. The fact that this is a pseudodistance is proved in the following section. \\ \par

We will be measuring the volume of the image of a cs-embedding of a contact manifold $Y$ into the relevant symplectization as follows. 

\begin{Def}
 Let $\theta$ be the canonical form of the symplectization and $\alpha_0:=\theta|_{\phi(Y)}$. Then $Vol(\phi(Y)):=\int_{\phi(Y)}\alpha_0 \wedge (d\alpha_0)^{n-1}$
\end{Def}

Dealing with contact forms and not just contact structures provides the advantage of being able to obtain dynamical (and not just topological) information about contact manifolds, thus being able to obtain obstructions (for instance by using the barcodes of corresponding persistence modules of contact homologies) to the existence of symplectic cobordisms between them or symplectic embeddings of their respective fillings, otherwise not detected considering the contact structure itself. Of course, in the overtwisted case, fillings are excluded by a theorem of Gromov-Eliashberg which states that if a contact manifold is fillable, then it is tight.\\ \par

As a last introductory note, the distance defined above is easily seen to be non trivial since contact forms yielding different volume for $Y$ are at positive distance apart. As we will see later on, the definition of this distance is not semi-vacuous by depending only on volume, as it is also possible for contact forms with the same volume to be positive distance apart.

\section{Statement of the Results}

\subsection{Main Results} \label{MainResults} \text{} \\ \par
In this section we provide the main results and we only give some of the most straightforward proofs. The rest of the proofs are given in following sections as we first need to recall some tools and ideas from the literature for each one respectively.

\begin{Th}
\label{pseudodistance}
$d_{CBM}$ is a pseudodistance on the space of contact forms supporting the contact structure $\xi$ on the contact manifold $(Y,\xi)$.
\end{Th}

\begin{proof}
We have to show non-negativity, symmetry and the triangle inequality.
The distance is by definition non-negative. Symmetry is also immediate from the definition. \par
We show in claim \ref{trans} that $\prec$ is transitive. Using this, the triangle inequality can be shown as follows. We have
$$d_{CBM}(\alpha,\beta)=\inf\{\ln(l)\mid \alpha \prec l \beta, \beta \prec l \alpha\}=\ln(L)$$ and
$$d_{CBM}(\beta,\gamma)=\inf\{\ln(m)\mid \gamma \prec m \beta, \beta \prec m \gamma\}=\ln(M)$$ ~\\
By transitivity, if $\alpha \prec l\beta \prec lm \gamma$ and $\gamma \prec m\beta \prec lm \alpha$  we obtain $\alpha \prec lm \gamma$ and $\gamma \prec lm \alpha$

\begin{gather*}
d_{CBM}(\alpha,\gamma)=\inf\{\ln(C)\mid \alpha \prec C \gamma, \gamma \prec C \alpha\} \leq \ln(LM) \\ =\ln(L)+\ln(M)=d_{CBM}(\alpha,\beta)+d_{CBM}(\beta,\gamma)
\end{gather*}

\end{proof}

\begin{Rem}
The distance captures dynamical information and degenerates as one expects, since two strictly contactomorphic manifolds have distance 0. Indeed, it is clear that if two contact manifolds $(Y,\alpha)$ and $(Y,\beta)$ are strictly contactomorphic, i.e. there exists a diffeomorphism $f:Y \rightarrow Y$ such that $f^*(\beta)=\alpha$, the embeddings $\phi:Y \rightarrow SY$ and $\phi \circ f :Y \rightarrow SY$ yield $d_{CBM}(Y_\alpha,Y_\beta)=0$
\end{Rem}

\begin{Claim}\label{trans}
$\prec$ is transitive.
\end{Claim}

For the proof we will need the following lemma. We give the following definition for notational convenience. We know that a diffeomorphism $\phi$ of a manifold $Y$ induces a map $F_\phi$ on $T^*Y$ given by $F_\phi(x,p)=(\phi(x),(\phi^{-1})^*p)$. Then we define $W(\phi_*\beta):=F_\phi(W(\beta))=W((\phi^{-1})^*\beta)$.

\begin{Lemma}
Consider $W(\beta)$ and a cs-embedding $\psi:(Y,\beta)\rightarrow (SY,d(\theta+\eta))$, for an exact one form $\eta$, compactly supported in a neighborhood of $\psi(Y)$, such that the Liouville vector field $L_{\theta+\eta}$ associated to $\theta+\eta$ is transverse to $\psi(Y)$. Then, we have an embedding $F:W(\beta)\rightarrow W((\psi^*)^{-1}\beta) \subseteq (SY,d(\theta+\eta))$ with $F^*(\theta+\eta)=r\beta$.
\end{Lemma}

\begin{proof}
Let $\psi:(Y,\beta) \rightarrow (SY,d(\theta +\eta))$ be the assumed cs-embedding, i.e. $\psi^*(\theta+\eta)=\beta$. As before, we denote the flow for time $t$ of the Liouville field $L_\alpha$ associated to the primitive $\alpha$ by $L_\alpha^t$. Define $F: (SY,r\beta) \rightarrow (SY,d(\theta + \eta))$ by

$$F(L_{r\beta}^{\ln(t)}s_{\beta}(y))=L_{\theta+\eta}^{\ln(t)}(\psi(y)), \qquad \forall y\in Y, t\in \R$$

where $s_\beta$ the standard embedding of $Y$ as the graph of the form $\beta$ in its symplectization and $L_{\alpha}^{\ln(t)}$ the Liouville flow in the symplectization $(SY,d\alpha)$. We show $F^*(\theta+\eta)=r\beta$. The tangent spaces in the source and target symplectizations split as $\langle L_{r\beta}\rangle \oplus\langle R_{r\beta} \rangle \oplus \ \xi$ and $\langle L_{\theta+\eta}\rangle \oplus \langle R_{\theta+\eta}\rangle \oplus \ \xi$. $L_\alpha$ denotes the Liouville vector field of $(SY,d\alpha)$.\\ \par

It will be enough to show that $F_*(L_{r\beta})=(L_{\theta+\eta})$ and $\forall t$ and $\Phi_s:=L_{\theta+\eta}^{\ln(t)}\circ \psi \circ s_\beta^{-1}\circ L_{r\beta}^{-\ln(t)}: L_{r\beta}^{\ln(t)}(s_{\beta}(y)) \rightarrow {L}_{\theta+\eta}^{\ln(t)}(\psi(y))$ we have $\Phi_t^*(\theta+\eta)=r\beta$, i.e. for fixed $t$, the hypersurfaces  $\big\{L_{r\beta}^{\ln(t)}(s_{\beta}(y))\big|y\in Y\big\}$ and $\big\{L_{\theta+\eta}^{\ln(t)}(\psi(y))\big|y\in Y\big\}$ are strictly contactomorphic.\\

For the first statement we have, \newline
$F_*((L_{r\beta})_{L_{r\beta}^{\ln(\tau)}s_\beta(y)})=\frac{d}{dt}\big|_{t=0}F(L_{r\beta}^{\ln(t)}(L_{r\beta}^{\ln(\tau)}s_\beta(y))=\frac{d}{dt}\big|_{t=0}F(L_{r\beta}^{\ln(t)\ln(\tau)}s_\beta(y))=\\ \frac{d}{dt}\big|_{t=0}F(L_{r\beta}^{\ln(t+\tau)}s_\beta(y))=\frac{d}{dt}\big|_{t=0}F(L_{\theta+\eta}^{\ln(t+\tau)}\psi(y))=(L_{\theta+\eta})_{F(L_{r\beta}^{\ln(\tau)}s_\beta(y))}$

~\\
For the second statement we have, \newline
$\Phi_t^*(\theta+\eta)=(L_{r\beta}^{-\ln(t)})^* \circ  (s_\beta^{-1})^* \circ \psi^*(t(\theta+\eta))=(L_{r\beta}^{-\ln(t)})^*(t r\beta)=r\beta$ \\ \par
Restricting $F$ to $Y \times \{t\leq 1\}$ yields the required embedding.
\end{proof}

\par
 
We prove claim \ref{trans}, namely that $\prec$ is transitive.

\begin{proof} 
Let $\alpha \prec \beta$ and $\beta \prec \gamma$. We show $\alpha \prec \gamma$. The assumption means that there exist embeddings $\phi:(Y,\alpha) \rightarrow (SY,d(r\beta)$ and $\psi:(Y,\beta) \rightarrow (SY,d(r\gamma)$ such that $\phi(Y)\subset W(\beta)$, $\psi(Y)\subset W(\gamma)$, $\phi^*(r\beta+\eta_1)=\alpha$ and $\psi^*(r\gamma+\eta_2)=\beta$ for two compactly supported exact one forms $\eta_1,\eta_2$. Let $F$ be the map defined in the previous lemma. Under $F$, $\phi(Y)$ maps into  $W((\psi^*)^{-1}\beta)$ and in particular into $W(\gamma)$. Moreover, setting $\Phi:=F\circ \phi:(Y,\alpha)\rightarrow (SY,d(r\gamma))$, we have a cs-embedding of $(Y,\alpha)$ in $(SY,d(r\gamma))$. Indeed, $\Phi(Y)\subseteq W((\psi^*)^{-1}\beta) \subseteq W(\gamma)$ and $\Phi^*(r\gamma+\eta+\eta_3)=(F\circ \phi)^*(r\gamma+\eta+\eta_3)=\phi^*(r\beta+\eta_1)=\alpha$, for any compactly supported exact one form $\eta_3$ such that $F^*(\eta_3)=\eta_1$. Thus, we have shown that $\alpha \prec \gamma$.
\end{proof}

The main result, already described in the introduction, is the following. Fix a co-oriented overtwisted contact structure $\xi$ on $Y$. $TY/\xi$ is a trivial line bundle so let $X \in (TY/\xi)^\perp$ be a global section. Consider $$\mathcal{C}_{ot}^{Y,\xi}=\{\alpha \in T^*Y| \alpha(X)>0, \ \alpha \wedge (d\alpha)^{n} \neq 0, \ ker(\alpha)=\xi\}$$
Moreover, the norm $||\cdot||_{\infty}$ on $\mathbb{R}^2$ induces a metric on $\R^2$ and in particular the half-space $\mathbb{H}=\{(x,y)\in \R^2 \mid y<0\}$. We denote it by $d_\infty$. 

\begin{Th}\label{quasiisometricembedding}
Let $(Y,\xi=ker(\alpha))$ be an overtwisted closed contact manifold. There exists a bi-Lipschitz embedding $(\mathbb{H},d_{\infty}) \rightarrow (\mathcal{C}_{ot}^{Y,\xi},d_{CBM})$
\end{Th} \vspace{4pt} \par

Note that there is no assumption on the dimension of the contact manifold $Y$. The main tool used in the proof of this theorem in 3 dimensions is the Lutz twist which is recalled before the proof of this theorem in subsection \ref{Lutztwist}. Overtwisted contact structures in 3 dimensions are classified by the homotopy type of the plane field $\xi$ and full Lutz twists do not alter the homotopy type of the plane field we start with. So, we can modify the overtwisted contact form representing the contact structure $\xi$ and consequently the dynamics on $Y$ without changing the contact structure. In dimensions higher than 3, the situation is similar for reasons that will be explained in section \ref{Higherdim}. The classification in higher dimensions is again homotopy theoretic as was recently shown in \cite{MR3455235}, yet there is some subtlety involved when generalizing the Lutz twist. This generalization is proposed in \cite{EP1},\cite{EP2} by Etnyre and Pancholi. There are drawbacks though with the main one being that half Lutz twists do not behave as expected (e.g. they alter the diffeomorphism type of the original manifold). Moreover, the pseudoholomorphic curve analysis is quite involved. Instead of using the generalization of the Lutz twist in higher dimensions, we will mostly follow the negatively stabilized open book decomposition approach as in \cite{MR2646902}. This is because the pseudoholomorphic curve analysis is already carried out. We explain this in subsection \ref{Lutztwist}.

\begin{Rem}
Since any two norms on a finite dimensional vector space are equivalent, we could have chosen to use any norm. For convenience we choose $||\cdot||_\infty$. It turns out that the proof shows something slightly stronger than the statement of the theorem.
\end{Rem} 

The proof of this theorem in 3 dimensions will be provided in section \ref{pf}, as we will need first to recall some basic tools, in order to modify contact forms and prove the existence of a certain unique pseudoholomorphic disk bounded by a Reeb orbit. The extension to higher dimensions will be given in section \ref{Higherdim}.

It is clear that any two contactomorphic manifolds are finite distance apart. Thus, the question that arises is, can we provide upper bounds on $d_{CBM}$?  A more elaborate answer is provided in terms of the bi-Lipschitz embedding theorem \ref{quasiisometricembedding}, yet a straightforward answer is given with respect to the positive function $f$ which relates the two forms.

\begin{Prop}\label{UB}
Let $(Y,\alpha)$, $(Y,\beta)$ two closed contactomorphic contact manifolds (i.e. $\beta=f\alpha$ for some smooth $f:Y\rightarrow \mathbb{R}_+$). Then, $$d_{CBM}(\alpha,\beta) \leq \max\{\ln(\max(f)),-\ln(\min(f))\}$$

\end{Prop}

\begin{proof}
We embed $(Y,\alpha)$ into its symplectization SY in the standard way, i.e. $\psi((Y,\alpha))=\{1\}\times Y$. Using $\psi$, $(Y,\beta)$ embeds as $\psi((Y,\beta))=(f,Y)$. A value for $k$ that will provide $\alpha \prec k\beta$ is $\max\{\frac{1}{min(f)},1\}$.~\\

Now, if we embed $(Y,\beta)$ in the standard way, i.e. $\phi((Y,\beta))=\{1\}\times Y$, then $(Y,\alpha)$ embeds as $\phi((Y,\alpha))=(\frac{1}{f},Y)$. A value for k that will provide $\beta \prec k\alpha$ is $\max\{\frac{1}{\min(\frac{1}{f})},1\}=\max\{\max(f),1\}$.~\\

Note that a value for k that works for both directions is $\max\big\{\max(f),\frac{1}{\min(f)}\big\}$ since this is always $\geq 1$. Taking logarithms as in the definition of $d_{CBM}$ we have $$\max\Big\{\ln(\max(f)),\ln(\frac{1}{\min(f)})\Big\}=\max\big\{\ln(\max(f)),-\ln(\min(f))\big\}$$

\end{proof}

 We remark that cleverer embeddings may provide sharper upper bounds for the distance. We provide such an example using symplectic folding. This is example \ref{folding}.~\\
 
 The proof of the above also works in the case of the so called conformal factor distance (to be defined shortly), yielding upper bounds. The conformal factor distance can be defined both on the space of contact forms supporting the same contact structure $\xi$ on Y and on the contactomorphism group $Cont^+(Y,\xi)$ of contactomorphisms preserving the co-orientation of $\xi$. We begin with the former.~\\
 
 If $(Y,\alpha)$ and $(Y,\beta)$ are contactomorphic, then there exist $\phi:Y\rightarrow Y$ and a smooth function $f:Y \rightarrow \R_+$ such that $\phi^*\alpha=f \beta$. On the other hand, if $\alpha$, $\beta$ support the same contact structure, then $Y_\alpha$ and $Y_\beta$ with the contact structures induced by the Liouville vector field in the symplectization, are of course contactomorphic. This can be seen from \cite{Cie-Eli}, Lemma 11.4, which states 
 
 \begin{Lemma}\label{holonomy}
 Let $\Sigma,\widetilde{\Sigma}$ be hypersurfaces in a Liouville manifold $(V,\omega,X)$ such that following the trajectories of $X$ defines a diffeomorphism $\Gamma:\Sigma\rightarrow\widetilde{\Sigma}$. Then $\Gamma$ is a contactomorphism for the contact structures induced by $i_X\omega$.
 \end{Lemma}
 
 \begin{Def}
  If $\alpha$, $\beta$ support the same contact structure on Y, then define $$d_{CF}(\alpha,\beta):=\displaystyle\inf_{\phi}\{\max_Y|\ln(f)| \Big| \phi^*(\alpha)=f\beta \}$$
 \end{Def}
 
 The conformal factor norm on $Cont^{+}(Y,\xi)$ is defined as follows.
 
 \begin{Def}
  Let $\xi=ker(\alpha_0)$ and $\phi \in Cont^{+}(Y,\xi)$ such that $\phi^*\alpha_0=f\alpha_0$. Then, $|\phi|_{CF}:=\displaystyle\max_Y|\ln(f)|$.
 \end{Def}

 \begin{Def}
  The conformal factor distance is $d_{CF}(\phi,\psi):=|\phi^{-1} \psi|_{CF}$
 \end{Def}

\vspace{10pt}

\subsection{The Rosen-Zhang definition of \texorpdfstring{$\mathbf{d_{CBM}}$}{TEXT}} \label{JD} \text{}\\

In this subsection, we briefly recall definitions and results from \cite{2020arXiv200105094R}. The proofs can be found in the relevant reference. We then compare their results to ours. For clarity, we denote Rosen and Zhang's version of the distance by $\Delta_{CBM}$. It appears that, if we restrict the set of contactomorphisms we are working with to the identity component of the contactomorphism group $Cont_0(Y,\xi)$, this definition is similar to the definition of the conformal factor distance between two contact forms that was defined in the previous section. This is the content of proposition \ref{similarity}. Let $(Y,\xi)$ a co-oriented contact manifold. Their strategy is first to define a distance on the space of forms supporting the same contact structure and then make geometric sense out of it. \\ \par

Making a choice of a contact form $\alpha_0$ supporting $\xi$, they view the space of all such contact forms supporting $\xi$ as an orbit space 
$$O_\xi(\alpha_0):=C^\infty(Y,\mathbb{R})\cdot\{\alpha_0\}$$
where the action of an element $f \in C^\infty(Y,\mathbb{R})$ is given by multiplication of the form $\alpha_0$ by $e^f$, i.e. based on a previous remark regarding the effect of the Liouville flow on the contact form in the symplectization, this multiplication is equivalent to flowing for time $t=\ln(e^f)=f$. Note that since $f$ is not constant, this time is not uniform. This will be helpful in the definition of the distance on hypersurfaces of restricted contact type. \\ \par

For any form $\beta$ supporting an isomorphic structure to $\xi$, there is a contactomorphism $\phi_\beta$ such that  $\beta=\phi_\beta^*\alpha_0=e^{f_\beta}\alpha_0$. A partial order is defined on $O_\xi(\alpha_0)$ as follows.

\begin{Def}
 $\alpha \preceq \beta$ iff $f_\alpha \leq f_\beta$ pointwise. 
\end{Def}

 Somehow notationally absurd, but chosen so as to remember that \say{$\preceq$} comes purely by an inequality between conformal factors \say{$\leq$}, we have the following.

\begin{Prop}\label{comparisonineq}
If $\alpha \preceq \beta$, then $\alpha \prec \beta$
\end{Prop}

\begin{proof}
We can identify the embeddings of $(Y,\alpha)$ and $(Y,\beta)$ into $SY$ with the graphs both $\alpha$ and $\beta$ in $SY$. Since $\alpha \preceq \beta$, i.e. $f_\alpha \leq f_\beta$, we get that the image of $(Y,\alpha)$ under the standard cs-embedding as the graph of the form $\alpha$, which we denote $s_\alpha$, satisfies $s_\alpha(Y)\subset W(\beta)$. Moreover, $s_\alpha^*(\theta)=\alpha$. Thus, both requirements in the definition of a cs-embedding are satisfied for the standard embedding.
\end{proof}

Now we recall definition 1.12 from \cite{2020arXiv200105094R} which is their definition of the Banach-Mazur distance on the space of forms supporting $\xi$. Denote by $Cont_0(Y,\xi)$ the identity component of the contactomorphism group of $(Y,\xi)$.

\begin{Def}
 For any $\alpha,\beta \in O_\xi(\alpha_0)$, we define $$\Delta_{CBM}(\alpha,\beta):=\inf\Big\{\ln C \geq 0 \mid \exists \phi \in Cont_0(Y,\xi) \text{ s.t. } \frac{1}{C}\alpha\preceq \phi^*\beta \preceq C\alpha \Big\}$$
\end{Def}
The condition in the definition is explicitly 
$$f_\alpha-\ln C \leq f_\beta \circ \phi + g_{\phi,\alpha_0} \leq f_\alpha+ \ln C$$
where $\phi^*\beta=e^{g_{\phi,\alpha_0}}\alpha_0$.

All expected properties hold according to the following proposition which is proposition 2.8 in \cite{2020arXiv200105094R}.

\begin{Prop}
For any $\alpha_1,\alpha_2,\alpha_3 \in O_\xi(\alpha_0)$ we have 
\begin{itemize}
    \item $\Delta_{CBM}(\alpha_1,\alpha_2) \geq 0$ and $\Delta_{CBM}(\alpha_1,\alpha_1)=0$
    \item $\Delta_{CBM}(\alpha_1,\alpha_2)=\Delta_{CBM}(\alpha_2,\alpha_1)$
    \item $\Delta_{CBM}(\alpha_1,\alpha_3) \leq \Delta_{CBM}(\alpha_1,\alpha_2) +\Delta_{CBM}(\alpha_2,\alpha_3)$
    \item $\Delta_{CBM}(\phi^*\alpha_1,\psi^*\alpha_2)=\Delta_{CBM}(\alpha_1,\alpha_2)$ for any $\phi,\psi \in Cont_0(Y,\xi)$
\end{itemize}
\end{Prop}
Since all information can be read off of the conformal factors one expects the following.

\begin{Prop}\label{similarity}
$\Delta_{CBM}=d_{CF}$ when restricted to $Cont_0(Y,\xi)$.
\end{Prop}

\begin{proof}
We have $$d_{CF}(\alpha,\beta)=\displaystyle\inf_{\phi \in Cont_0(Y,\xi)}\{\displaystyle\max_Y|\ln(f)|\big| \phi^*(\alpha)=f\beta\}$$
and also
$$\Delta_{CBM}(\alpha,\beta)=\inf\{\ln(C)\big|\exists \phi \in Cont_0(Y,\xi) \text{ s.t. } \frac{1}{C}\beta \preceq \phi^*(\alpha) \preceq C\beta\}$$
We rewrite the second distance in order to make it look more like the first one.

\begin{gather*}
    \Delta_{CBM}(\alpha,\beta)= \Delta_{CBM}(\beta,\alpha)=
    \displaystyle\inf_{\phi \in Cont_0(Y,\xi)}\{\ln(C)\big|\frac{1}{C}\beta \preceq \phi^*(\alpha)=f\beta \preceq C\beta\} \\
   =\displaystyle\inf_{\phi \in Cont_0(Y,\xi)}\{\ln(\displaystyle\max_Y(f))\big|\phi^*(\alpha)=f\beta\}\\
   =\displaystyle\inf_{\phi \in Cont_0(Y,\xi)}\{\displaystyle\max_Y|\ln(f)|\big| \phi^*(\alpha)=f\beta\}\\
   =d_{CF}(\alpha,\beta)
\end{gather*}

\end{proof}

The way to make geometric sense out of this definition of the distance between forms is as follows. Let $(Y,\alpha_0)$ be a closed Liouville fillable contact manifold so that there exists a domain $(W,\omega,L)$ with $\partial W=Y, (\iota_L\omega)\mid_Y=\alpha_0$ and complete flow for $t<0$. The completion of $W$ is denoted by $\widehat{W}$ and is $SY\sqcup Core(W)$ with $SY$ being the symplectization. The coordinates on $SY$ are $(u,x)$. In these coordinates, $Y=\{u=1\}$, $W=\{u\leq 1\}$, $Core(Y)=\{u=0\}$. Pick $\alpha=e^f\alpha_0 \in O_\xi(\alpha_0)$, for some $f:Y \rightarrow \mathbb{R}$. Define the corresponding Liouville domain $$W^\alpha=\{(u,x)\in \widehat{W}\mid u < e^{f(x)}\}$$

\begin{Rem}
Note that $W(\alpha)$ is not the same as $W^\alpha$ as the former refers to a subset of the symplectization and the latter to the Liouville domain bounded by the contact hypersurface equipped with the form $e^f\alpha_0$. If $\alpha_0=\alpha$, so $e^f=1$, we have that $Core(\widehat{W})\sqcup W(\alpha)=W^\alpha$. The difference becomes apparent in the absence of core.
\end{Rem}

The following was not explicitly defined in \cite{2020arXiv200105094R}, yet it was implied by their stability result, Theorem 1.14. 

\begin{Def}
 The contact Banach-Mazur distance between domains $W^{\alpha_i}$ is defined to be $\Delta_{CBM}(W^{\alpha_1},W^{\alpha_2}):=\Delta_{CBM}(\alpha_1,\alpha_2)$.
\end{Def}

Using $d_{CBM}$ as defined definition \ref{dcbmdefn}, we can also provide the following definition between such Liouville manifolds.

\begin{Def}
 The contact Banach-Mazur distance between domains $W^{\alpha_i}$ is defined to be $d_{CBM}(W^{\alpha_1},W^{\alpha_2}):=d_{CBM}(\alpha_1,\alpha_2)$.
\end{Def}

We now recall the definition of the coarse and symplectic Banach-Mazur distances.

\begin{Def}
 For two open star-shaped domains $U,V$ of a Liouville manifold $(W,\omega,L)$ we define their coarse Banach Mazur distance by 
 $$d_c(U,V):=\inf\{\ln(C)>1 | \exists (\phi,\psi) \text{ s.t. } U \xhookrightarrow{\phi}CV \text{ and } V \xhookrightarrow{\psi} CU\}$$
 where $\xhookrightarrow{\phi}$ means that between the two starshaped domains $U,V$ there exists a Hamiltonian isotopy $\{\phi_t\}_{t\in [0,1]}$ defined on $W$ such that $\phi_0=Id$ and $\phi_1(U)\subset V$.
\end{Def}

The stronger version $d_{SBM}$ of this distance is defined by additionally requiring that the composition $\psi \circ \phi$ is isotopic inside $CU$ to the identity map on $W$ through Hamiltonian isotopies. This is what is called the unknottedness condition. \par
Let $P(a,b):=B^2(a)\times B^2(b)\subset \C^2$. As far as we know, the first time a similar question was raised was in \cite{FHW}, where the authors show that if $a\leq b<c$ and $a+b>c$, then the embeddings $\phi_1,\phi_2:\rightarrow P(c,c)^\circ$ given by $\phi_1(w,z)=(w,z)$ and $\phi_2(w,z)=(z,w)$ are not isotopic through compactly supported symplectomorphisms of $P(c,c)^\circ$. Recently, a stronger notion of knottedness appeared. If $A\subset U \subset \C^n$, a symplectic embedding $\phi:A\rightarrow U$ is called knotted if there is no symplectomorphism $\psi:U\rightarrow U$ with $\psi(A)=\phi(A)$. Some examples of knotted embeddings in this stronger sense were produced by Usher and Gutt in \cite{MR3999447}, see theorem 1.10. The authors exhibit examples of toric domains in $\R^4$ which using filtered positive $S^1$-equivariant symplectic homology are shown to be knotted.

The following stability result, theorem 1.14 in \cite{2020arXiv200105094R}, holds.

\begin{Th}\label{djineq}
For any $\alpha_1,\alpha_2 \in O_\xi(\alpha_0)$, we have $$d_c(W^{\alpha_1},W^{\alpha_2})\leq d_{SBM}(W^{\alpha_1},W^{\alpha_2})\leq \Delta_{CBM}(W^{\alpha_1},W^{\alpha_2})$$
\end{Th}

It is obvious from these definitions that the rescaling takes place in $SY$ using the Liouville vector field corresponding to $\alpha_0$. This is the major restriction when working with $\Delta_{CBM}$ as one has to make a choice, that is to select the reference form $\alpha_0$. This is equivalent to picking the Liouville vector field by which we rescale our contact manifolds. This appears to be very restrictive and all relevant information can be read off of the conformal factors of the forms in question, namely the function $f$ such that $\phi^*\beta=e^f\alpha_0$. By allowing cs-embeddings as we do in the definition in this article, the Liouville vector field is allowed to be modified within the compact support of the exact form $\eta$. So, we can realize that the distance defined in here is finer than the one in \cite{2020arXiv200105094R} or in other words $d_{CBM}\leq \Delta_{CBM}$. We provide a proof for this statement below.\\ \par

\begin{Th}\label{reltoJun}
Let $(Y,\xi)$ a closed contact manifold and $\alpha_1,\alpha_2$ two contact forms with $\xi=ker(\alpha_1)=ker(\alpha_2)$. Then we have $d_{CBM}(\alpha_1,\alpha_2) \leq \Delta_{CBM}(\alpha_1,\alpha_2)$.
\end{Th}

For the proof of this theorem we will need the following lemma.

\begin{Lemma}
Let $(Y,\xi)$ a closed contact manifold, $\phi \in Cont_0(Y,\xi)$ and $\alpha_1,\alpha_2$ two contact forms supporting $\xi$. If $\frac{1}{C} \alpha_1 \leq \phi^*(\alpha_2)\leq C \alpha_1$, then $\frac{1}{C}\preceq \phi^*(\alpha_2)\preceq C \alpha_1$.
\end{Lemma}

\begin{proof}
The assumption $\frac{1}{C} \alpha_1 \leq \phi^*(\alpha_2)\leq C \alpha_1$ implies that there are embeddings of $(Y,\frac{1}{C}\alpha_1),(Y,\phi^*(\alpha_2))$ and $(Y,C \alpha_1)$ in $SY$ as the graphs of $\frac{1}{C}\alpha_1, \phi^*(\alpha_2)$ and $C \alpha_1$ in $T^*Y$. Moreover, the setting $\frac{1}{C}\preceq \phi^*(\alpha_2)\preceq C \alpha_1$ requires the existence of embeddings $\Phi:(Y,\frac{1}{C}\alpha_1)\rightarrow SY$ and $\Psi:(Y,\alpha_2)\rightarrow SY$ with $\Phi^*(\theta+\eta_1)=\frac{1}{C}\alpha_1$, $\Psi^*(\theta+\eta_2)=\alpha_2$, for some compactly supported exact one forms $\eta_1,\eta_2$. For any one form $\beta$ on $Y$, we define the map $s_\beta : Y \rightarrow SY$ given by $y \mapsto (y,\beta_y)$. \par 

For $\phi \in Cont_0(Y,\xi)$ we also define $F_\phi \in Symp(SY)$ by $p \mapsto (\phi^{-1})^*p$.  We set $\Psi:=s_{\phi^*\alpha_2}\circ \phi^{-1}:(Y,\alpha_2) \rightarrow SY$ and $\Phi:=F_\phi \circ s_{\frac{1}{C}\alpha_1}: (Y,\frac{1}{C}\alpha_1) \rightarrow SY$ and we check that they have the desired properties. \par

First, $\Psi^*(\theta)=(s_{\phi^*\alpha_2}\circ (\phi^{-1}))^*\theta=(\phi^{-1})^*\circ (s_{\phi^*\alpha_2})^*\theta=(\phi^{-1})^*(\phi^*\alpha_2)=\alpha_2$ and also 
$\Phi^*(\theta)=(F_\phi \circ s_{\frac{1}{C}\alpha_1})^*=s_{\frac{1}{C}\alpha_1}^* \circ F_\phi^*\theta=s_{\frac{1}{C}\alpha_1}^*\theta=\frac{1}{C}\alpha_1$. It is also immediate that $\Psi(Y) \subseteq W(C\alpha_1)$ and $\Phi(Y) \subseteq W(\alpha_2)$. Thus, we showed $\frac{1}{C}\preceq \phi^*(\alpha_2)\preceq C \alpha_1$ as required.
\end{proof}
 
 We now prove theorem \ref{reltoJun}.
 
 \begin{proof}
 From the previous lemma we have that the setting $\frac{1}{C} \alpha_1 \leq \phi^*(\alpha_2)\leq C \alpha_1$ induces two embeddings $\Psi, \Phi$ required in the definition of $d_{CBM}$. This proves the result.
 \end{proof}

The setting we have talked about so far can be a bit more general by using $d_{CBM}/\Delta_{CBM}$ in defining a distance between hypersurfaces of restricted contact type inside Liouville manifolds $\widehat{W}$.

Pick a hypersurface of restricted contact type $Y_0$ in $(W,\theta)$ and define $\alpha_0:=\theta|_{Y_0}$. For every smooth $g:Y_0 \rightarrow \mathbb{R}$ we have a hypersurface of restricted contact type $Y_g:=\{L^{g(y)}(y) \mid y \in Y \}$, where we recall that $L^t$ denotes the Liouville flow for time $t$. We have a contactomorphism $\phi_g:Y_0 \rightarrow Y_g$ defined by the Liouville flow. More precisely, $\phi^*_g(\theta|_{Y_g})=e^g\alpha_0$. Obviously, there is a one to one correspondence

$$
\{g:Y_0 \rightarrow \mathbb{R}\} \Longleftrightarrow \Bigg\{\begin{aligned} Y_g \text{ hypersurface of restricted contact} \\ \text{ type diffeomorphic to } Y_0 \text{ under } L^t\end{aligned}\Bigg\}
$$

The definitions for the distances between hypersurfaces of such type is provided below.

\begin{Def}
 $d_{CBM}(Y_{g_1},Y_{g_2}):=d_{CBM}(e^{g_1}\alpha_0,e^{g_2}\alpha_0)$
\end{Def}

\begin{Def}
 $\Delta_{CBM}(Y_{g_1},Y_{g_2}):=\Delta_{CBM}(e^{g_1}\alpha_0,e^{g_2}\alpha_0)$
\end{Def}

Since we appeal to the definitions between contact forms, all the results we have proved so far hold also in this case. In particular, $d_{CBM}\leq \Delta_{CBM}$. The following example uses symplectic folding to create embeddings compatible with the definition of $d_{CBM}$, yet not allowed  in $\Delta_{CBM}$. It provides a smaller upper bound for $d_{CBM}$ than we have for $\Delta_{CBM}$. It is not clear though if this example shows that the inequality is strict as the best upper bound for $\Delta_{CBM}$ comes from inclusions and may not be optimal. \\ \par

We work in $\mathbb{C}^2$ with coordinates $\rho_j=\frac{1}{2}|z_j|^2$, $\theta_j \in \mathbb{Z}/2\pi \mathbb{Z}$, $j=1,2$. Let $\lambda=\rho_1d\theta_1+\rho_2d\theta_2$ be the canonical primitive for the standard symplectic structure.

\begin{Def}
 We define the ellipsoid $$E(a,b)=\Big\{(\rho_1,\theta_1,\rho_2,\theta_2)\in \C^2 \mid 2\pi\Big(\frac{\rho_1}{a}+\frac{\rho_2}{b}\Big) \leq 1, 0\leq \rho_1 \leq a, 0 \leq \rho_2 \leq b\Big\}$$
\end{Def}

\begin{Ex}\label{folding}
We normalize $S^3$ as the boundary of $E(\pi,\pi)$. Let $\alpha_0=\lambda|_{S^3=\partial E(\pi,\pi)}$. We can view the hypersurfaces of restricted contact type in $\mathbb{C}^2$, $\partial E(1,3)$ and $\partial B(\frac{1}{2})=\partial E(\frac{1}{2},\frac{1}{2})$, as copies of $S^3$ equipped with different contact forms as follows. There is a diffeomorphism $\Phi:S^3 \rightarrow \partial E(a,b)$ defined essentially by the Liouville flow. Its formula is $$\Phi(\rho_1,\theta_1,\rho_2,\theta_2)=\Big(\frac{\rho_1}{2\pi(\frac{\rho_1}{a}+\frac{\rho_2}{b})},\theta_1,\frac{\rho_2}{2\pi(\frac{\rho_1}{a}+\frac{\rho_2}{b})},\theta_2\Big)$$
Pulling back $\lambda|_{\partial E(a,b)}$ under $\Phi$ we get the form $\frac{1}{2\pi(\frac{\rho_1}{a}+\frac{\rho_2}{b})}\alpha_0$ on $S^3$. Thus, if $a=1$ and $b=3$ we get that $(S^3,\frac{1}{2\pi(\rho_1+\frac{\rho_2}{3})}\alpha_0)$ is the corresponding copy to $(\bd E(1,3),\lambda|_{\bd E(1,3)})$. Similarly, for $a=b=\frac{1}{2}$, $(S^3,\frac{1}{2\pi}\alpha_0)$ is the corresponding copy to $(\bd B(\frac{1}{2}),\lambda|_{\bd B(\frac{1}{2})})$. Now using proposition \ref{UB} we get the upper bound $\ln(6)$ for the distance between the hypersurfaces of restricted contact type $\bd E(1,3)$ and $\bd B(\frac{1}{2})$. \par
This essentially comes from the inclusions $$\frac{1}{6}\cdot B\Big(\frac{1}{2}\Big)\subseteq E(1,3) \subseteq 6\cdot B\Big(\frac{1}{2}\Big)$$

As we know, inclusion is not always optimal and the next theorem, which is a particular case of theorem 2 in \cite{Schlenk} helps us provide a smaller upper bound for $d_{CBM}$.

\begin{Th}
Assume $a_2>2a_1$. Then there exists a symplectic embedding of $E(a_1,a_2)$ into $B(a_2-\delta)$, $\forall \delta \in (0,\frac{a_2}{2}-a_1)$.
\end{Th}

The proof amounts to the symplectic folding construction which is a composition of Hamiltonian diffeomorphisms. In order to be precise, $\forall \epsilon>0$ we have symplectic embeddings
$$E(a,b)\hookrightarrow T(a+\epsilon,b+\epsilon)\text{  and  } T(a,a)\hookrightarrow B(a+\epsilon)$$
where $T(a,b)$ is the open trapezoid and what is being folded is actually the trapezoid. We won't need to talk about the proof further in here so we avoid excess definitions. \\ \par

According to the previous theorem we have the embeddings
$$\frac{1}{6-\delta}B\Big(\frac{1}{2}\Big) \subseteq E(1,3)\xhookrightarrow{fold} (6-\delta)B\Big(\frac{1}{2}\Big), \quad \forall \delta \in \Big(0,\frac{1}{2}\Big)$$
Remark \ref{HamRem} ensures that this embedding restricted on $\bd E(1,3)$ provides a cs-embedding as defined earlier and thus compatible with the definition of $d_{CBM}$. This shows that an upper bound for $d_{CBM}$ is $\ln(6-\delta)$, with $\delta \in (0,1/2)$ clearly less than $\ln(6)$ which was obtained using proposition \ref{UB}. Moreover, the folding works in such a way that $B(\frac{1}{2})$ remains fixed under the folding embedding map. An upper bound for $\Delta_{CBM}$ is only $\ln(6)$ coming from the standard inclusions.
\end{Ex}

Theorem \ref{djineq} induces the next natural question. What is the relationship between $d_c/d_{SBM}$ and $d_{CBM}$?

\begin{Prop}\label{ineqdcdcbm}
For two starshaped domains $U,V$ inside a Liouville manifold $(\widehat{W},\omega,L)$ we have 
$d_c(U,V)\leq d_{CBM}(\bd U,\bd V)$.
\end{Prop}

\begin{Rem}\label{isot1}
If the isotopies produced in the proof satisfy the unknottedness condition, then we can also show that $d_{SBM}(U,V) \leq d_{CBM}(\bd U, \bd V)$. This condition will be explained after the proof of proposition \ref{ineqdcdcbm}, in remark \ref{isot1pf}.
\end{Rem}

The extra difficulty in proving this theorem arises due to the extra freedom we have when embedding $\bd U$ and $\bd V$ in the symplectization in order to calculate their distance. As explained before, the embeddings are only required to be transverse to some locally modified Liouville vector field associated to the form $\theta+\eta$. Thus, the proof amounts to be able to extend the embedding of the boundary $\bd U$ (or $\bd V$) to an embedding of the whole domain $U$ (or $V$). The first strategy one might think is to follow the flowlines of the Liouville vector field $L_{\theta+\eta}$, yet the vector field can now vanish disallowing the definition of a full symplectomorphism from $U$ (or $V$) to the domain bounded by $\phi(U)$ (or $\psi(V)$). The way to define such a symplectomorphism is using Liouville homotopies and a helpful proposition, proposition 11.8 from \cite{Cie-Eli}, which we now state.

\begin{Prop}\label{diffeotopy}
Let $(\widehat{W},\omega_s,L_s), s\in [0,1]$, be a homotopy of Liouville manifolds with Liouville forms $\lambda_s$. Then there exists a diffeotopy $h_s:\widehat{W} \rightarrow \widehat{W}$ with $h_0=Id$ such that $h^*_s\lambda_s-\lambda_0=df_s$ for all $s\in [0,1]$. If moreover $\overline{\bigcup_{s \in [0,1]}Core(\widehat{W},\omega_s,L_s)}$ is compact (e.g. for the completion of a homotopy of Liouville domains), then we can achieve $h^*_s\lambda_s-\lambda_0=0$ outside of a compact set.
\end{Prop}
We will be interested in the case of Liouville domains and the completion of a homotopy between them.~\\

The following lemma is essentially exercise 10.2.6 in \cite{McDuff-Salamon} which is a generalization of their proposition 9.3.1 in the non-compact manifolds case. It is used to detect when a compactly supported symplectomorphism is a Hamiltonian one.

\begin{Lemma}\label{McS}
Let $(M,\omega=-d\lambda)$ be a non-compact symplectic manifold. $\phi \in Symp_c(M,\omega)$ belongs to $Ham_c(M,\omega) \iff \phi \in Symp_{c,0}$ and there exists a compactly supported smooth function $F:M \rightarrow \R$ such that $\phi^*\lambda-\lambda=dF$
\end{Lemma}

If one prefers to use the flux homomorphism, the following holds.

\begin{Lemma}\label{McS1}
If $\omega=-d\lambda$ and $\psi_t:M\rightarrow M$ is a compactly supported symplectic isotopy, then $Flux(\{\psi_t\})=[\lambda-\psi_1^*\lambda]$.
\end{Lemma}

The following lemma, lemma 2.10 from \cite{2018arXiv181100734U}, will be useful in order to control the support of Hamiltonian diffeomorphisms.

\begin{Lemma}\label{usher}
Let X be a manifold without boundary equipped with a smooth family of 1-forms $\lambda_t, 0\leq t \leq 1$ such that $d\lambda_t$ is symplectic and independent of $t$.  Assume furthermore that the Liouville vector fields $L_{\lambda_t}$ of $\lambda_t$ are each complete. Let W be a compact codimension-zero submanifold of X with boundary $\bd W$, having the properties that each $L_{\lambda_t}$ is positively transverse to $\bd W$ and that every point of $X$ lies on a flowline of $L_{\lambda_t}$ that intersect W. Then, there exists a smooth family of symplectomorphisms $F_t:X\rightarrow X$ such that $F_0=Id$ and the support of $F_t$ is contained in the interior of $W$ for all $t$.
\end{Lemma}

It is worth noting that the definition of $d_c/d_{SBM}$ dictates that if we want to make any meaningful comparison between these distances and $d_{CBM}$, we have better to restrict to the case where the cs-embeddings of the contact manifolds $\bd U$ and $\bd V$ (equipped with the contact form induced  by the transverse Liouville vector field) in the symplectization are isotopic through contact hypersurfaces to the canonical embeddings as the boundaries of $U$ and $V$ via the restriction of a Hamiltonian isotopy. This is because the maps  $\phi$ and $\psi$ interleaving $U$ and $V$ in the definition of $d_c$ are Hamiltonian isotopic to the identity and the goal is the cs-embeddings to be the restrictions of these Hamiltonian isotopies to the respective boundaries. With this in mind we now prove proposition \ref{ineqdcdcbm}.

\begin{proof}
We denote by $(Y,\alpha)$ and by $(Y,\beta)$ the hypersurfaces of restricted contact type $(\bd U,\theta|_{\bd U})$ and $(\bd V,\theta|_{\bd V})$. Also, let $U=W^\alpha$ and $V=W^\beta$ their fillings.
The goal is to show that the cs-embeddings required to define $d_{CBM}(\alpha,\beta)$ induce a Hamiltonian isotopy defined on $\widehat{W}$ which is used to calculate $d_c(W^\alpha,W^\beta)$. This will yield the desired inequality $d_c(U,V)=d_c(W^\alpha,W^\beta)\leq d_{CBM}(\alpha,\beta)=d_{CBM}(\bd U,\bd V)$ since when calculating $d_c$, we calculate the infimum over a possibly larger set of functions than $d_{CBM}$.\par

Let $C\geq 1$. Let $\phi_1:(Y,\alpha)\rightarrow(SY,d\theta)$ and $\psi_1:(Y,\beta)\rightarrow (SY,d\theta)$ be cs-embeddings, i.e. $\phi_1^*(\theta+\eta_1)=\alpha$ and $\psi_1^*(\theta+\eta_2)=\beta$, for two compactly supported exact one forms $\eta_1,\eta_2$, with $\phi_1(Y)\subset W^{C\beta}$ and $\psi_1(Y)\subset W^{C\alpha}$.\par

Since we make the assumption that $\phi_1$ and $\psi_1$ are isotopic to the standard embeddings $s_\alpha$ and $s_\beta$ through cs-embeddings, there are families $\phi_s:(Y,\alpha)\rightarrow (SY,d\theta)$ and $\phi_s:(Y,\alpha)\rightarrow (SY,d\theta)$, $s\in[0,1]$, of cs-embeddings with $\phi_0=s_\alpha$ and $\psi_0=s_\beta$ respectively. This means that $\phi_s^*(\theta+\eta_{1,s})=\alpha$ and $\psi_s^*(\theta+\eta_{2,s})=\beta$ for two families of compactly supported exact 1-forms $\eta_{1,s}$ and $\eta_{2,s}$ with $\eta_{1,0}=\eta_{2,0}=0$ and $\eta_{1,1}=\eta_1,\eta_{2,1}=\eta_2$. It is enough to extend these families to families of Hamiltonian isotopies of $\widehat{W}$ denoted by $\Phi_s$ and $\Psi_s$ that are compactly supported in neighbourhoods of $s_\alpha(Y)$ and $s_\beta(Y)$. The idea is that we want $\Phi_s|_{s_\alpha(Y)}=\phi_s$ and $\Psi_s|_{s_\beta(Y)}=\psi_s$ This is where we we need to use proposition \ref{diffeotopy}.\\ \par

For $s\in[0,1]$, we have a homotopy of Liouville manifolds $(\widehat{W},d(\theta+\eta_{1,s}),L_s)$ with Liouville forms $\theta+\eta_{1,s}$. Then according to proposition \ref{diffeotopy} there is a diffeotopy $h_s:\widehat{W} \rightarrow\widehat{W}$, $h_0=Id$, $h_s^*(\theta+\eta_{1,s})-\theta=df_s$ where $f_s$ compactly supported in a neighbourhood of $s_\alpha$. Proposition \ref{McS} yields that $h_s$ is indeed a Hamiltonian isotopy. In this case, $\lambda=-\theta$, $\phi=h_1$ and $F=-f_1+g_1$, where $dg_1=\eta_1$. Alternatively, one can use lemma \ref{McS1} and observe that $\theta-h_1^*\theta$ is exact. Moreover, we have $h_1(W^\alpha)\subset W^{C\beta}$. So, this yields a map $U\xhookrightarrow{\Phi_1}CV$ as required in the definition of $d_c$.
Running the same argument for the homotopy of Liouville manifolds $(\widehat{W},d(\theta+\eta_{2,s}),L'_s)$ gives the second map $V\xhookrightarrow{\Psi_1}CU$ required in the definition of $d_c$.\par

Hence, what we achieved is that given any two cs-embeddings $\phi_1,\psi_1$ used to calculate $d_{CBM}$, we get maps $\Phi_1,\Psi_1$ used to calculate $d_c$. This yields the desired inequality $d_c(U,V)\leq d_{CBM}(\bd U,\bd V)$.
\end{proof}

We now explain the condition under which we can show the inequality $d_{SBM}(U,V)\leq d_{CBM}(\bd U,\bd V)$ mentioned in remark \ref{isot1}.

\begin{Rem}\label{isot1pf}
If $\Psi_1 \circ \Phi_1:\frac{1}{C}U\rightarrow CU$ is Hamiltonian isotopic to the identity map of $\widehat{W}$ within $CU$, then we say that the isotopies satisfy the unknottedness condition. In such favorable case, the proof can be extended to show that also $d_{SBM}(U,V)\leq d_{CBM}(\bd U,\bd V)$. Lemma \ref{usher} helps with controlling the support of the relevant Hamiltonian diffeomorphisms. In particular, if $X=\widehat{CU}$ and $\lambda_t$ is the pullback of $\theta$ under $\Psi_t \circ \Phi_t$, then we see that the unknottedness condition is satisfied and we effectively showed that $d_{SBM}(U,V)\leq d_{CBM}(\bd U,\bd V)$.
\end{Rem}

The next question is to what extent the inequality in proposition \ref{ineqdcdcbm} is strict. If one wants to show the reverse inequality, namely $d_c\geq d_{CBM}$ the only problem that they encounter is whether the Hamiltonian isotopies used in the definition of $d_c$ are collapsing parts of the boundaries $\bd U$ and $\bd V$ into the core of $\widehat{W}$ as such types of Hamiltonians do not yield cs-embeddings in the symplectization which are needed to calculate $d_{CBM}$. Let us explain this more thoroughly in the next remark.

\begin{Rem}\label{assum}
Recall that $d_{CBM}$ is defined only using the symplectization part and not part of the core as it is designed to measure the distance even between non-fillable contact manifolds. This means that we can relate $d_{CBM}$ and $d_{SBM}$ only if we assume that the Hamiltonian isotopies in the definition of $d_c$, by which the infimum in the definition is achieved, are compactly supported outside of a neighbourhood of the core we can show that equality holds.
\end{Rem}

We have the following proposition.

\begin{Prop}
Under the assumption discussed in remark \ref{assum}, we have $d_c(U,V)=d_{CBM}(\bd U, \bd V)$.
\end{Prop}

\begin{proof}
By proposition \ref{ineqdcdcbm}, it is enough to show $d_{CBM}\leq d_c$. Hence, if $U,V$ are starshaped domains, it is enough to prove that Hamiltonian isotopies of $\widehat{W}$ achieving $U\xhookrightarrow{\Phi}CV$ and $V\xhookrightarrow{\Psi}CU$ induce cs-embeddings $\phi$ and $\psi$ which can be used to calculate $d_{CBM}(\bd U,\bd V)$. Pick a hypersurface $Y_0$ as before and fix the contact form $\alpha_0=\theta|_{Y_0}$. Then $\bd U=Y_{g_1}$ and $\bd V=Y_{g_2}$, for two smooth functions $g_1,g_2:Y \rightarrow \R$. Then the corresponding forms are $\alpha=e^{g_1}\alpha_0$ and $\beta=e^{g_2}\alpha_0$. Under this formulation the goal is to show $d_{CBM}(Y_{g_1},Y_{g_2})\leq d_c(U,V)=d_c(W^\alpha,W^\beta)$.\par

Having cs-embeddings as allowed in the definition of $d_{CBM}$ means that there exist $\phi:(Y,\alpha)\rightarrow (SY,d\theta)$ and $\psi:(Y,\beta)\rightarrow (SY,d\theta)$ with $\phi^*(\theta+\eta_1)=\alpha$ and $\psi^*(\theta+\eta_2)=\beta$ for some exact compactly supported 1-forms $\eta_1,\eta_2$ and moreover $\phi(Y)\subset W(Ce^{g_2}\alpha_0)$ and $\psi(Y)\subset W(Ce^{g_1}\alpha_0)$. Given the Hamiltonian isotopies $\Phi=\Phi_t$ and $\Psi=\Psi_t$, we will show that $\Phi_1|_{s_\alpha(Y)}$ and $\Psi_1|_{s_\beta(Y)}$ are such embeddings. Hence the infimum is calculated using a possibly larger set of functions and the inequality $d_{CBM}\leq d_c$ is true.\par

The only thing left to check is that indeed $\phi=\Phi_1\circ s_\alpha$ and $\psi=\Psi_1\circ s_\beta$ have the required properties. First we need to check that $(\Phi_1)^*((\theta+\eta_1)|_{\Phi_1(s_\alpha(Y))})=\theta|_{s_\alpha(Y)}$ and $(\Psi_1)^*((\theta+\eta_2)|_{\Psi_1(s_\beta(Y))})=\theta|_{s_\beta(Y)}$ for two exact compactly supported 1-forms $\eta_1,\eta_2$.  Then it is straightforward to observe that $\phi^*\theta=\alpha$ and $\psi^*\theta=\beta$ as $s_\alpha^*\theta=\alpha$ and $s_\beta^*\theta=\beta$. We actually show this by showing something stronger, namely that this holds for the families $\Phi_t,\Psi_t$ and the families $\eta_{1,t},\eta_{2,t}$ of exact compactly supported 1-forms.\par

We know that Hamiltonian flows are exact symplectomorphisms, i.e. in this case $\Phi_t^*\theta-\theta=df_{1,t}$ and $\Psi_t^*\theta-\theta=df_{2,t}$. Furthermore, by assumption, the support of the functions $f_{1,t}$ and $f_{2,t}$ is compact as the Hamiltonian isotopy is assumed to be compactly supported outside of a neighbourhood of the core. Then we have $\Phi_t^*(\theta-\eta_{1,t})=\theta$ and $\Psi_t^*(\theta-\eta_{2,t})=\theta$ for the exact compactly supported 1-forms $\eta_{1,t}=-(\Phi_1^*)^{-1}(df_{1,t})$ and $\eta_{2,t}=-(\Psi_1^*)^{-1}(df_{2,t})$.\par

The last thing to check is that $\phi(Y)\subset W(C\beta)$ and $\psi(Y)\subset W(C\alpha)$. This is straightforward as this holds for both $\Phi_1$ and $\Psi_1$.
\end{proof}

\begin{Rem}
 If the isotopies satisfy the unknottedness condition, namely $\Psi_1 \circ \Phi_1$ in the proof is isotopic to the identity map of $\widehat{W}$ through Hamiltonian isotopies supported inside $CU$, then $d_{CBM}(\bd U,\bd V)=d_{SBM}(U,V)$. 
\end{Rem}

What we showed is that in the restricted case where the Hamiltonians are compactly supported outside of a neighbourhood of the core, $d_{CBM}$ calculates exactly $d_c$ (or $d_{SBM}$ under the unknottedness condition). So, one can view $d_{CBM}$ as the generalization of $d_c/d_{SBM}$ in the non-fillable case.\vspace{15pt}

\section{Review of contact homology} \label{Contact homology}
\text{} \\ \par

The first degree of freedom in theorem \ref{quasiisometricembedding}, namely one of the two directions of $\mathbb{R}^2$ will be the volume of the contact manifold $(Y,\lambda)$. The second degree of freedom will come from the persistence module of filtered contact homology. It is going to be the filtration level for which the unit of the algebra becomes exact. Recall that for overtwisted contact structures the unit is always exact as contact homology vanishes. A reference for this is \cite{2004math.....11014Y}. Equivalently, it is always the case that the bar corresponding to the empty word of orbits is finite, or in other words the empty word is an exact chain. The class of the empty word in various homology theories is known as the contact invariant, yet this terminology is not standard for the contact homology algebra. It will be helpful for the reader to briefly recollect the notions of contact homology here. We mostly follow \cite{Bo} as this is a concise treatment of the subject.\\ \par

Contact homology mimics the idea of Morse homology working with the action functional $\mathcal{A}: C^{\infty}(S^1,Y) \rightarrow \mathbb{R}$ defined as $$\mathcal{A}(\gamma):=\int_\gamma \alpha$$

\begin{Def}
 Let $\alpha$ be a contact form on Y. The unique vector field characterized by $d\alpha(R_\alpha,\cdot)=0$ and $\alpha(R_\alpha)=1$ is called the Reeb vector field of $\alpha$ 
\end{Def}

\begin{Def}
 A map $\gamma : \mathbb{R}/T\mathbb{Z} \rightarrow Y$ is called a closed Reeb orbit of $\alpha$ if $\gamma'(t)=R_\alpha(\gamma(t))$.
\end{Def}

The following lemma, whose proof can be found in \cite{Bo}, shows that the critical points of the functional $\mathcal{A}$ are precisely the closed Reeb orbits of $\alpha$.

\begin{Lemma}
$\gamma \in Crit(\mathcal{A})$ iff $\gamma$ closed Reeb orbit of $\alpha$ with period $\int_\gamma \alpha$
\end{Lemma}

In a complete analogy with Morse theory, critical points have to be non-degenerate. The non vanishing condition for the Hessian in Morse theory translates in contact homology to the fact that the linearized Reeb flow over a periodic orbit does not have an eigenvalue equal to 1.

\begin{Def}
 Let $\gamma$ be a closed Reeb orbit and $\gamma(0)=p$. $d\phi_t\mid_\gamma:=\Phi_\gamma : \xi_p \rightarrow \xi_p$ is called the linearized return map or Poincar\'e return map of $\gamma$.
\end{Def}

Note that the contact condition on $\alpha$ says that $d\alpha$ is non-degenerate, i.e. $(\xi,d\alpha)$ is a symplectic vector bundle. Since the deRham differential commutes with the Lie derivative $\mathcal{L}_{R_\alpha}$ we have that the map $\Phi_\gamma: \xi_p \rightarrow \xi_p$ is symplectic.

\begin{Def}
 The closed Reeb orbit $\gamma$ is non-degenerate if the map $\Phi_\gamma: \xi_p \rightarrow \xi_p$ has no eigenvalue equal to 1.
\end{Def}

Note that $\gamma$ is non-degenerate if and only if $\gamma$ is a non-degenerate critical point of $\mathcal{A}$, modulo reparametrization. Thanks to Bourgeois \cite{Bourgeois}, we have the following lemma.

\begin{Lemma}
Fix a period threshold $T>0$. The space of contact forms $\mathcal{C}_T$ supporting $\xi$ with all orbits of period $\leq T$ being non-degenerate is open and dense in the space of contact forms supporting $\xi$.
\end{Lemma}

Using Baire's theorem we have in particular $ \displaystyle\bigcap_{T \geq 0} \mathcal{C}_T \neq \emptyset $ which yields the following lemma.

\begin{Lemma}\label{degeneracy}
For any contact structure $\xi$ on Y, there exists a contact form $\alpha$ for $\xi$ such that all closed orbits of $R_\alpha$ are non-degenerate.
\end{Lemma}

The next natural question is what is the grading of each orbit $\gamma$. Contact homology has a relative grading by $\mathbb{Z}/2c_1(\xi)\cdot H_2(Y)$ which is absolute on null-homologous orbits. We now introduce the grading which will also help us select the generators of the algebra. The grading is highly dependent on the notion of the Conley-Zehnder index. More details about its definition and properties can be found in \cite{Bo}.  \\ \par

The Conley-Zehnder index depends on the chosen symplectic trivialization $T$ for $\xi$. If $\gamma$ is a closed Reeb orbit in $Y^{2n-1}$ its grading is given by $$|\gamma|:=CZ_T(\gamma)+n-3$$
where $T$ is a trivialization of $\xi$ along $\gamma$ and $CZ_T(\gamma)$ the Conley-Zehnder index of $\gamma$ with respect to $T$. \par

This dependence becomes more explicit when $\gamma$ is null-homologous and we restrict to trivializations which extend over the corresponding spanning surfaces. Let $\Sigma_\gamma$ be a spanning surface for $\gamma$ and $\tau$ a trivialization for $\xi$ over $\Sigma_\gamma$ which agrees with $T$ over the Reeb orbit $\gamma$. It is easy to see that if $A \in H_2(Y)$ then we can connect sum an already chosen spanning surface $\Sigma_\gamma$ for $\gamma$ with A and this will alter the Conley-Zehnder index as follows. If $\tau'$ is a trivialization for $\xi$ over $\Sigma_\gamma \# A$ which agrees with $T$ over $\gamma$ then
$$CZ_{\tau'}(\gamma,\Sigma_\gamma \# A)=CZ_{\tau}(\gamma,\Sigma_\gamma)+2 \langle c_1(\xi),A \rangle$$

Thus it becomes helpful to also introduce a grading on $H_2(M,\mathbb{Z})$. The grading of $A \in H_2(M,\mathbb{Z})$ is $|A|=-2 \langle c_1(\xi),A \rangle $.\\ \par

$CH(Y,\xi)$ is a DG-algebra which is generated by closed, non-degenerate Reeb orbits which are good. The reason for excluding the collection of bad orbits is orientation issues when gluing pseudoholomorphic curves. Let $\gamma=\gamma_1$ be a simple Reeb orbit and for $k \in \mathbb{N}$, let $\gamma_k$ be the k-fold cover of $\gamma$, namely if $f_k:S^1 \rightarrow S^1$ is given by $f_k(e^{2\pi i t})=e^{2\pi i k t}$, then $\gamma_k$ is the map $\gamma_k:= \gamma \circ f_k : S^1 \rightarrow Y$. We partition the set of closed Reeb orbits into two subsets according to their grading behavior when multiply covered. The grading of $\gamma$ can behave in two distinct ways. Either the parity of $|\gamma_k|$ is equal to the parity of $|\gamma|$ for all $k$ or the parity of $|\gamma_{2k}|$ is not equal to the parity of $|\gamma_{2k-1}|$ for all $k$.

\begin{Def}
 Orbits in the first category are called good and orbits in the second category are called bad.
\end{Def}

The reason for this is the behavior of the Conley-Zehnder index under iterations. The index can also be seen as an integer valued winding number which counts the winding of any push-off of $\gamma$ using the chosen trivialization for $\xi$ or equivalently the rotation of the eigenspaces of the linearized flow over $\gamma$.
\\ \par

The grading in monomials is defined as in any DG-algebra by $|\gamma_1\cdot \cdot \cdot \gamma_k|=|\gamma_1|+\cdot \cdot \cdot +|\gamma_k|$.

\par
In the case that the contact form is degenerate, the proof of lemma \ref{degeneracy} suggests that we have to perturb the contact form by first choosing an action threshold and after the perturbation all orbits of action $\leq T$ are non-degenerate. The formula for the grading changes according to 
$$|\gamma_p|=CZ(S_T)+\frac{1}{2}\dim(S_T)+\text{index}_p(f_T)+n-3$$
where $f_T$ the perturbing function to make $\alpha$ less symmetric. Also, $S_T:=\frac{\{q \in M | \phi^T (q) = q\}}{\sim}$, where $\sim$ the equivalence relation under the circle action induced by the Reeb flow $\phi^t$ and $p$ the critical point corresponding to one of the Reeb orbits $\gamma_p$ created after perturbation by $f_T$ corresponding to p. ~\\

The last step in defining the chain complex is to define the differential $$\partial:CC_*(Y,\xi) \rightarrow CC_{*-1}(Y,\xi)$$.

As in the Morse case, we have to count certain \say{trajectories} between our critical points which in this case are Reeb orbits. As it is the case with all field theories, these trajectories will be special surfaces in the symplectization asymptotic to our orbits. In symplectic field theory, these trajectories are pseudoholomorphic curves of genus zero with one positive end and finitely many negative ends. A Stokes' theorem argument shows that $\partial$ reduces the action, so we get a restriction on how many negative ends can exist and how large their periods can be.\\~

The difficulty that arises here is to define a proper count of pseudoholomorphic curves and thus obtain coefficients for the differential. For this to work, we need the associated moduli space to be cut out transversely, which means that the relevant linearized operator is everywhere surjective, i.e. we have the regularity property. We will describe the situation when we work under this rare favorable situation, but we refer to \cite{MR3981989} for the full picture. In this work, the only count we need to make is for the lowest action orbit which needs to be bounding a unique pseudoholomorphic plane. We will explain why we can obtain a meaningful count in this special case, using the main theorem from \cite{MR3981989} when we describe how to obtain the pseudoholomorphic plane. ~\\ 

Contact homology turns out to be invariant under different choices of the supporting contact form $\alpha$ on $Y$ and of the chosen compatible almost complex structure on $SY$. Thus it is both permitted and easier for us to implicitly make a choice of form in order to describe the theory concretely. In words, we have picked a contact form $\alpha$ on $Y$ which induces a splitting $(SY=\mathbb{R}_+\times Y,d\theta=d(r\alpha))$.\par

We consider an almost complex structure on $(SY,d(r\alpha))$ cylindrical in the following sense.
\begin{itemize}
    \item $J$ is $\mathbb{R}$-invariant.
    \item $J(\partial_r)=R_\alpha$.
    \item $J(\xi)=\xi$ is compatible with $\alpha$, i.e. $g(\cdot,\cdot)=d\alpha(\cdot,J\cdot)$ is an inner product.
\end{itemize}

\begin{Def}
 A map $u:(S^2,j) \rightarrow (SY,J)$ is called pseudoholomorphic if $du\circ j=J\circ du$, i.e. $du$ is complex linear.
\end{Def}

Note that up to isomorphism $S^2$ admits a unique complex structure $j$. In other words, any complex surface of genus zero is biholomorphic to the Riemann sphere $\mathbb{C}P^1$. In the following, we consider punctured holomorphic spheres inside the symplectization $SY$. By this we mean that we consider curves $\Sigma=S^2-\{x,y_1,...,y_k\}$ such that if $(r,\theta)$ are polar coordinates centered at each puncture and $u(r,\theta)=(a(r,\theta),f(r,\theta))\in SY=\mathbb{R}\times Y$ we have
$$\lim_{r \rightarrow 0}a(r,\theta)=\begin{cases}
+\infty, \text{          for the puncture $x$  }\\
-\infty, \text{       for the punctures   } y_i
\end{cases}$$

$$\lim_{r \rightarrow 0}f(r,\theta)=\begin{cases}
\gamma(-\frac{T}{2\pi}\theta), \text{          for the puncture $x$  }\\
\gamma_i(\frac{T_i}{2\pi}\theta), \text{       for the punctures   } y_i
\end{cases}$$ \\~\\ \vspace{5pt}

Despite the fact that pseudoholomorphic curves were an effective tool when studying closed symplectic manifolds (Gromov-Witten theory), for quite a while it was not understood how pseudoholomorphic curves behave in non compact symplectic manifolds. Hofer in \cite{Ho} in a successful effort to prove the Weinstein conjecture for $S^3$, showed that bounds on the energy of such object, yield and interesting behavior and force such curves to be asymptotic to Reeb orbits $\gamma$, $\gamma_i$ and this asymptotic behavior is the requirement we ask for the curves we will be counting. Their periods will be denoted $T$ for the orbit corresponding to the positive puncture $x$ and $T_i$ for the rest of the punctures corresponding to $y_i$. \\~

Recall that in Morse homology we have to count trajectories from critical points of index $s$ to critical points of index $s-1$. This is also what we do here. The fact that we work with orbits and not points gives rise to an extra difficulty. Orbits can be multiply covered, so that raises new issues when a complex curve is asymptotic to them. This is of course not an issue in Morse homology. We denote by $m(\gamma_j)$ the multiplicity of $\gamma_j$ over its underlying simple Reeb orbit. The coefficients involved in the differential are related to the moduli spaces of such punctured pseudoholomorphic curves.\\ \par

Let $\widehat{\mathcal{M}}(\gamma;\gamma_1,...,\gamma_k)$ be the set pseudoholomorphic curves satisfying the asymptotic conditions above. We define an equivalence relation $\sim$ on this set as follows.

\begin{Def}
The maps $u : (S^2-\{x, y_1, . . . , y_k\}, j) \rightarrow (SY, J)$ and $u': (S^2 -\{x', y_1', . . . , y_k'\}, j') \rightarrow (SY, J)$ are equivalent if
and only if there exists a biholomorphism $h : (S^2, j) \rightarrow (S^2, j')$ so that $h(x) = x', h(y_i) = y_i'
\text{  for  }
i = 1, . . . , k$, and $u = u' \circ h$.
\end{Def}

We denote  $\mathcal{M}(\gamma;\gamma_1,...,\gamma_k)=\widehat{\mathcal{M}}(\gamma;\gamma_1,...,\gamma_k)/\sim$. The set of equivalence classes $\mathcal{M}(\gamma;\gamma_1,...,\gamma_k)$ has a natural $\mathbb{R}$-action induced by the translation in the $\mathbb{R}$-coordinate. We will define the differential by counting certain generalized flowlines with appropriate associated weights. In order for this count to be finite, the fact that $\mathcal{M}(\gamma;\gamma_1,...,\gamma_k)/\mathbb{R}$ has a nice topological structure, i.e. that it is a compact 0-dimensional manifold is required. In general, this is not the case. A way to overcome this is the main theorem from \cite{MR3981989}. As stated previously though, in this presentation we assume that the favorable assumption of transversality holds. We have the following.

\begin{Lemma}
$\mathcal{M}(\gamma;\gamma_1,...,\gamma_k)/\mathbb{R}$ is a union of compact manifolds with corners along a codimension 1 branching locus. Each such manifold has a rational weight, so that near each branching point, the sum of all entering weights
equals the sum of all exiting weights. Moreover, each manifold with corner in this union has dimension
$$(n-3)(k-1)+CZ(\gamma)+\sum_{i=1}^k CZ(\gamma_i) +2c_1^{rel}(\xi,\Sigma)-1$$
where $c_1^{rel}(\xi,\Sigma)$ is the first Chern class of $\xi$ on $\Sigma$, relative to the fixed trivializations of $\xi$ along the closed
Reeb orbits at the punctures.
\end{Lemma}

The rational weights take into account the group of automorphisms of the holomorphic curves. If $u$ is a rigid element in $\mathcal{M}(\gamma;\gamma_1,...,\gamma_k)/\mathbb{R}$ of dimension 0, then the weight of $u$ is $\frac{1}{k}$, where $k$ is the order of the
automorphism group of $u$.\\ \par

In order to be able to show that the operator $\partial$ is indeed a differential, we need to understand the boundary of $\mathcal{M}(\gamma;\gamma_1,...,\gamma_k)$, i.e. possible degenerations of such pseudoholomorphic curves.

\begin{Def}
 A broken pseudoholomorphic curve is a set $C_1,...,C_N$ of finite collections of punctured pseudoholomorphic curves $C_i=\{u_{1i},...,u_{li}\}$ such that the negative punctures/orbits of $C_i$ coincide with the positive punctures/orbits of $C_{i-1}$. Moreover, the only positive puncture of $C_1$ corresponds to $\gamma$ and the negative orbits of $C_N$ correspond to $\gamma_1,...,\gamma_k$.
\end{Def}

The boundary of $\mathcal{M}(\gamma;\gamma_1,...,\gamma_k)$ is made up of broken pseudoholomorphic curves. Energy bounds on punctured pseudoholomorphic curves help to control the degeneration behavior and show that $\mathcal{M}(\gamma;\gamma_1,...,\gamma_k)$ together with its boundary broken pseudoholomorphic curves is a compact manifold. See \cite{2003math......8183B} section 10.~\\

A class $A\in H_2(Y,\mathbb{Z})$ is associated to each curve in $\mathcal{M}(\gamma;\gamma_1,...,\gamma_k)$. This has the effect that we can decompose $\mathcal{M}(\gamma;\gamma_1,...,\gamma_k)$ into the connected components corresponding to $A$. Then, we denote by $\mathcal{M}^A(\gamma;\gamma_1,...,\gamma_k)$ the corresponding connected component of the moduli space. This association is essential in order to define the differential of the chain complex, yet the only case we explain here is the case of null homologous orbits. This is important for this work as the most essential orbit, i.e. the one providing the $l$-invariant, is contractible. We do not need to explain all cases here as there is no direct reference to it in this text. A thorough explanation of this can be found in \cite{Bo}, lecture 2. \par

Assuming that the orbit $\gamma$ is null-homologous, we pick a spanning surface $\Sigma_\gamma$ and we use it to trivialize $\xi$ over $\gamma$. Now, to a pseudoholomorphic curve in $\mathcal{M}(\gamma;\gamma_1,...,\gamma_k)$, we glue the surfaces we obtained for each of the orbits in $\{\gamma,\gamma_1,...,\gamma_k\}$ and we obtain a closed surface. We let $A\in H_2(Y,\mathbb{Z})$ be its homology class.\\ 

Using the trivializations discussed before, the formula for the dimension of the corresponding connected component of $\mathcal{M}(\gamma;\gamma_1,...,\gamma_k)$, which we denote $\mathcal{M}^A(\gamma;\gamma_1,...,\gamma_k)$, is
$$\dim\mathcal{M}^A(\gamma;\gamma_1,...,\gamma_k)=|\gamma|-\sum_{i=1}^k|\gamma_i|+2 \langle c_1(\xi),A \rangle$$

We define the coefficients of the differential. Let $\Gamma:=\gamma_1\gamma_2...\gamma_k$. First, define the numbers
$$n_{\gamma,\Gamma}^A=\begin{cases}
& 0,\hspace{142pt} \text{  if           } \dim\mathcal{M}^A(\gamma;\gamma_1,...,\gamma_k) \neq 1    \\
& \displaystyle\sum_{c \in \mathcal{M}^A(\gamma;\gamma_1,...,\gamma_k)/\mathbb{R}}\frac{1}{|Aut(c)|}, \hspace{25pt} \text{  if           } \dim\mathcal{M}^A(\gamma;\gamma_1,...,\gamma_k) = 1
\end{cases}$$
where $|Aut(c)|$ is the order of the automorphism group of the rigid element $c$ of the moduli space. These numbers count rigid pseudoholomorphic curves positively asymptotic to $\gamma$ and negatively asymptotic to $\Gamma$ in the homology class $A \in H_2(M,\mathbb{Z})$. These numbers are finite and nonzero for finitely many classes $A$ due to the fact that the moduli spaces involved are compact. ~\\ 

The coefficients of the differential are now defined to be 

$$n_{\gamma,\Gamma}=\displaystyle\sum_{A \in H_2(M,\mathbb{Z})} n_{\gamma,\Gamma}^Ae^{\pi(A)}\in \mathbb{Q}[H_2(M,\mathbb{Z})/\mathcal{R}]$$
where $\mathcal{R}$ a submodule of $H_2(M,\mathbb{Z})$ with zero grading and $\pi: H_2(M,\mathbb{Z}) \rightarrow H_2(M,\mathbb{Z})/\mathcal{R}$ the natural projection. The usual choice is $\mathcal{R}=H_2(M;\mathbb{Z})$ which gives coefficients in $\mathbb{Q}$.

~\\
The differential is defined by $$\partial \gamma:=m(\gamma)\sum_{\Gamma=(\underbrace{\gamma_1,...,\gamma_1}_{i_1},...,\underbrace{\gamma_k,...,\gamma_k}_{i_k})} n_{\gamma,\Gamma}\underbrace{\gamma_1\cdot \cdot \gamma_1}_{i_1} \cdot \cdot \cdot \underbrace{\gamma_k\cdot \cdot \gamma_k}_{i_k}$$
where we recall that $m(\gamma)$ is the multiplicity of the orbit at $+\infty$, $\gamma$. ~\\

We extend the differential to monomials using the graded Leibniz's rule and to any element of the DG-algebra by linearity. The unit of this DG-algebra is $\Gamma$ for which $k=0$. Contact homology is the homology of this complex.\\ \par 

Contact homology is a functor from the category with objects contact manifolds and morphisms deformation classes exact symplectic cobordisms to the category with objects supercommutative $\mathbb{Z}/2$-graded unital $\mathbb{Q}$-algebras and morphisms graded unital $\mathbb{Q}$-algebra homomorphisms. In the proof of lemma \ref{linvineq}, we will be interested in a refinement which keeps track of the action of orbits. Since the differential decreases action, we are allowed to consider the subcomplex $CC_*(Y,\alpha)^{\leq t} \subseteq CC_*(Y,\alpha)$ consisting of all orbits of action $\leq t$. Notice that this refinement depends on the chosen contact form. Whenever $\frac{\alpha}{t} \geq \frac{\alpha'}{t'}$ pointwise, we get an induced functorial map $CH_*(\alpha)^{\leq t} \rightarrow CH_*(\alpha')^{\leq t'}$, see for example \cite{MR3981989} section 1.7. We will be mostly interested in the case $t=t'$.

\section{Proof of the bi-Lipschitz embedding} \label{pf}

The way to prove theorem \ref{quasiisometricembedding} in 3 dimensions will be to construct a 2-parameter family of overtwisted contact forms on $Y$ modifying the Lutz twist construction as found in \cite{Wendl}. The first parameter will be related to the volume of $Y$ and the second one to the $l$-invariant whose definition is immediately provided as it will be of often use in the rest of this section. Let $\alpha$ be a contact form supporting an overtwisted contact structure on $Y$. Since $CH(Y,\xi)$ vanishes, there is some filtration level for which a primitive $x$ for the unit element appears. This filtration is basically the action whose definition we now recall. Note that by definition, talking about the action always assumes a choice of a contact form.\\ \par 

An element of $CH(Y,\xi)$ has the form $y=\gamma_1^{a_1}\cdot \cdot \cdot \gamma_k^{a_k}+\cdot \cdot \cdot+\gamma_m^{a_m}\cdot \cdot \cdot \gamma_n^{a_n}$, i.e. it is a polynomial in good orbits. In the beginning of section \ref{Contact homology}, we only defined the action of an orbit and not the action of an element of $CH(Y,\xi)$. This definition is provided below.

\begin{Def}
The action of an element $y \in CH(Y,\xi)$ is defined as follows. If $y$ is a monomial in good orbits, i.e $y=\gamma_1^{a_1}\cdot \cdot \cdot \gamma_k^{a_k}$, then $$\mathcal{A}(y):=\displaystyle\sum_{i=1}^{k} a_i\mathcal{A}(\gamma_i)$$
If $y=\Gamma_1+\cdot \cdot \cdot +\Gamma_n$, where $\Gamma_i=\gamma_{i_1}^{i_{a_1}}\cdot \cdot \cdot \gamma_{i_m}^{i_{a_m}}$, i.e $y$ is a polynomial in good orbits, then $$\mathcal{A}(y):=\max_{j}\big\{\mathcal{A}(\Gamma_j)\big\}$$

\end{Def}

We are now ready to provide the definition of the $l$-invariant of a contact form $\alpha$.

\begin{Def}\label{linvt}
We call the lowest action $\mathcal{A}(x)$ of such a primitive $x$ the $l$-invariant.  
\end{Def}

\begin{Rem}
 The $l$-invariant is an invariant of a contact form $\alpha$ supporting a contact structure $\xi$. In the cases that we need to emphasize this, we will use the notation $l(\alpha)$.
\end{Rem}

The section is organized as follows. In subsection \ref{Lutztwist} we recall the notion of a Lutz twist in 3 dimensions. In \ref{Wendlwork} we recall some of Wendl's work and adapt results to our case. In \ref{2family} we construct the 2-parameter family needed in the proof and finally in \ref{pfofqiemb} we prove the main theorem, i.e. theorem \ref{quasiisometricembedding}, in the 3-dimensional case. The higher dimensional case is discussed in section \ref{Higherdim}.

\begin{Rem}
 We provide a brief list of the parameters involved in the constructions of this section and brief explanations. This will not make sense unless the reader arrives to the part of this work where they are needed. We only mention them here in order to state explicitly that the constructions below will only make sense for sufficiently small choice of these parameters.
 \begin{itemize}
     \item $\delta$ : Used in the perturbation of the horizontal Morse-Bott torus. Defined in page \hyperlink{delta}{29}.
     \item $\delta'$ : Used in the perturbation in order to set up the contact homology chain complex up to a certain action threshold $\mathcal{A}_0$. Defined in page \hyperlink{delta'}{30}.
     \item $A$ : Lowest action of Reeb orbit before the Lutz twist. Discussed in page \hyperlink{B}{35}.
     \item $B$ : Action of the family of orbits in the neighbourhood of a transverse knot $K_1$. Discussed in page \hyperlink{B}{35}.
      \item $\epsilon$ : It is $\min\{\ln(A),\ln(B)\}$ and will define the parameter domain of the 2-parameter family of 1-forms to be defined.
     \item $\varepsilon$ : Radius of specific Lutz tube we use. See section \ref{Lutztwist}.
     \item $\varepsilon_0$ : First time when behavior of $h_1(r),h_2(r)$ changes. Defined in subsection \ref{rineq}.
     \item $\delta_0$ : Defines smoothing interval. Defined in subsection \ref{rineq}.
     \item $\delta_1$ : Chosen so as the path of the specific Lutz tube to be continuous (related to $h_1(r)$). See subsection \ref{rineq}
     \item $\delta_2$ : Chosen so as the path of the specific Lutz tube to be continuous (related to $h_2(r)$). See subsection \ref{rineq}.
 \end{itemize}
\end{Rem}

\subsection{The Lutz twist} \label{Lutztwist}\text{} \\ \par

The first step towards the construction of the 2-parameter family is to perform a full Lutz twist. Although by now these constructions are standard and \cite{geiges_2008},\cite{Wendl} are excellent references, we include it here for completeness.\\ \par

Let $Y$ be a co-oriented contact 3-manifold and $P\subset Y$ an embedded $S^1$ positively transverse to $\xi$. Let $S^1 \times D^2$ be a tubular neighbourhood of $P$ in $Y$ (i.e. $P=S^1 \times \{0\}$) such that $\xi=ker(d\theta+r^2d\phi)$ where $\partial_\theta$ agrees with the positive orientation on $P$. This is possible to consider due to the fact that in a neighborhood of a transverse knot, the local model for $\xi$ looks like the one described. See for instance \cite{geiges_2008}, example 2.5.16. The radius of the disk factor is assumed to be $\varepsilon$ which will be sufficiently small. Performing a Lutz twist along $P$ means that we replace the contact structure $\xi=ker(d\theta+r^2d\phi)$ inside $S^1 \times D^2$ by the structure $\xi'=ker(h_1(r)d\theta+h_2(r)d\phi)$. The functions $h_1,h_2$ are only required to satisfy the following 3 properties  \vspace{5pt}
\label{properties}
\begin{itemize}\label{conditions}
    \item $h_1(r)=\pm 1$ and $h_2(r)=\pm r^2$ near $r=0$. 
    \item $h_1(r)=1$ and $h_2(r)=r^2$ near $r=\varepsilon$, where the radius of $D^2$ is $\varepsilon$.
    \item $(h_1(r),h_2(r))$ is never parallel to $(h_1'(r),h_2'(r))$ for $r \neq 0$.
\end{itemize} \vspace{5pt} \par

The \say{$+$} version is called a full Lutz twist, whereas the \say{$-$} version is called a half Lutz twist. The result of a Lutz twist is always an overtwisted contact structure as it is straightforward to see that we create an overtwisted disk inside the tubular neighbourhood of $P$ where for the smaller of the two radii, $r_0$ such that $h_2(r_0)=0$ the Legendrian/meridian $\{(0,r_0,\phi) \mid \phi \in [0,2\pi)\}$ is the boundary of an overtwisted disk.\\ \par

\begin{figure}[h!]
    \centering
    \includegraphics[scale=0.8]{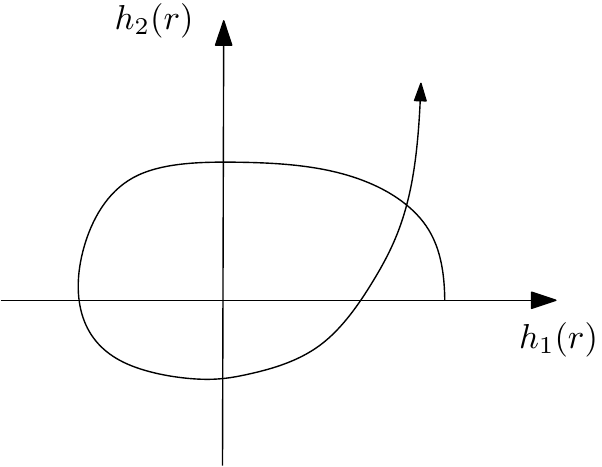}
    \caption{Description of a general full Lutz twist}
    \label{fig:Lutz0}
\end{figure}

\begin{Rem}
 A Lutz twist is primarily a contact topological alteration of the structure. In this work we would like to view it as a dynamical one. To this end, although it is not always true that in a local neighborhood of a transverse knot the contact form is precisely the standard $d\theta+r^2d\phi$, since it is true that two contact forms with the same kernel differ by multiplication by a smooth positive function, we can can make the form locally to look like the standard one by multiplying by an appropriate positive smooth function. In this work, using the above justification, anytime we are interested in working with a transverse knot, we assume that in a neighborhood of it the form looks like the standard one.
\end{Rem}

The advantage of a full twist is that it does not alter the homotopy class of $\xi$ as a 2-plane field thus giving us, according to Eliashberg \cite{Eli}, the unique up to isotopy overtwisted contact representative of the original plane distribution. There is a generalization of the full Lutz twist in higher dimensions according to \cite{EP1},\cite{EP2} and as it is shown there, the contact structure obtained after the generalized Lutz twist is homotopic to the original one through almost contact structures, i.e. the two contact structures are formally homotopic. This, combined with corollary 1.3 from \cite{MR3455235} implies that after the twist, we still get up to isotopy, the unique overtwisted representative of the homotopy class of $\xi$ that we started with. A half Lutz twist generalization is also discussed, yet it seems that this is not a natural way to define a half Lutz twist in higher dimensions as it not only changes the contact structure, but the original smooth manifold itself. Although using this approach of thinking about Lutz twists in higher dimensions seems very natural, negatively stabilized open books help us more. There is a very helpful pseudoholomorphic curves analysis in \cite{MR2646902} which we use to generalize the result to higher dimensions.

\subsection{Recollection of Wendl's work and adaptation to our case} \label{Wendlwork} \text{}\\ \par

Our goal is to control the volume and the $l$-invariant of contact forms supporting the overtwisted contact structure $\xi$ on $Y$. These are the two degrees of freedom of $\mathbb{R}^2$. Although the first is easy to do just by multiplying the contact form, for the second one we have to work more. We have to control the $l$-invariant as defined in definition \ref{linvt}. In order to do this, we have to adapt the Lutz twist construction in a way that allows us to control the action of the lowest action orbit which bounds a unique pseudoholomorphic plane. ~\\

In this we mainly follow \cite{Wendl}, subsection 4.2. Our first goal will be to understand the Reeb dynamics within the Lutz tube and then the foliation by pseudoholomorphic curves of the part of the symplectization of our overtwisted contact manifold $Y$ that corresponds to the Lutz tube $S^1\times D^2$. In what follows, we work in $S^1 \times int(D^2)$ assuming that the radius of $D^2$ is equal to $\varepsilon$. \\ \par

Let's let $X=S^1 \times int(D^2)$ equipped with a contact form described as above, $h_1(r)d\theta+h_2(r)d\phi$. We pick a suitable basis $\{v_1,v_2\}$ for $\xi$ on the complement of the core circle $P=S^1$, where $v_1=\partial_r$, $v_2=\frac{1}{D(r)}(-h_2(r)\partial_\theta+h_1(r)\partial_\phi)$ and $D(r)=h_1(r)h_2'(r)-h_1'(r)h_2(r)$. \\ \par

We define a complex structure $J:\xi \rightarrow \xi$, given by $v_1 \mapsto \beta(r)v_2$ and $v_2 \mapsto -\frac{1}{\beta(r)}v_1$, where $\beta$ a smooth function chosen so as to ensure smoothness of $J$ near the core $P$. So we can assume that is different than 1 only sufficiently close to $P$.\par

Naturally, we let $\Tilde{J}$ be the unique $\mathbb{R}$-invariant compatible almost complex structure on $\mathbb{R}\times X$ determined by $J$ and $\alpha=h_1(r)d\theta+h_2(r)d\phi$. \\ \par

The following proposition describes the local Reeb dynamics within $S^1\times int(D^2)$.

\begin{Prop}\label{dynamics}
Let $r_0>0$ and $\frac{h_1'(r_0)}{2\pi h_2'(r_0)}=\frac{p}{q} \in \mathbb{Q}\cup \{\infty\}$, where $p,q\in \mathbb{Z}$ relatively prime, $ \sign(p)=\sign(h_1'(r_0))$ and $ \sign(q)=\sign(h_2'(r_0))$. Then the torus $$L_{r_0}:=\{r=r_0\}$$ is foliated by orbits of the form 

$$x(t)=\Big(\theta_0+\frac{h_2'(r_0)}{D(r_0)}t,r_0,\phi_0\frac{h_1'(r_0)}{D(r_0)}t\Big)=\Big(\theta_0 +\frac{q}{T}t,r_0,\phi_0-\frac{2\pi p}{T}t\Big)$$ 

all having minimal period 

$$T=q\frac{D(r_0)}{h_2'(r_0)}=2\pi p \frac{D(r_0)}{h_1'(r_0)}$$

If $h_1'(r_0)=p=0$ or $h_2'(r_0)=q=0$ pick whichever formula makes sense. The torus $L_{r_0}$ is Morse-Bott iff $\frac{h_1'(r)}{h_2'(r)}$ (or its reciprocal if needed) has non vanishing derivative at ${r=r_0}$. Moreover, $P:=S^1\times \{0\}$ is a closed orbit f minimal period $T=|h_1(0)|$. For $k\in\mathbb{Z}$ its k-fold cover $P^k$ is degenerate iff $\frac{k h_1''(0)}{2\pi h_2''(0)}\in \mathbb{Z}$ and otherwise it has $$CZ_{\Phi_0}(P^k)=2\left \lfloor{-\frac{k h_1''(0)}{2\pi h_2''(0)}} \right \rfloor +1$$
where $\Phi_0$ the natural symplectic trivialization of $\xi$ along $P$ provided by the coordinates.
\end{Prop} ~\\

Observe that a formula for the action of the orbits in terms of the coordinate $r$ is given by either $\mathcal{A}(r)=2\pi p \frac{D(r)}{h_1'(r)}$ or $\mathcal{A}(r)=q\frac{D(r)}{h_2'(r)}$ depending on where the derivative of either $h_1(r)$ or $h_2(r)$ vanishes. A quick analysis checking the signs of the derivative of $\mathcal{A}(r)$ with respect to $r$ and the possible values for $p,q$ yields two minima for the action precisely when $r$ is such that $h_1(r)=0$. Figure \ref{fig:signs} may help the reader with this argument. According to figure \ref{fig:Lutz0} there are two such values for $r$, which we denote by $r_+$ and $r_+'$. The first of the two will be used to control the $l$-invariant while for the second one, when we pick the specific $h_2(r)$, we will require that $h_2(r_+)<h_2(r_+')$ so as for the orbits corresponding to $r_+$ to have the least action among all orbits. This is explained in claim \ref{claction}. ~\\

\begin{figure}[h!]
    \centering
    \includegraphics[scale=1]{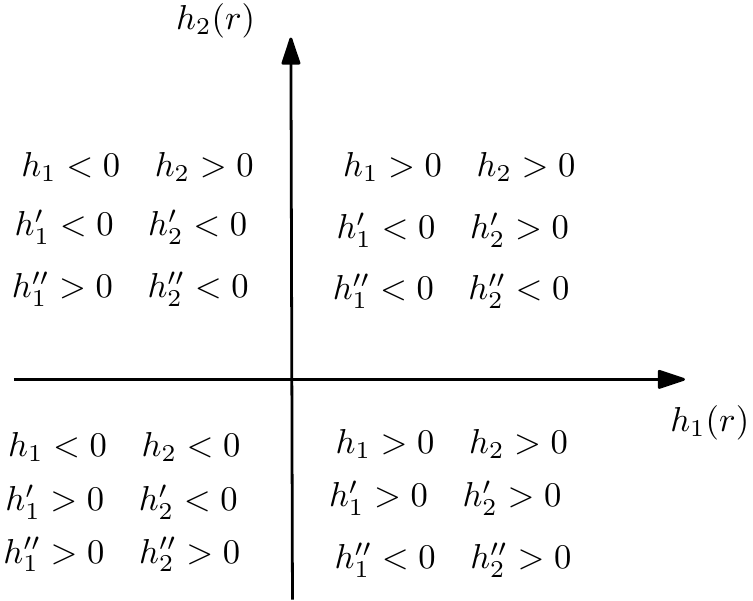}
    \caption{Analysis for the signs of $\mathcal{A}'(r)$}
    \label{fig:signs}
\end{figure}

Let us know proceed with understanding the foliation. Let $r_-,r_+ \in (0,\varepsilon)$. In his thesis, \cite{Wend} section 3.1, Wendl proves (by solving a system of ODEs using a reasonable geometric ansatz) that the region in the symplectization 
$$\{(a,\theta,r,\phi)\in \mathbb{R}\times X \mid r\in(r_-,r_+)\subset(0,\varepsilon)\}$$ is foliated by finite energy $\Tilde{J}$-holomorphic cylinders asymptotic to $\{r=r_-\}$ and $\{r=r_+\}$. \\ 

Next, also within section 3.1, he studies the case when $r_- \rightarrow 0$. There are two cases to consider, namely $q=0$ and $q\neq 0$. In this work we are interested only in the setting where our function $h_2(r)$ so that $h_2'(r_+)=0$, i.e. $q=0$. This ensures that the positive ends of the pseudoholomorphic planes foliating the solid tube do not wind around the $\theta$ direction.\\

Then, using Gromov's removable singularity theorem, he shows that the previous foliation can be extended smoothly to a finite energy foliation of the region $\{r<r_+\}$ by $\Tilde{J}$-holomorphic disks each positively asymptotic to some simply covered Reeb orbit on the torus $\{r=r_+\}$ and transverse to the core P. As explained previously, these simply covered Reeb orbits have period $T=2\pi\cdot h_2(r_+)$. These Reeb orbits, which form a 2-torus, can be arranged, by suitable choices to be explained, to be the lowest action orbits bounding a holomorphic plane. As stated before, the proof of claim \ref{claction} provides more insight. Hence, we end up having at this point that $2\pi\cdot h_2(r_+)$ is the $l$-invariant. Strictly speaking, we need a perturbation so as to have degenerate orbits, yet the action is not far from $2\pi\cdot h_2(r_+)$ as will be explained below.\\ \par

The upshot is that modification of $h_2(r)$ provides the second degree of freedom in the construction of the 2-parameter family. This is because, for any choice of $h_2(r)$ providing a Lutz twist, Wendl's construction is feasible, thus we get holomorphic planes bounded by orbits of certain chosen action. The key idea is that this modification is Lipschitz. This will be proved in subsection \ref{pfofqiemb} where we will get bounds using Gray's stability theorem. Although not relevant to our work, for the reader interested in contact structures, we also remark that the pseudoholomorphic  disk is \say{smaller} than the overtwisted disk as the overtwisted disk has its Legendrian boundary at the specific $r$ between $r_+$ and $r_+'$ for which $h_2(r)=0$.  \\ \par

What was not mentioned explicitly above, is that actually the asymptotes of the pseudoholomorphic curves forming the stable finite energy foliation are foliating in turn the tori $\{r=r_+\}$ and $\{r=r_-\}$ (recall that in our case $r_-=0$ and $h_2'(r_+)=0$ so there is only one interesting torus corresponding to $r_+$). This turns out to be helpful in the following manner. If we need to set up the contact homology chain complex we have to perturb our form so that all Reeb orbits are isolated. This can be done as in \cite{Wend} following \cite{Bourgeois}. We have to be extra careful about the orbits bounding pseudoholomorphic planes. Our perturbation will be performed in two steps.\\ \par

The first which agrees with \cite{Wend} section 3.3 guarantees that the resulting holomorphic planes generically are positively asymptotic to the elliptic and hyperbolic orbits created after perturbation of the form in a small neighbourhood of the torus $\{r=r_+\}$. Recall that our goal is of course to exhibit existence and uniqueness of the pseudoholomorphic plane bounded by an orbit of least action $l$ and that means that the contact homology class of the identity vanishes precisely at action level $l$.\\ \par

The perturbation is performed as follows. We first choose a smooth cut off function $b(r)$ supported in a neighbourhood of the torus $L_{+}=\{r=r_+\}$ and is equal to 1 near $L_+$. Moreover, choose a small number \hypertarget{delta}{$\delta$} and a Morse function $\mu:S^1 \rightarrow \mathbb{R}$ with two critical points at $\theta=\theta_+$ and $\theta=\theta_-$. The perturbed contact form is then 

\begin{gather*}
    \alpha^\delta:= \begin{cases}
(1+\delta b(r)\mu(\theta))\alpha,& r\in supp(b)\\
\alpha,& otherwise
\end{cases}
\end{gather*}

Then one can see that the corresponding Reeb vector field in the solid tube in coordinates $(\theta,r,\phi)$ is given by $$R_\delta=\frac{(h_2'+\delta \mu (b'h_2+bh_2'),-\delta \mu' b h_2,h_1'+\delta \mu (b'h_1+bh_1'))}{D(1+\delta b\mu)^2}$$ has two orbits at $r=r_+$ corresponding to the critical points of $\mu$ (since $h_2'(r_+)=0$, $b'(r_+)=0$ and $b(r_+)=1$). One of them is the elliptic $O_e$ and the other one is the hyperbolic $O_h$.\\ \par 

What has been achieved so far is that the orbits on the positive torus $L_{r_+}$ are non-degenerate, thus we have two non-degenerate, isolated ones $O_h,O_e$. Moreover, those are simply covered so they are not bad. Since $\delta$ can be chosen arbitrarily small, their action is arbitrarily close to the action of the orbits of the original Morse-Bott torus. The foliation consists of a family of pseudoholomorphic planes. One rigid plane positively asymptotic to $O_h$ and a family of planes parametrized by the open interval $(0,1)$ positively asymptotic to $O_e$. For more details the interested reader can consult \cite{Wend}, section 3.3. \\ \par

The next step in the perturbation process is the one described by Bourgeois in \cite{Bourgeois}. This allows all the remaining orbits (not lying on $L_+$ which is already perturbed) to become non-degenerate, thus isolated so as to be able to set up the CH complex. The initial worry is whether this perturbation will affect the foliation and possibly make the rigid plane we are interested in disappear. Thanks to \cite{Wend} section 4.5, the necessary Fredholm analysis shows that for sufficiently small deformation parameter $\delta'$, the foliation is stable under deformations of the form $\alpha^\delta$ and of the almost complex structure $J$. ~\\

 The perturbed contact form is $\alpha^{\delta,\delta'}:=(1+\delta' f_T)\alpha^\delta$, where $f_T$ a smooth Morse function with support a small neighbourhood of the set consisting of points on non-isolated orbits of action $\leq T$. The new Reeb vector field is given by $$R_{\alpha^{\delta,\delta'}}=R_{\alpha^\delta}+X$$
 where $X$ is the vector field with the properties
 $$i(X)d\alpha^\delta=\delta' \frac{df_T}{(1+\delta' f_T)^2}$$
 $$\alpha^\delta (X)=-\delta' \frac{f_T}{1+\delta' f_T}$$ \par

Due to the fact that $\delta'$ is chosen sufficiently small, this new perturbation will create no new orbits below a sufficiently large action threshold $\mathcal{A}_0$. \\ \par

So, we are thus now able to set up the contact homology chain complex. The rest of this section will be a discussion on how to control the $l$-invariant. We remark that after the perturbations above, we get $\mathcal{A}(O_h)=2\pi h_2(r_+)(1+\delta\mu(\theta_-))$. This will be the precise $l$-invariant.

\begin{Claim}\label{cldegree}
$O_h$ has degree 1.
\end{Claim}

\begin{proof}
The grading of any null-homologous orbit $\gamma$ in contact homology algebra is given by the formula $$|\gamma|=CZ_\tau(\gamma)+n-3+ \langle 2c_1(\xi,\tau),A\rangle \in \mathbb{Z}/c_1(\xi) \cdot H_2(Y)$$
for any trivialization $\tau$ and any null-homology A of $\gamma$. In our case, $n=2$ and as we show $CZ_\tau(O_h)=0$ and $\langle 2c_1(\xi,\tau),A\rangle =2$. \par
Let's start with the calculation of $CZ_\tau(O_h)$. Since initially our orbits are degenerate, the formula in order to calculate Conley-Zehnder indices of simply covered orbits is, according to lemma 2.4 in \cite{Bourgeois},
$$CZ(\gamma^p_{T'})=\mu(S_{T'})-\frac{1}{2}dim(S_{T'})+index_p(f_{T'})$$
where to recall things, $\phi_T$ is the Reeb flow for time $T$, $N_T=\{p\in Y\mid \phi_T(p)=p \}$ and $S_T$ the quotient of $N_T$ under the Reeb flow. $\mu(S_{T'})$ is the generalized Conley-Zehnder (see \cite{2013arXiv1307.7239G}). $T'$ is an action level less or equal to $T$. Recall that perturbation is made by fixing an action level $T$ and then the result of this perturbation process is that all orbits of action $\leq T$ become non-degenerate. In particular, here $N_{T'}$ is the torus of radius $r=r_+$ foliated by degenerate horizontal orbits, $S_{T'}=S^1$, $f_{T'}$ a Morse function on $S_{T'}$ and $p$ a critical point of $f_{T'}$ corresponding to some degenerate orbit. \par
We have $dim(S_{T'})=1$, $index_p(f_{T'}) \in \{0,1\}$ and as we now show $\mu(S_{T'})=\frac{1}{2}$. Thus, the elliptic orbit corresponding to the critical point of index 1 will have $CZ(O_e)=1$ and the hyperbolic orbit $CZ(O_h)=0$. \par
In coordinates $(\theta,r,\phi)$, the flow in general is given by
$$\phi_t(\theta,r,\phi)=(\theta+\frac{h_2'(r)}{D(r)} t,r,\phi-\frac{h_1'(r)}{D(r)} t)$$
So we get,
$$d\phi_t=\begin{pmatrix} 
1 & \Big(\frac{h_2'(r)}{D(r)}\Big)' t & 0 \\
0 & 1 & 0 \\
0 & \Big(\frac{-h_1'(r)}{D(r)}\Big)' t & 1
\end{pmatrix}$$
We will perform a symplectic change of basis so as to calculate the index using some helpful axiom. It is also helpful to work in this basis for the relative Chern class term. The new basis for the linearized flow will be $\langle \partial_\theta,\partial_r,-h_2(r)\partial_\theta+h_1(r)\partial_\phi \rangle$.

So now the matrix is given by
$$d\phi_t=\begin{pmatrix} 
 1 & 0 & 0 \\
 0 & 1 & 0 \\
 0 & f(r) t & 1
 
\end{pmatrix}
$$

where $f(r)=-\frac{h_1''(r)h_2'(r)-h_1'(r)h_2''(r)}{D(r)^2}$. We restrict the linearized flow on $\xi$ and we look at $t=1$. This restriction looks like

$$d\phi_t=\begin{pmatrix} 
  1 & 0 \\
  f(r)t & 1
 
\end{pmatrix}
$$

We aim to use the symplectic shear axiom. In our case, $h_2'(r_+)=0$ hence $sgn(f)>0$. Making the obvious last symplectic change of basis, the matrix for $d\phi_t$ when restricted to $\xi$ looks like 
$$
\begin{pmatrix}
1 & -f(r)t\\
0 & 1
\end{pmatrix}
$$
with $sgn(-f)<0$. Thus, by the symplectic shear axiom $\mu(d\phi_1)=-\frac{sgn(-f)}{2}=\frac{1}{2}$. Hence, as expected $\mu(S_{T'})=\frac{1}{2}$.\\
Let's now focus on $\langle 2c_1(\xi,\tau),A\rangle=2$. We have to pick a section of $\xi$, constant along $O_h$ with respect to our given trivialization for $\xi$. This is $\langle \partial_r,h_2(r)\partial_\theta-h_1(r)\partial_\phi \rangle$. We extend it over A and count its zeroes. We choose the second basis vector here. This only vanishes at the origin of the disk $A$ positively once. Thus, as it is expected $\langle 2c_1(\xi,\tau),A\rangle=2$.
\end{proof}

This means that $O_h$ has degree one more than the empty word.

\begin{Claim} \label{claction}
Any other orbit bounding a holomorphic plane must have action more than that of $O_h$.
\end{Claim} 

As mentioned earlier, this depends on the Lutz twist modification parameters and especially on the function $h_2(r)$ which can be chosen accordingly in order to ensure this. We now provide a rigorous proof of this statement.
 
\begin{proof}
Let $A$ be the action of the lowest action orbit bounding a unique pseudoholomorphic plane before the Lutz twist. If there is no such orbit we let $A=+\infty$. We need to choose $h_1(r),h_2(r)$ satisfying the conditions described in section \ref{conditions} in order to be able to perform the Lutz twist. We recall that $r_+<r_+'$ are the two values of $r$ for which the function $h_1(r)$ vanishes. There, the tori $L_{r_+}$ and $L_{r_+'}$ are foliated by horizontal Reeb orbits and the minima for the action occur. We have to perturb the Morse-Bott torus $L_{r_+}=\{r=r_+\}$ in order to get an isolated orbit of least action $O_h$ which bounds a unique pseudoholomorphic plane. Recall that then $\mathcal{A}(O_h)=2\pi h_2(r_+)(1+\delta\mu(\theta_-))$. \par
We have to impose some additional requirements in order to know what the $l$-invariant after the Lutz twist is. The first one is $2\pi h_2(r_+)<A$ or equivalently $h_2(r_+)<\frac{A}{2\pi}$ and the second one is $|h_2(r_+)(1+\delta\mu(\theta_-)| < |h_2(r_+')|$. The combination of both guarantees first that the hyperbolic orbit at $r=r_+$ is the one with the least possible action among all Reeb orbits of the contact manifold.
\end{proof}

 We denote the empty word by $1$. In order to obtain the needed count and thus get the coefficient $\langle \bd O_h,1 \rangle$ (which we need to it be equal to 1), one needs to ensure transversality at the holomorphic plane $u_0$, or in other words to ensure that the linearized Cauchy-Riemann operator at $u_0$ is surjective. Of course, if one wants to use techniques from \cite{MR3981989} and thicken the moduli spaces in order to obtain a proper count of curves, i.e. coefficients for the differential and thus a well defined contact homology algebra, they have to make sure that this thickening process does not alter the coefficient $\langle \bd O_h,1 \rangle$, which geometrically at least was calculated to be equal to 1. \\ \par 
 
 Problems that arise when compatifying the moduli space are multiply covered curves or breaking along Reeb orbits. Intuitively, since $O_h$ is embedded, we do not have to worry about multiple covers and since $O_h$ has the lowest action, no breaking of the pseudoholomorphic plane into buildings can occur as it would have to break along an orbit of lower action than that of $O_h$ and there are no such orbits. This intuition is backed up by part $(iv)$ of Theorem 1.1 in \cite{MR3981989}. In words, it states that when the moduli space is of dimension 0 and regular, then the algebraic count we obtain by thickening agrees with the geometric count we already have. Concretely, in our case, the geometric count is 1 and the formal algebraic count we get after setting up $CH(Y,\lambda)$ is also 1. Thus, the coefficient $\langle \bd O_h,1 \rangle$ is indeed 1 as needed.\\ \par
 
 As mentioned above, $u_0$ is a leaf of a stable finite energy foliation so transversality holds and moreover all neighboring
finite energy surfaces obtained by the implicit function theorem are also leaves
of the foliation. For more on this, the interested reader should consult \cite{Wend}, section 4.5.\par 
\begin{Claim}\label{clposend}
$O_h$ is not the positive end of any pseudoholomorphic cylinder.
\end{Claim}

\begin{proof}
The differential decreases action and $O_h$ is designed to be the orbit of least action among all Reeb orbits.
\end{proof}
The fact that this is the unique holomorphic plane bounded by $O_h$ comes from an argument regarding positivity of intersections of pseudoholomorphic curves in 4 dimensions as presented by Bourgeois and Van Koert in \cite{MR2646902} adapted to our case.
\begin{Claim}\label{clunique}
$u_0$ is the unique holomorphic plane bounded by $O_h$, so $\partial O_h=1$.
\end{Claim} 

\begin{proof}
A general plane in the symplectization looks like
\begin{equation}\label{gen}
u(s,t)=(u_a(s,t),u_\theta(s,t),u_r(s,t),u_\phi(s,t))
\end{equation}
We split the proof into two cases.~\\

\underline{Case 1:} If $u_\theta$ is constant, then $u$ is equivalent to $u_0$. \par
Take $r_0$ a regular value of $u_r(s,t)$ and a circle $\gamma$ in the preimage $u_r^{-1}(r_0)$. Consider $\partial_{\tilde{\psi}}, \tilde{\psi} \in [0,2\pi)$ the tangent vector to $\gamma$ on the Riemann surface $\Sigma$. We note that $\partial_{\tilde{\psi}}u$ has components only in the $\phi$ direction. We prove that in the end of the proof for case 1. Now, $u$ passes through $\{r=r_0\}$ with constant $a$ and $\theta$ directions, hence translating $u$ in the $a$ direction, $u$ turns out to intersect $u_0$ in at least the circle at $\{r=r_0\}$. This yields a contradiction to positivity of intersections and thus any other such $u$ is equivalent to $u_0$.\\
Let us now see why $\partial_{\tilde{\psi}}u$ has components only in the $\phi$ direction. We decompose the tangent space of the symplectization as $\langle v_1,v_2,\partial_a,R_\alpha \rangle$ where $v_1,v_2$ the trivialization of the contact structure as defined in the beginning of the section. We write

\begin{equation*}
\begin{split}
\partial_{\tilde{\psi}}u & =A(s,t)v_1+B(s,t)v_2+\Gamma(s,t)\partial_a+\Delta(s,t)R_\alpha \\
& =A(s,t)\partial_r+\frac{-B(s,t)h_2(r)+\Delta(s,t)h_2'(r)}{D(r)}\partial_\theta+\Gamma(s,t)\partial_a+\frac{B(s,t)h_1(r)-\Delta(s,t)h_1'(r)}{D(r)}\partial_\phi
\end{split}
\end{equation*}

Since we are restricted at a regular value $r=r_0$ we get $A=0$. Moreover, $u_\theta$ is constant so $-Bh_2+\Delta h_2'=0$. Now, applying $J$ to $\partial_{\tilde{\psi}}u$ we get

\begin{equation*}
\begin{split}
J\partial_{\tilde{\psi}}u & =A(s,t)\beta(r)v_2-\frac{B(s,t)}{\beta(r)}v_1+\Gamma(s,t)R_\alpha-\Delta(s,t)\partial_a \\
& =A(s,t)\beta(r)(\frac{-h_2(r)}{D(r)}\partial_\theta+h_1(r)\partial_\phi)+\frac{-B(s,t)}{\beta(r)}\partial_r+\Gamma(s,t)\frac{h_2'(r)\partial_\theta-h_1'(r)\partial_\phi}{D(r)}-\Delta(s,t)\partial_a
\end{split}
\end{equation*}
Since $u$ is $J$-holomorphic, $J$ has to preserve the tangent space and the $\theta$-component is constant. Thus the coefficient of $\partial_\theta$ is equal to 0, i.e.
$$\frac{-A\beta(r)h_2(r)+\Gamma h_2'(r)}{D(r)}=0$$
We already have $A=0$, thus $\Gamma=0$. So the plane $u$ passes through $\{r=r_0\}$ with constant $a$ and $\theta$ coordinates.

\begin{flushleft}
\underline{Case 2:} If $u_\theta$ is not constant, then positivity of intersections for pseudoholomorphic curves is contradicted.
\end{flushleft}
If $u$ is a solution to the Cauchy-Riemann equations of the general form \ref{gen}, then $$u_c(s,t)=(u_a(s,t),u_\theta(s,t)+c,u_r(s,t),u_\phi(s,t))$$
is also one. This is asymptotic to some other orbit $\gamma_c$. Concretely $\gamma_0(t)=(\theta_0,r_0,-\frac{h_1'(r)}{D(r)}t)$ and $\gamma_c(t)=(\theta_0+c,r_0,-\frac{h_1'(r)}{D(r)}t)$. They are obviously not linked to each other so $lk(\gamma_0,\gamma_c)=0$. Calculating the linking number in another way we get a contradiction. Since $u_\theta$ is not constant, there is a constant $c$ s.t. $u_c\cap u_0 \neq \emptyset$. By positivity of intersections, $lk(\gamma_0,\gamma_c)>0$, which yields the desired contradiction.
\end{proof}
~\\
What we showed is that the contact element becomes exact at most at filtration level $\mathcal{A}(O_h)$, thus its class vanishes for action less or equal than $\mathcal{A}(O_h)$. This means that the $l$-invariant is less or equal than action of $O_h$ which is $2\pi h_2(r_+)\cdot (1+\delta\mu(\theta_{-}))$. The fact that this is actually the $l$-invariant comes from the fact that any other orbit bounding a holomorphic plane must have action more than that of $O_h$, i.e. claim \ref{claction}. 
\par

Summarizing what has been achieved, we have shown that using Wendl's construction, we are in the position to know precisely what the $l$-invariant is, i.e. what is the action level for which the identity of contact homology algebra becomes exact or equivalently what is the right endpoint of the bar corresponding to $1$ in the barcode of the persistence module $CH^{\leq t}(Y,\lambda)$. More concisely,

\begin{Prop}\label{linvariant}
The $l$-invariant of the contact form $\alpha^{\delta,\delta'}$ which is an appropriate perturbation of $h_1(r)d\theta+h_2(r)d\phi$ in the Lutz tube is $l(\alpha)=\mathcal{A}(O_h)=2\pi h_2(r_+)\cdot (1+\delta\mu(\theta_{-}))$
\end{Prop}

\vspace{10pt}

\subsection{Construction of the 2-parameter family} \label{2family} \text{}\\ \par

The parameter domain of this family is $$H_\epsilon=\{(\ln(\sqrt{x}),\ln(y))\in\mathbb{R}^2\mid \ln(y)<\epsilon, \ x,y>0\}$$ It is depicted in figure \ref{fig:H_epsilon}.

\begin{figure}[h!]
    \centering
    \includegraphics[scale=0.9]{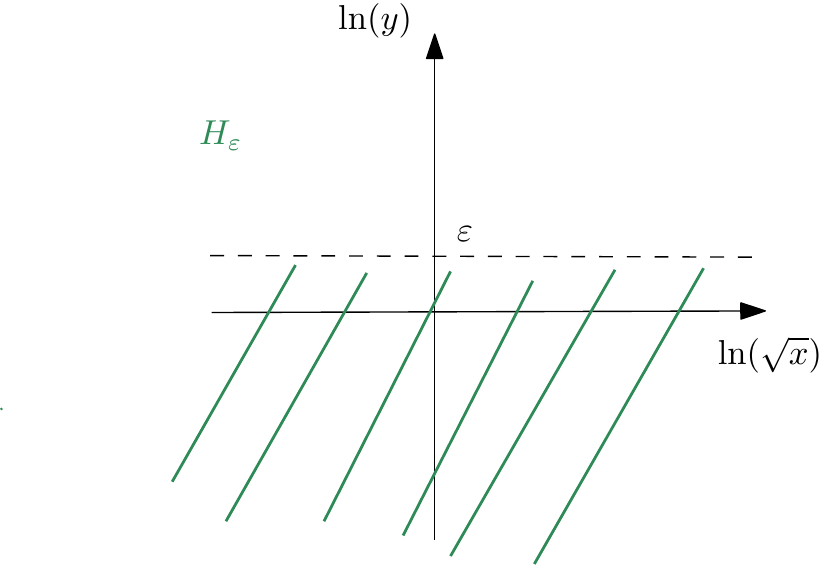}
    \caption{The source $H_\epsilon$}
    \label{fig:H_epsilon}
\end{figure}

The $\epsilon$ involved in the definition of $H_\epsilon$ depends on the contact form $\alpha$. It is less than the minimum of two quantities. The first one is the lowest action of a Reeb orbit of $(Y,\xi)$ before  we perform the Lutz twist. The second one is the action of a special Reeb orbit which helps us control volume. These are more thoroughly explained below.\\ \par

For $\varepsilon>0$ and for a knot K, we denote by $nbd(K,\varepsilon)$ a tubular neighbourhood of the knot $K$ of radius $\varepsilon$. Fix two transverse knots $K_1,K_2$ to $\xi$ in $Y$. Fix $\varepsilon>0$ sufficiently small such that $nbd(K_1,\varepsilon)\cap nbd(K_2,\varepsilon)=\emptyset$ and a contact form $\lambda$ on Y with $\ker(\lambda)=\xi$ which in tubular neighborhoods of $K_1,K_2$ looks like the standard form $d\theta+r^2d\phi$.
We perform a Lutz twist along $K_2$ and obtain an overtwisted contact form on Y. We work with the full twist here as we need to preserve the homotopy type of the plane field $\xi$. In order to be able to control the $l$-invariant, we impose the restriction that $ 2\pi h_2(r_+)(1+\delta \mu(\theta_-))<\epsilon$. This yields a differential form of the form 
$$\lambda_{ot}=\begin{cases}
\lambda & \text{on } Y\backslash (S^1\times D^2) \\
h_1(r)d\theta+h_2(r)d\phi & \text{on } S^1\times D^2
\end{cases}$$
where in the above formula $K_2$ is identified with $S^1$. In order to avoid confusion, we emphasize that $\epsilon$ is related to the $l$-invariant and $\varepsilon$ to the radius of the tubular neighbourhood of $K_2$.
\\ \par

The coordinates on the tubular neighborhood $nbd(K_2,\varepsilon)$ are $(\theta,r,\phi)\in S^1\times (0,\varepsilon)\times [0,2\pi)$.
We normalize by requiring $Vol(Y,\lambda_{ot})=\int_Y\lambda_{ot} \wedge d\lambda_{ot}=1$. Let $L$ be the $l$-invariant, namely the action of the horizontal orbits $\{r=r_+\}$ in $S^1 \times D^2$ or equivalently the lowest filtration level for which a primitive for the unit of the contact homology algebra appears. This normalization has the effect that $(0,\ln(L)) \in H_\epsilon$ maps under the bi-Lipschitz embedding to $\lambda_{ot}$. \\ \par 

We have that the $l$-invariant is the least action of an orbit bounding a unique pseudoholomorphic plane. After the Lutz twist we performed, and the subsequent perturbations, it is equal to $2\pi h_2(r_+)(1+\delta\mu(\theta_-))$, where $r_+$ is the smallest real number such that $h_2'(r_+)=0$.\\ \par

In order to affect the volume and make it equal to k, it is enough to multiply $\lambda_{ot}$ by $\sqrt{k}$. Moreover, we can modify the $l$-invariant just by modifying the choice of $h_2(r)$. We have to be careful though as we can only obtain information about the $l$-invariant as long as $2\pi h_2(r_+)(1+\delta\mu(\theta_-))$ is less than the next filtration level for which a primitive for the empty word appears. This, as already explained, is related to the quantity $\epsilon$ in the definition of $H_\epsilon$ before theorem \ref{quasiisometricembedding}. For this modification we require that $2\pi h_2(r_+)(1+\delta \mu(\theta_-))=l$. Pictorially, figure \ref{fig:Lutz} suggests that the modification has the result that the y-intercept of the path $(h_1(r),h_{2,l}(r))$ is $\frac{l}{2\pi(1+\delta \mu(\theta_-))}$. Note that the only restrictions we have for the functions $h_1(r),h_2(r)$ is the behavior of this path close to the endpoints of [0,$\varepsilon$] and that the vector $(h_1(r),h_{2,l}(r))$ has to wind around the origin of $\mathbb{R}^2$ once without ever being parallel to $(h_1'(r),h_{2,l}'(r))$. No other requirement in the interior of $[0,\varepsilon]$ (in particular close to $r_+\in (0,\varepsilon)$) is assumed.\\ \par

Note that this last alteration of the contact form does not have a significant impact on volume, as it is enough for the change to take place only in a tube of small radius $\varepsilon$. Yet, in order for the modification of the $l$-invariant to have no impact on volume, we compensate by multiplying the form by a bump function $\nu_l$, supported outside of the solid tube in question. In particular, $\nu_l$ is supported in the tubular neighborhood around $K_1$. This creates no new orbits of action less than that of the horizontal orbits $\{r=r_+\}$ as will be explained below.\\ \par

The next question that arises is about the $\epsilon$ in the definition of $H_\epsilon$. In short, $\epsilon$ is the logarithm of the largest controlled $l$-invariant one can have. The word control here means both being able to leave the volume of the contact manifold unchanged as we modify the $l$-invariant (this is where the number $B$ is needed in what follows) and determine what is the lowest action of a primitive for the unit of the contact homology algebra (and this is where the number $A$ comes from). The value of $\epsilon$ is determined as follows. \par 
First, let as before $A$ be the lowest action of a Reeb orbit of $(Y,\xi)$ before we perform the Lutz twist. Also, let \hypertarget{B}{$B$} be the action of the $D^2$-parametrized family of Reeb orbits in the tubular neighborhood of $K_1$. Since the contact form locally is $d\theta+r^2d\phi$, these orbits are specified by the local Reeb vector field $\partial_\theta$. Thus, we have to set $\epsilon := \min\{\ln(A),\ln(B)\}$.
\\ \par

After briefly describing the construction, we are now able to provide it concretely and thus define the embedding $F$. This embedding sends $(\ln(\sqrt{k}),\ln(l))\in H_\epsilon$ to the form

\begin{equation}\label{constr}
\alpha_{k,l}=\sqrt{k}\cdot\lambda_{ot}=\begin{cases}
\sqrt{k}\cdot (\nu_l+1) \cdot \lambda & \text{on } Y\backslash (S^1\times D^2) \\
\sqrt{k}\cdot(h_1(r)d\theta+h_{2,l}(r)d\phi) & \text{on } S^1\times D^2
\end{cases}
\end{equation}

where $h_2(r)$ and $h_{2,l}(r)$ are given in figure \ref{fig:Lutz} and $\nu_l$ will be explained shortly. Later, it will be helpful for our calculations to give an explicit parametrization for this path $(h_1(r),h_{2,l}(r))$ and this is what we will do. It will be mostly part of two ellipsoidal arcs. The reason for considering $(\ln(\sqrt{k}),\ln(l))$ instead of $(k,l)$ is that multiplication of forms by constants, i.e. flowing uniformly using the Liouville vector field in the symplectization is translated to linear movement in the source space $H_\epsilon$. \\ \par

\begin{figure}[h!]
    \centering
    \includegraphics[scale=0.8]{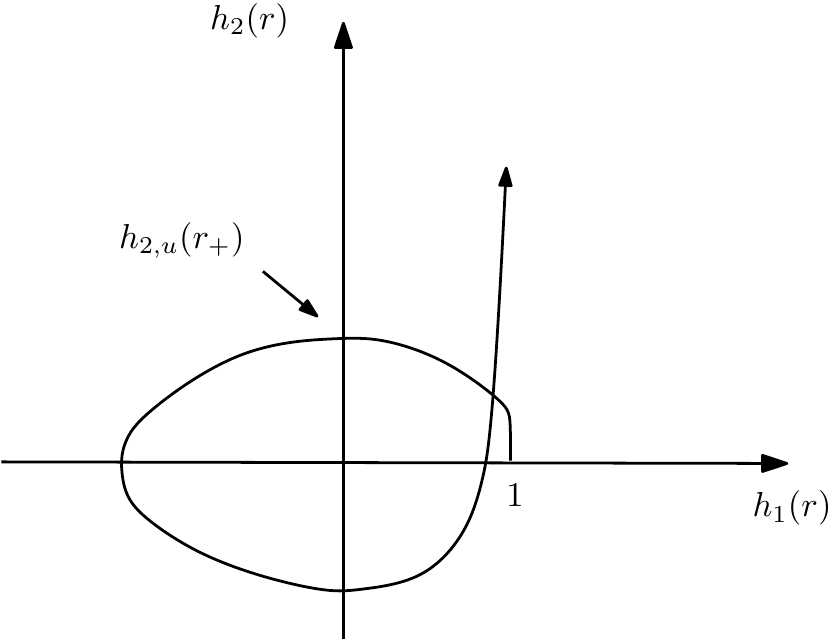}
    \caption{The path defining the particular Lutz twist}
    \label{fig:Lutz}
\end{figure}

We now explain the compensating function $\nu_l$. Figure \ref{fig:nu} helps on this task. In order to describe $\nu_l$ we work in the tubular neighborhood of $K_1$. This is equipped with the contact form $\lambda=d\theta+r^2d\phi$. The coordinates are given by $(\theta,r,\phi)\in[0,2\pi]\times[0,\delta]\times[0,2\pi]$. Pick $\epsilon_0<<\min\{\delta,2\pi\}$. Then $\nu_l$ is a bump function supported within $[0,\epsilon_0]\times[0,\delta-\epsilon_0]\times[0,2\pi]$. The form locally now becomes $\lambda_l=(\nu_l+1)\cdot (d\theta+r^2d\phi)$. The function $\nu_l$ is additionally required to have the property that if modifying $h_{2,l}(r)$ leads to a change in volume of $Y$ by adding $V_0 \in \mathbb{R}_{>0}$, $\nu_l$ controls the volume of the tubular neighborhood of $L_1$ and yields $Vol(K_1\times D^2,\lambda_l)=Vol(K_1\times D^2,\lambda)-V_0$.\par
Note that this process creates no new orbits of action less than $B$, so in particular no orbits of action less than $A$ after this modification could serve as a primitive of the unit of contact homology. Recall $A$ is the least action for a primitive of the empty word and since the volume compensating process does not reduce the action of any orbit, it cannot yield primitives for the empty word of action less than $A$.\\ \par

\begin{figure}
    \centering
    \includegraphics[scale=0.8]{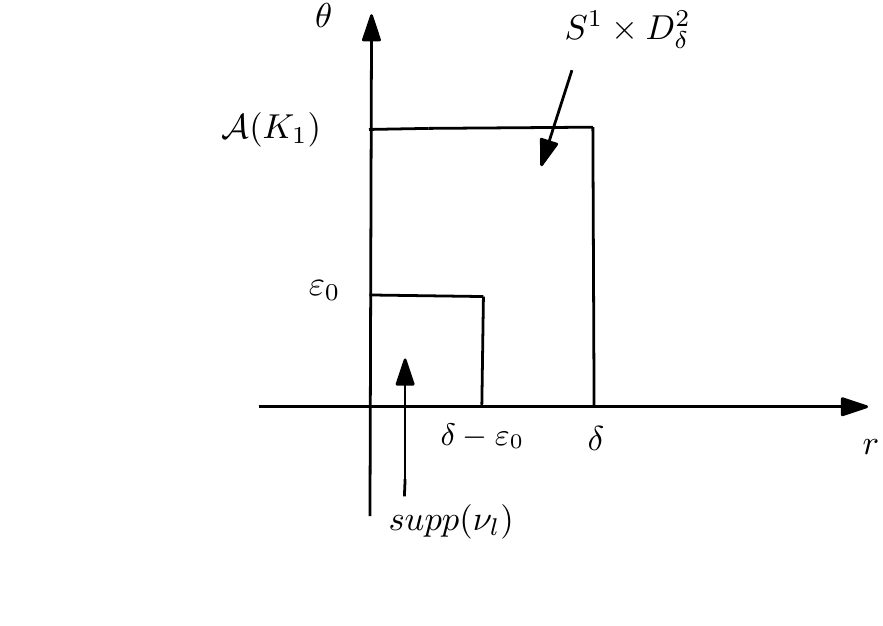}
    \caption{The support of the compensating function $\nu_l$}
    \label{fig:nu}
\end{figure}

 One can easily check that $\alpha_{k,l}$ is a contact form and less importantly see that multiplying our form by $\sqrt{k}$ corresponds to flowing using the Liouville vector field in the symplectization for $\ln(\sqrt{k})$ \say{seconds}. This has exactly the effect that the contact form gets multiplied by $\sqrt{k}$.\\ \par 

As expected, it turns out that $\epsilon$ can be directly related to the systolic ratio of the initial contact form $\lambda$. We have $\rho_{sys}(\alpha_{k,l})=\frac{T_{min}(\alpha_{k,l})^2}{Vol(Y,\alpha_{k,l})}=\frac{l^2}{k}$, as far as $O_h$ discussed in the previous section is the orbit with the least action. If before the twist, we can find a transverse $K_1$ with its surrounding orbits having action $\geq \mathcal{A}(O_h)$, we have $\rho_{sys}(\alpha_{k,l})=\frac{l^2}{k}$. In particular, in our examples we have $\sqrt{k \cdot \rho_{sys}(\alpha_{k,l})} = l= \ln(\epsilon)$, which in other words says $\epsilon=e^{\sqrt{k \cdot \rho_{sys}(\alpha_{k,l})}}$. Moreover, since we have $\rho_{sys}(\alpha_{k,l}) \leq \rho_{sys}(\lambda)$ essentially by construction, any $\epsilon \leq e^{\sqrt{k \cdot \rho_{sys}(\lambda)}}$ works for the construction. We can possibly choose an even larger $\epsilon$ if the orbit bounding a unique pseudoholomorphic plane for $\lambda$ is of large enough action.

\subsection{Proof of the bi-Lipschitz embedding theorem}\label{pfofqiemb}\text{} \\ \par

Our goal for this section is to prove the following inequalities 
$$\frac{1}{2}d_{\infty}(\Vec{x},\Vec{y})\leq d_{CBM}(F(\Vec{x}),F(\Vec{y})) \leq 2d_{\infty}(\Vec{x},\Vec{y})$$
where $F:(\mathbb{H},d_{\infty})\rightarrow (\mathcal{C}_{ot}^{Y,\xi},d_{CBM})$ the bi-Lipschitz embedding in question.\par To be more precise, the parameter domain of the 2-parameter family of forms $\alpha_{k,l}$ is $H_\epsilon=\{(\ln(\sqrt{x}),\ln(y))\in\mathbb{R}^2\mid \ln(y)<\epsilon, \ x,y>0\}$, so the embedding will actually be $F:(H_\epsilon,d_{\infty})\rightarrow (\mathcal{C}_{ot}^{Y,\xi},d_{CBM})$. This is not an issue though since $(\mathbb{H},d_\infty)$ and $(H_\epsilon,d_\infty)$ are isometric. The left inequality will be proved in subsection \ref{lineq} and the right one which is the more involved one in subsection \ref{rineq}.

This map will be defined shortly. Recall that $\epsilon$ is a number chosen to be less than the action of a primitive for the empty word of orbits before we perform any Lutz twist. We only have to worry about picking a sufficiently small $\epsilon$ in the case when $Y$ is already algebraically overtwisted. What we will actually prove will be slightly stronger, yet less symmetric.\\ \par

The following lemmas will be helpful below. They describe the behavior of volume and $l$-invariant under dilation of the contact form.

\begin{Lemma}\label{volume}
 If $\alpha \prec \beta$, then $Vol((Y,\alpha)) \leq Vol((Y,\beta))$. Also, $Vol((Y,C\cdot \alpha))=C^2\cdot Vol((Y,\alpha))$.
\end{Lemma} 

\begin{proof}
First, we prove that $Vol(\phi(Y))=Vol(Y,\alpha))$, or in other words that we can only map $(Y,\alpha)$ into $SY$ in a volume preserving manner. Recall that $\alpha \prec \beta$ means that there exist a cs-embedding $\phi:(Y,\alpha) \rightarrow W(\beta)\subset (SY,d(r\beta))$ with $\phi^*(r\beta+\eta)=\alpha$, for some exact, compactly supported in a neighbourhood of $\phi(Y)$, one-form $\eta=df$. Recall also that the way we measure the volume of a hypersurface in $SY$ is first by considering the form $\alpha_0=r\beta|_{\phi(Y)}$ and then 
$Vol(\phi(Y))=\displaystyle\int_{\phi(Y)}\alpha_0\wedge d\alpha_0$.\\ We thus have,

\begin{gather*}
Vol(\phi(Y))=\displaystyle\int_{\phi(Y)}\alpha_0\wedge d\alpha_0=\displaystyle\int_{Y}\phi^*\alpha_0\wedge d(\phi^*(\alpha_0))=\displaystyle\int_{Y} (\alpha-\phi^*(\eta))\wedge d(\alpha-\phi^*(\eta)) \\
=\displaystyle\int_{Y}\alpha \wedge d\alpha - \displaystyle\int_{Y}\alpha \wedge d(\phi^*\eta) - \displaystyle\int_{Y} \phi^*\eta \wedge d\alpha +\displaystyle\int_{Y} \phi^*\eta\wedge d(\phi^*\eta)\\
= \displaystyle\int_{Y}\alpha \wedge d\alpha + \displaystyle\int_{Y} \phi^*\eta \wedge d\alpha 
=\displaystyle\int_{Y}\alpha \wedge d\alpha + \displaystyle\int_{\bd Y=\emptyset} \phi^*f \wedge \alpha\\
=\displaystyle\int_{Y}\alpha \wedge d\alpha=Vol(Y,\alpha)
\end{gather*}
where the $5^{th}$ equality above follows from the fact that $\eta$ is exact and the $6^{th}$ one by Stokes' theorem.

The embedding $\phi(Y)$ yields that there exists a Liouville cobordism between $\phi(Y)$ and $\{1\}\times Y$ in $(SY,d(r\beta))$. Stokes' theorem implies the first claim. The second claim is obvious from the definition of the volume of $Y$.
\end{proof}

\begin{Rem}
The proof of this shows in particular that the allowed cs-embeddings preserve the volume.
\end{Rem}

\begin{Lemma}\label{linvineq}
If $\alpha \prec \beta$, then $l(\alpha) \leq l(\beta)$. Also, $l(C\cdot \alpha)=C \cdot l(\alpha)$
\end{Lemma}

\begin{proof}
The assumption of the lemma means that there is an embedding $\phi : (Y,\alpha) \rightarrow W(\beta)$ which implies the existence of a trivial Liouville cobordism between $\phi(Y)$ and $Y_\beta$. So, we get a map $CH_*(\beta)^{\leq t} \rightarrow CH_*(\alpha)^{\leq t}$ which maps 0 to 0. In particular, the filtration level for which the contact invariant for $\beta$ vanishes, i.e. $l(\beta)$, has to be larger or equal to the filtration level for which the contact invariant for $\alpha$ vanishes, i.e. $l(\alpha)$. This proves $l(\alpha) \leq l(\beta)$.\par The second claim follows by the definition of the action of an orbit and the fact that multiplying the contact form by some number C just rescales the dynamics.
\end{proof}

\subsubsection{Left Inequality}\label{lineq} \text{} \\ \par

 Let $\vec{x}=(\ln(\sqrt{k_1}),\ln(l_1))$ and $\vec{y}=(\ln(\sqrt{k_2}),\ln(l_2))$ so that $F(\vec{x})=\alpha_{k_1,l_1}=\alpha$ and $F(\vec{y})=\alpha_{k_2,l_2}=\beta$. By the definition of $d_{CBM}$ and the previous lemma, as far as volume is concerned, we have that  
$$
\frac{Vol(F(\Vec{y}))}{C^2}\leq Vol(F(\Vec{x})) \leq C^2\cdot Vol(F(\Vec{y}))
$$
for any $C$ such that $\alpha_{k_2,l_2} \prec C \cdot \alpha_{k_1,l_1}$ and $\alpha_{k_1,l_1} \prec C \cdot \alpha_{k_2,l_2}$.\\ \par

 The above inequalities are equivalent to 
 \begin{equation}
 \label{volm}\ln{\frac{1}{C}}\leq \frac{1}{2} \ln{\frac{Vol(F(\Vec{x}))}{Vol(F(\Vec{y}))}}=\frac{1}{2} \ln\Big(\frac{k_1}{k_2}\Big)= \ln(\sqrt{k_1})-\ln(\sqrt{k_2})\leq \ln{C}
 \end{equation} \par
 
 Because of the symmetry in the definition of $d_{CBM}$, we also get
 $$
\frac{Vol(F(\Vec{x}))}{C^2}\leq Vol(F(\Vec{y})) \leq C^2\cdot Vol(F(\Vec{x}))
$$
for any $C$ such that $\alpha_{k_2,l_2} \prec C \cdot \alpha_{k_1,l_1}$ and $\alpha_{k_1,l_1} \prec C \cdot \alpha_{k_2,l_2}$.\\ \par

These inequalities are again equivalent to 

\begin{equation}
 \label{volm1}\ln{\frac{1}{C}}\leq \frac{1}{2} \ln{\frac{Vol(F(\Vec{y}))}{Vol(F(\Vec{x}))}}=\frac{1}{2} \ln\Big(\frac{k_2}{k_1}\Big)= \ln(\sqrt{k_2})-\ln(\sqrt{k_1})\leq \ln{C}
 \end{equation} 

Combining now \ref{volm} and \ref{volm1} and taking infimum over such $C$ we obtain 
\begin{equation}\label{LHSV}
    |\ln(\sqrt{k_1})-\ln(\sqrt{k_2})| \leq d_{CBM}(\alpha,\beta)
\end{equation}

\par

 Next, for any C such that $\alpha_{k,l_1} \prec C \cdot \alpha_{k,l_2}$ and $\alpha_{k,l_2} \prec C \cdot \alpha_{k,l_1}$ or equivalently $\alpha_{k,l_1} \prec C \cdot \alpha_{k,l_2}$  and $\frac{1}{C}\alpha_{k,l_2} \prec \alpha_{k,l_1}$, we get 
 
 $$\frac{1}{C}\cdot \alpha_{k,l_2} \prec \alpha_{k,l_1} \prec C \cdot \alpha_{k,l_2}$$
 which implies by the previous lemma that
 
 $$\frac{1}{C}l_2 \leq l_1 \leq Cl_2$$
 
Dividing by $l_2$ and taking logarithms yields

$$\ln\Big(\frac{1}{C}\Big)\leq \ln\Big(\frac{l_1}{l_2}\Big) \leq \ln(C)$$ \par

Again by the symmetry in the definition we obtain a similar inequality, 

$$\ln\Big(\frac{1}{C}\Big)\leq \ln\Big(\frac{l_2}{l_1}\Big) \leq \ln(C)$$

Then, using both inequalities and taking infimum over all such $C$ we obtain 

\begin{equation}\label{LHSL}
    |\ln(l_2)-\ln(l_1)| \leq d_{CBM}(\alpha,\beta)
\end{equation}
 Hence (\ref{LHSL}) and (\ref{LHSV}) yield 
\begin{equation}\label{LHS}
    d_\infty(\vec{x},\vec{y}) \leq d_{CBM}(\alpha,\beta)
\end{equation} 

\subsubsection{Right Inequality}\label{rineq} \text{ } \\ \par

Obtaining the right hand inequality of the bi-Lipschitz condition above follows by the fact that changing the parameters in the 2-parameter family of contact forms is compatible with the following triangle inequality. In order to keep things clear, we abuse notation and instead of $d_{CBM}(\alpha_{k_1,l_1},\alpha_{k_2,l_2})$  we write $d_{CBM}((\sqrt{k_1},l_1),(\sqrt{k_2},l_2))$. 

\begin{gather}
    d_{CBM}(F(\Vec{x}),F(\Vec{y}))=d_{CBM}((\sqrt{k_1},l_1),(\sqrt{k_2},l_2)) \leq \nonumber \\ d_{CBM}\Big(\Big(\sqrt{k_1},l_1\Big),\Big(\sqrt{k_2},\sqrt{\frac{k_2}{k_1}}l_1\Big)\Big)+d_{CBM}\Big(\Big(\sqrt{k_2},\sqrt{\frac{k_2}{k_1}}l_1\Big),\Big(\sqrt{k_2},l_2\Big)\Big) 
    \label{keyinequality}
\end{gather}

 We are going to study the behavior after modifying each parameter separately, though in the case of volume this has a small effect on the $l$-invariant and thus we have to travel through the point $\Big(\ln(\sqrt{k_2}),\ln\Big(\sqrt{\frac{k_2}{k_1}}l_1\Big)\Big)$ in $H_\epsilon$. This is explained below.\\ \par

First, we modify volume. As mentioned, in contrast with the modification of the $l$-invariant, our modification procedure does not allow us to modify solely the volume, or in other words we cannot only move horizontally in $H_\epsilon$. For this reason, we will work with a view towards the triangle inequality \ref{keyinequality}. Let $\alpha=F(\Vec{x})=F(\ln(\sqrt{k_1}),\ln(l_1))$ and $\gamma=F(\Vec{y})=F(\ln(\sqrt{k_2}),\ln \big(\sqrt{\frac{k_2}{k_1}}l_1\big))$ with $k_1 \leq k_2$. We have
\begin{equation}
d_{CBM}(\alpha,\gamma)=\{\ln{C}\mid \frac{1}{C}\cdot \gamma \prec \alpha\prec C\cdot \gamma\} \leq \ln{\sqrt{\frac{k_2}{k_1}}}=\ln{\sqrt{k_2}}-\ln{\sqrt{k_1}} \nonumber
\end{equation}
because in this case we have that $\alpha=\sqrt{\frac{k_1}{k_2}}\cdot \gamma$, so one such $C$ that obviously works is $\max\Big\{\sqrt{\frac{k_1}{k_2}},\sqrt{\frac{k_2}{k_1}}\Big\}$. If we don't assume $k_1 \leq k_2$ we just have to use an absolute value in the above inequality. Hence, we obtain 

\begin{equation}\label{RHSV}
d_{CBM}(\alpha,\gamma) \leq \big|\ln{\sqrt{k_2}}-\ln{\sqrt{k_1}}\big|
\end{equation}

Note again that changing volume modifies the $l$-invariant precisely turning $l$ into $\sqrt{\frac{k_2}{k_1}}l$. This explains the form of the middle term appearing in \ref{keyinequality}.\\ \par

Let's now discuss the way to modify $h_2(r)$, i.e. effectively the $l$-invariant. Recall that $h_2(r)$ is a function from $[0,\varepsilon)$ to $\mathbb{R}$ satisfying the properties for a Lutz twist from section \ref{properties}. Recall also that in order to alter it, we additionally require $2\pi h_{2,l}(r_+)(1+\delta \mu(\theta_+))=l$ and that $\forall l$, $(h_1(r),h_{2,l}(r))$ is never parallel to $(h_1'(r),h_{2,l}'(r))$ in order for the Lutz twist to yield a contact form.\\ \par

We consider a smooth family $h_{2,t}:[0,\varepsilon)\times[l_1,l_2] \rightarrow \mathbb{R}$ interpolating between $h_{2,l_1}$ and $h_{2,l_2}$. This family is assumed $\forall t \in [l_1,l_2]$ to satisfy the 3 properties required for a Lutz twist as in section \ref{properties}. Namely, $\forall t \in [l_1,l_2]$ a Lutz twist can be performed using $h_{2,t}$ instead of $h_2$. It will also be helpful to assume either $\frac{\partial h_{2,t}}{\partial t} \geq 0$ or $\frac{\partial h_{2,t}}{\partial t} \leq 0$ depending on whether we increase or decrease the value of the $l$-invariant. This family of functions yields a family of contact forms on the manifold $Y$ given by

$$\alpha_t=\lambda_{ot}^t=\begin{cases}
(\nu_t+1) \cdot \lambda & \text{on } Y\backslash (S^1\times D^2) \\
(h_1(r)d\theta+h_{2,t}(r)d\phi) & \text{on } S^1\times D^2
\end{cases}$$

Where $\nu_t$ as described previously, is designed to compensate for the small change in volume when modifying $h_{2,l_1}$ or stated differently, to allow us to move vertically in $H_\epsilon$. We have already shown that the effect of this in dynamics is controlled well enough for our computations. Namely, no new orbits of less action are created when multiplying by the compensating function. We denote by $\xi_t$ the corresponding contact structures and we let $\gamma=\alpha_1$, $\beta=\alpha_2$. \\ \par

Since we have a smooth family of contact structures, Gray's stability theorem provides a  function $f_t$ such that $\psi^*_t\gamma_t=f_t\alpha_1=f_t\gamma$, where $\psi_t$ a smooth isotopy of $Y$. We will specify $f_t$ using Moser's trick as in the proof of Gray's theorem. \\ \par

First, we have 
\begin{equation}\label{rhs1}
d_{CBM}(\gamma,\beta) \leq ||\ln{f_1}-\ln{f_2}||_{\infty}=||\ln{f_2}||_{\infty}
\end{equation}

This can be seen as follows. We have that $f_1(y)=1, \forall y \in Y$. Moreover, $\beta=f_2\gamma$. This implies that $s_\beta(Y) \subseteq W(\beta) \subseteq W(||f_2||_{\infty}\cdot \gamma)$. It is also clear that $s_\gamma(Y) \subseteq W(\beta) \subseteq W(||f_2||_{\infty}\cdot \beta)$ as $\beta \neq \gamma$.\\

 \par
Now, working as in the proof of Gray's stability theorem, we have $\mu_t=\Dot{\alpha}_t(R_t)=\frac{d}{dt}(\ln{f_t})\circ \psi_t^{-1}$ which in this case translates to

\begin{equation}\label{defnofmu}
\mu_t=\frac{d}{dt}(h_{2,t}(r))\cdot\Big(\frac{-h_1'(r)}{D_t(r)}\Big)
\end{equation}
where $D_t(r)=h_1(r)h_{2,t}'(r)-h_1'(r)h_{2,t}(r)$.\\ \par

If we consider the 1-parameter family of contact forms $\{\alpha_{1,l}\}_{l=s}^{l=t}$, we have 

$$||\ln{f_t}-\ln{f_s}||_{\infty} = \Big|\Big|\int_s^t\mu_u \circ \psi_u du\Big|\Big|_{\infty}$$

Replacing $\mu$ by its formula as in (\ref{defnofmu}) and using the triangle inequality for integrals we get

\begin{equation}\label{tobebounded}
||\ln{f_t}-\ln{f_s}||_{\infty} \leq \int_s^t \Big|\Big|\frac{d}{du}(h_{2,u}(r))\cdot\Big(\frac{-h_1'(r)}{D_u(r)}\Big)\Big|\Big|_{\infty}du
\end{equation}

So, it all boils down then to finding an upper bound for the right hand side of equation (\ref{tobebounded}). Although the analysis can be done in general, we are going to work with the concrete case where $(h_1(r),h_2(r))$ is mostly part of arcs of two ellipses. This is both concrete and sufficient. \\ \par

To this end, we parametrize $(h_1(r),h_{2,l}(r))$ as follows. Recall that we denote by $\varepsilon$ the radius of the Lutz tube. Pick some small $\varepsilon_0$ such that $\varepsilon >> \varepsilon_0 >0$. This $\varepsilon_0$ will be the first time when the path $(h_1(r),h_{2,l}(r))$ changes behavior. Moreover, pick some small $\delta_1>0$ and $\delta_2>0$ such that $(1+\delta_1)\cos(2\pi \varepsilon_0)=1$ and $\varepsilon_0^2=(1+\delta_2)\frac{u\sin(2\pi \varepsilon_0)}{2\pi(1+\delta \mu(\theta_-))}$. These $\delta_1,\delta_2$ ensure continuity of the path $(h_1(r),h_{2,l}(r))$ since for $0\leq r \leq \varepsilon_0$ the path is $(h_1(r),h_{2,l}(r))=(1,r^2)$ and for $r$ slightly larger that $\varepsilon_0$ the path is $(h_1(r),h_{2,l}(r))=\Big((1+\delta_1)\cos(2\pi r),(1+\delta_2)\frac{u\sin(2\pi r)}{2\pi(1+\delta \mu(\theta_-))}\Big)$. So, to summarize, we let

$$h_1(r):=\begin{cases}
1, & 0 \leq r \leq \varepsilon_0 \\
(1+\delta_1)\cos(2\pi r), & \varepsilon_0 < r \leq \frac{\varepsilon}{2}
\end{cases}$$

and moreover we let

$$h_{2,u}(r):=\begin{cases}
r^2 & , \quad 0 \leq r \leq \varepsilon_0 \\
(1+\delta_2)\frac{u\sin(2\pi r)}{2\pi(1+\delta \mu(\theta_-))} & , \quad \varepsilon_0 < r \leq \frac{\varepsilon}{2}
\end{cases}$$ \par

The paths can be extended to $(\frac{\varepsilon
}{2},\varepsilon)$ in a similar way. This way needs to respect that for any time parameter $u$, the absolute value of the second $y$-intercept is larger than the one of the first $y$-intercept $h_{2,u}(r_+)$, it is independent of $u$ and close to $\varepsilon$ the paths become $(h_1(r),h_{2,u}(r))=(1,r^2)$. The first requirement is in order to control the $l$-invariant, the second to control an upper bound in the proof and the last requirement in order for the path to be compatible with the definition of a path used to perform a Lutz twist. \\ \par

As it is obvious from their definition, $h_1(r)$, $h_{2,u}(r)$ are not even differentiable at $r=\varepsilon_0$. We will use a standard mollifier called the truncated Gaussian distribution in order to make both functions smooth. For some $\delta_0 << \varepsilon_0$, we consider the interval $(\varepsilon_0-\delta_0,\varepsilon_0+\delta_0)$ over which the smoothing will take place. The basic idea in what follows is to convolute our functions $h_1(r)$ and $h_{2,u}(r)$ with a truncated Gaussian supported in $(\varepsilon_0-\delta_0,\varepsilon_0+\delta_0)$ in order to define the respective smoothings $H_1(r)$ and $H_{2,u}(r)$. As it is evident here, capitalizing the notation means smoothing. \\ \par
In what follows, we will be using the truncated Gaussian

$$g(r;\mu,\sigma,a,b)=\frac{2}{\sigma}\frac{\frac{1}{\sqrt{2\pi}}\exp(-\frac{1}{2}(\frac{r-\mu}{\sigma})^2)}{(erf(\frac{b-\mu}{\sqrt{2}\sigma})-erf(\frac{a-\mu}{\sqrt{2}\sigma}))}$$
where the error function is given by $$erf(x)=\frac{2}{\sqrt{\pi}}\displaystyle\int_0^{x}e^{-t^2}dt$$

The first truncated Gaussian we will use in our case is the following

$$g(x):=\begin{cases}
\frac{5exp\big(\frac{-25(x-\varepsilon_0)^2}{2\delta_0^2}\big)}{\delta_0\sqrt{2}\Big(\displaystyle\int_{\frac{-5}{\sqrt{2}}}^{\frac{5}{\sqrt{2}}}e^{-t^2}dt\Big)} & \text{,  }  \varepsilon_0-\delta_0 \leq x \leq \varepsilon_0+\delta_0\\
0  & \text{, otherwise}

\end{cases}$$

This is basically the truncated Gaussian with carefully chosen parameters  $$g(x;\mu=\varepsilon_0,\sigma=\frac{\delta_0}{5},a=\varepsilon-\delta_0,b=\varepsilon+\delta_0)$$
Note that the non trivial part of this Gaussian can be written as a constant $O(\frac{1}{\delta_0})$ of order $\frac{1}{\delta_0}$ times the function $exp\big(\frac{-25(x-\varepsilon_0)^2}{2\delta_0^2}\big)$. Moreover, its derivative is of order $O(\frac{1}{\delta_0^3})$. We will use the notation $Gf$ which for any function $f$ stands for the convolution of $f$ with the specific truncated Gaussian chosen above. $Gf$ is smooth and is supported in $(\varepsilon_0-\delta_0,\varepsilon_0+\delta_0)$.\\ \par

We now define the smoothing of any function $f(r)$ supported in $[\varepsilon_0-\delta_0,\varepsilon_0+\delta_0]$.\\ \par
Let $B_1=(\varepsilon_0-\frac{8\delta_0}{7},\varepsilon_0-\frac{6\delta_0}{7})$ and $B_2=(\varepsilon_0+\frac{6\delta_0}{7},\varepsilon_0+\frac{8\delta_0}{7})$. Moreover, let $g_{B_1}=g(r;\varepsilon_0-\delta_0,\frac{\delta_0}{10},\varepsilon_0-\frac{8\delta_0}{7},\varepsilon_0-\frac{6\delta_0}{7})$ and $g_{B_2}=g(r;\varepsilon_0+\delta_0,\frac{\delta_0}{10},\varepsilon_0+\frac{6\delta_0}{7},\varepsilon_0+\frac{8\delta_0}{7})$. Using $g_{B_1},g_{B_2}$ we define the smooth cutoff functions 
$$\overline{\chi}_{B_i}(r)=\displaystyle\int_{\R}\chi_{B_i}(r-y)g_{B_i}(y)dy,\qquad i=1,2$$
where $\chi_{B_i}$ the characteristic function of the interval $B_i,\quad i=1,2$.\\ \par

Now the smoothing of any function $f$ supported in $(\varepsilon_0-\delta_0,\varepsilon_0+\delta_0)$ is defined as

\begin{gather*}
F(r):=(\overline{\chi}_{B_1(r)}+\overline{\chi}_{B_2}(r))f(r)+Gf(r)\\
=(\overline{\chi}_{B_1}(r)+\overline{\chi}_{B_2}(r))f(r)+\displaystyle\int_{\varepsilon_0-\delta_0}^{\varepsilon_0+\delta_0}g(r-t)f(t)dt
\end{gather*}

It is easy to check that if the inputs of the smoothing process are $h_1(r),h_{2,u}(r)$ and the outputs are $H_1(r),H_{2,u}(r)$ then the outputs are smooth functions that agree with $h_1(r)$ and $h_{2,u}(r)$ at the endpoints of the interval $[\varepsilon_0-\delta_0,\varepsilon_0+\delta_0]$. So, in $[\varepsilon_0-\delta_0,\varepsilon_0+\delta_0]$ we replace the continuous path $(h_1(r),h_{2,u}(r))$ by the smooth path $(H_1(r),H_{2,u}(r))$. The following picture is illuminating.

\begin{figure}[h!]
    \centering
    \includegraphics[scale=0.9]{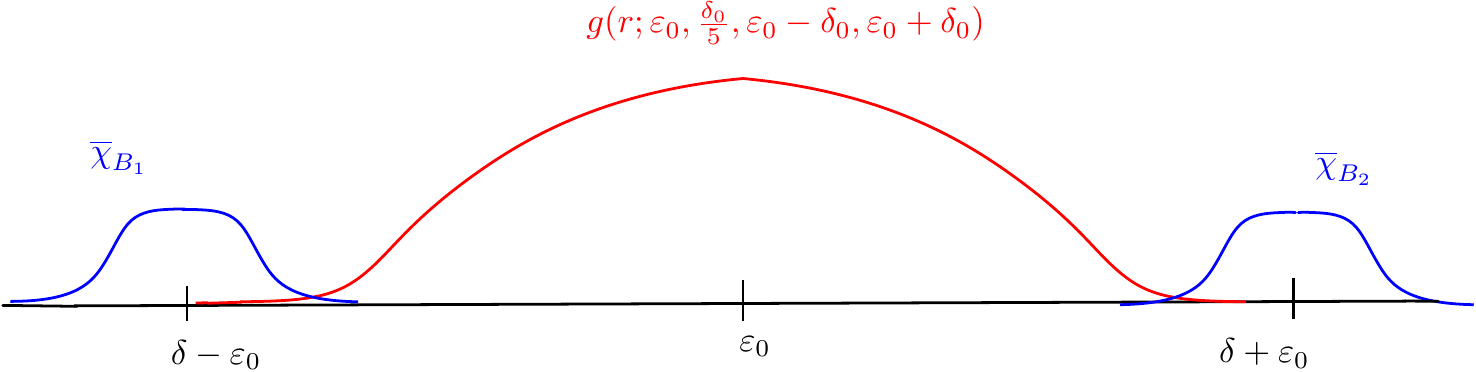}
    \caption{The functions used for smoothing}
    \label{fig:smoothing}
\end{figure}

~\par

Our task is to bound (\ref{tobebounded}). First, the maximum of the integrand  in (\ref{tobebounded}) cannot occur if $r\in [0,\varepsilon_0] \cup [\frac{\varepsilon}{2},\varepsilon)$ since there we have either $\frac{d}{du}(h_{2,u}(r))=0$ or  $h_1'(r)=0$. Next, if $\varepsilon_0 +\delta_0 \leq r \leq \frac{\varepsilon}{2}$ then we get

\begin{gather*}
    h_1'(r)=(1+\delta_1)2\pi \sin(2\pi r)\quad \text{ and } \quad 
    D_u(r)=\frac{u(1+\delta_1)(1+\delta_2)}{1+\delta\mu(\theta_-)}
\end{gather*}
so this yields

\begin{gather*}
    \frac{-h_1'(r)}{D_u(r)}=\frac{(1+\delta\mu(\theta_-))2\pi \sin(2\pi r)(1+\delta_2)}{u}
\end{gather*}
 and thus (\ref{tobebounded}) becomes
 
\begin{gather*}
\int_s^t\Big|\Big| \frac{d}{du}(h_{2,u}(r))\cdot\Big(\frac{(1+\delta\mu(\theta_-))2\pi \sin(2\pi r)(1+\delta_2)}{u}\Big)\Big|\Big|_{\infty}du \\
\leq |(1+\delta\mu(\theta_-))2\pi (1+\delta_2)| \int_s^t\Big|\Big| \frac{d}{du}(h_{2,u}(r))\cdot\Big(\frac{1}{u}\Big)\Big|\Big|_{\infty}du
\\
=|(1+\delta\mu(\theta_-))2\pi (1+\delta_2)| \int_s^t \Big| \frac{d}{du}(h_{2,u}(r_+))\cdot\Big(\frac{1}{(1+\delta\mu(\theta_-))2\pi (1+\delta_2)h_{2,u}(r_+)}\Big)\Big|du
\\
=\int_s^t \Big| \frac{d}{du}(h_{2,u}(r_+))\cdot\Big(\frac{1}{h_{2,u}(r_+)}\Big)\Big|du\\ =|\ln(h_{2,t}(r_+))-\ln(h_{2,s}(r_+))|
=|\ln(l_t)-\ln(l_s)|
\end{gather*}

The first equality above comes from the fact that $\Big|\Big| \frac{d}{du}(h_{2,u}(r))\Big|\Big|_\infty$ occurs at $r=r_+$. Recall also that the $l$-invariant is designed to be $h_{2,u}(r_+)=\frac{(1+\delta_1)u\sin(2\pi r)}{2\pi(1+\delta \mu(\theta_-))}$. So altering $u$ has the effect of altering the $l$-invariant. For brevity, we denote the $l$-invariant at time $t$ by $l_t$. \\ \par

The only thing to show now is that the smoothing does not have the effect that $\big|\frac{-H_1'(r)}{\mathcal{D}_u(r)}\big|>\frac{1}{u}$ in $(\varepsilon_0-\delta_0,\varepsilon_0+\delta_0)$, where $\D_u(r)=H_1(r)H_{2,u}'(r)-H_1'(r)H_{2,u}(r)$ the smooth version of $D_u(r)$. What we will show is that $\delta_0$ can be chosen sufficiently small in order to ensure this as the quotient is of order $\delta_0$.\\ \par

For $r \in [\varepsilon_0-\delta_0,\varepsilon_0+\delta_0]$ we have 

$$H_1'(r)=(\overline{\chi}_{B_1}'(r)+\overline{\chi}_{B_2}'(r))h_1(r)+(\overline{\chi}_{B_1}(r)+\overline{\chi}_{B_2}(r))h_1'(r)+Gh_1'(r)$$
and
$$H_{2,u}'(r)=(\overline{\chi}_{B_1}'(r)+\overline{\chi}_{B_2}'(r))h_{2,u}(r)+(\overline{\chi}_{B_1}(r)+\overline{\chi}_{B_2}(r))h_{2,u}'(r)+Gh_{2,u}'(r)$$
Now, by performing the long and necessary calculations we have that 
$$\displaystyle\lim_{\delta_0\rightarrow 0}\Big|\frac{-H_1'(r)}{\D_u(r)}\Big|=0$$
so $\delta_0$ can be chosen sufficiently small in order to ensure that $\Big| \frac{-H_1'(r)}{\D_u(r)}\Big| \leq \frac{1}{u}$, $\forall r\in [0,\varepsilon)$.

\vspace{15pt}

Summarizing our results after this analysis, we obtain
\begin{equation}\label{actualbound}
||\ln{f_t}-\ln{f_s}||_{\infty} \leq |\ln(l_t)-\ln(l_s)|
\end{equation}

Evaluating at $t=2,s=1$, (i.e. considering the family $\{\alpha_{1,l}\}_{l=l_1}^{l=l_2}$) and using \ref{rhs1} we get 
\begin{equation}\label{RHSL}
    d_{CBM}(\gamma,\beta) \leq ||\ln{f_t}-\ln{f_s}||_{\infty} \leq |\ln(l_2)-\ln(l_1)| = |\ln(l(\beta))-\ln(l(\gamma))|
\end{equation}
as required.  \\ \par 

Using (\ref{LHS}), (\ref{keyinequality}), (\ref{RHSV}), (\ref{RHSL}) for the first contact form $\alpha=\alpha_{k_1,l_1}$, the middle contact form $\gamma=\alpha_{k_2,\sqrt{\frac{k_2}{k_1}}l_1}$ and the second contact form $\beta=\alpha_{k_2,l_2}$ we get

\begin{gather}
d_{\infty}(\vec{x},\vec{y})=d_{\infty}((\sqrt{k_1},l_1),(\sqrt{k_2},l_2))\leq d_{CBM}(\alpha,\beta) \leq d_{CBM}(\alpha,\gamma)+d_{CBM}(\gamma,\beta) \nonumber \\
\leq |\ln(\sqrt{k_2})-\ln(\sqrt{k_1})|+|\ln(\sqrt{\frac{k_2}{k_1}}l_1)-\ln(l_2)| \nonumber\\
\leq 2|\ln(\sqrt{k_2})-\ln(\sqrt{k_1})|+|\ln(l_2)-\ln(l_1)|\leq 2d_{\infty}(\Vec{x},\Vec{y})
\end{gather}
or in the weaker but more symmetric form

\begin{equation}
   \frac{1}{2} d_{\infty}(\vec{x},\vec{y})\leq d_{CBM}(\alpha,\beta) \leq 2d_{\infty}(\Vec{x},\Vec{y})
\end{equation}
\qed

\vspace{20pt}

\section{Extension of the result to higher dimensions} \label{Higherdim}
 It was not known for quite a while what is the natural generalization of the notion of overtwistedness in higher dimensions. For instance, the question of how a higher dimensional analogue of an overtwisted disk should look like was very recently answered in \cite{MR3455235}, along the process of classifying and establishing an h-principle for overtwisted structures in higher dimensions. The most compatible, with the theory known so far (in the sense that one can show existence of $D^{2n}_{ot}$, bLobs and plastikstufes after performing the twist) version of generalized Lutz twists appeared recently in \cite{2016arXiv161009672A}. In this work, Adachi describes the construction of the higher dimensional analogue of the Lutz tube and instead of considering a contact form which generalizes the 3-dimensional $h_1(r)d\theta+h_2(r)d\phi$, a confoliation 1-form on $S^1\times \mathbb{R}^{2n}$ is picked, called $\omega_{tw}$, which in 3 dimensions is forced to be a contact form. This confoliation is shown to be conductive, thus thanks to \cite{altschuler2000}, there is a contact form which is $C^{\infty}$-close to this confoliation form Although Adachi's construction seems to link nicely with the development of the theory so far (i.e. for instance he proves existence of overtwisted disks as defined in \cite{MR3455235}), it is difficult to work with his local models and eventually calculate the dynamics of the Reeb vector field.\\ \par

Etnyre and Pancholi, in \cite{EP1} and \cite{EP2}, provided  a generalization of the Lutz twist having in mind the notion of a plastikstufe. Although their construction is very explicit in terms of dynamics, it requires a very hard work understanding the pseudoholomorphic curves involved. The approach that best fits the scope of this work is to use the description of contact manifolds as open book decompositions and what we need in this work was developed in \cite{MR2646902}. We describe it and we relate it to our goal. 

\subsection{Strategy for the extension of the result} \text{} \\ \par
Bourgeois and Van Koert in \cite{MR2646902} proved that negatively stabilized open books have vanishing contact homology. Casals, Murphy and Presas recently showed in \cite{MR3904160}  that such contact manifolds are overtwisted. The notion of open books which is easily generalized to higher dimensions will be the most fundamental tool in this extension to the higher dimensional cases. This section is inspired and follows the work from \cite{MR2646902}. \\ \par

Stabilization means the following. Let $P$ be the $(2n-2)$-dimensional page of the open book of the contact manifold $(M,\xi)$. Let $L$ be a Lagrangian $(n-1)$-disk with $\bd L\subset \bd P$. Suppose that the monodromy of the open book is the identity on a neighborhood of $L$. Attaching a Weinstein $(n-1)$-handle to $P$ along $\bd L$ we get a Lagrangian sphere in the new page $\widetilde{P}$. Choosing as monodromy a right-handed Dehn twist along this Lagrangian sphere and composing with the original monodromy on the rest of the page, we obtain a contact manifold. If we choose a boundary parallel Lagrangian ball $L$ we have that the resulting contact manifold is contactomorphic to $(M,\xi)$. It is conjectured that this holds more generally, i.e. removing the assumption that the Lagrangian ball is boundary parallel. In our case, we need an alternative description via contact connected sums (or Murasugi sums) and this is explained in \cite{MR2646902},
 section 7. We briefly describe the idea.\\ \par
 
Negative stabilization corresponds to contact connected summing our initial contact manifold $(M^{2n-1},\eta)$ with a special contact manifold $(S^{2n-1},\alpha_L)$ whose construction will be reviewed. This latter manifold is viewed as an open book decomposition with pages being $T^*S^{n-1}$ and monodromy a left-handed Dehn-Seidel twist. The situation highly resembles the 3-dimensional case as the contact form we consider is a generalization of the 3-dimensional overtwisted contact form we used when performing the Lutz twist. In fact, we still work locally and the 3-dimensional situation is recovered quite naturally as it is explained in \cite{MR2646902}, at the end of section 6. It will be enough to work on $(S^{2n-1},\alpha_L)$. By this we mean that we will describe an analogous to the 3-dimensional case modification on $(S^{2n-1},\alpha_L)$ and then we will connect sum with the contact manifold of interest $(M^{2n-1},\eta)$. It is important to mention that the modification and the connected sum processes commute. \\ \par

One might be worried that the process of connected summing has the effect that the homotopy type of the original hyperplane field $\ker(\eta)$ changes and thus the promised result cannot be extended to higher dimensions. This concern can be easily lifted by the fact that the set of almost contact structures (and in particular of contact structures) forms a group under connected sum. Thus, after performing the connected sum with the special contact $(S^{2n-1},\xi_-=\ker(\alpha_L))$, we can invert by connected summing with its inverse $(S^{2n-1},\xi_+)$ and cancel the possible alteration of the homotopy type of $\ker(\eta)$. Note that this inversion will not affect our calculations as we will primarily be interested in the behavior of the lowest action orbit. We can adjust for the lowest action orbit of the inverting $(S^{2n-1},\xi_+)$ to have arbitrarily large action. \\ \par

This is also good point to explain how the contact connected sum procedure affects the dynamics. Orbits that do not pass through the connected sum region are not affected. Furthermore, there are orbits that lie entirely in the connecting tube which are called tube orbits and there are also wandering orbits which start at the manifold $M$ go through the tube to $S^{2n-1}$ and return to $M$. Their action can be made arbitrarily large so they are of least importance for us. Moreover, Ustilovsky in his thesis \cite{MR2700259} showed that we can choose a contact form on the tube so that all tube
orbits lie within a contact sphere in the middle of the tube and all such orbits have odd degree $k\geq 2n-3$. Thus in dimensions $\geq 5$, we don't have any orbits of degree 1 within the tube. Recall that the 3-dimensional case was handled previously. It turns out that our focus has to be purely on $S^{2n-1}$.

\subsection{The Bourgeois-Van Koert open book construction} \label{OBA} \text{} \\ \par
We now describe a special case of the construction from \cite{MR2646902} where an open book decomposition for $S^{2n-1}$ is provided. This case is the most relevant to our work. Since we will need Dehn twists on $(T^*S^{n-1},d\lambda)$ with $\lambda=\vec{p}d\vec{q}$ in order to describe the construction, we briefly recall them.\\ \par

We start with the auxiliary map $$\sigma_t(\vec{q},\vec{p})=\begin{pmatrix}
\cos(t) & -\frac{1}{|\vec{p}|}\sin(t)\\
|\vec{p}|\sin(t) & \cos(t)
\end{pmatrix}
\begin{pmatrix}
\vec{q}\\
\vec{p}
\end{pmatrix}$$
Although the construction that appears in \cite{MR2646902} is more general, we will only need the special case of a $(-1)$-fold Dehn twist. We will define a function $g$ using an auxiliary function $\widetilde{g}$. Let $\widetilde{g}$ be a function satisfying
\begin{itemize}
    \item $\widetilde{g}(0)=-\pi$.
    \item Fix $p_0>0$. For $|\vec{p}|>0$, it increases to 0 at $\widetilde{g}(p_0)$.
    \item For $|\vec{p}|\geq p_0$, $\widetilde{g}(|\vec{p}|)=0$.
\end{itemize}
For a small $\widetilde{\varepsilon}>0$, we define $g=\widetilde{g}+\widetilde{\varepsilon}|\vec{p}|$. A helpful figure illustrating $g$ is figure \ref{fig:Dehn}.
\begin{figure}
    \centering
    \includegraphics[scale=0.8]{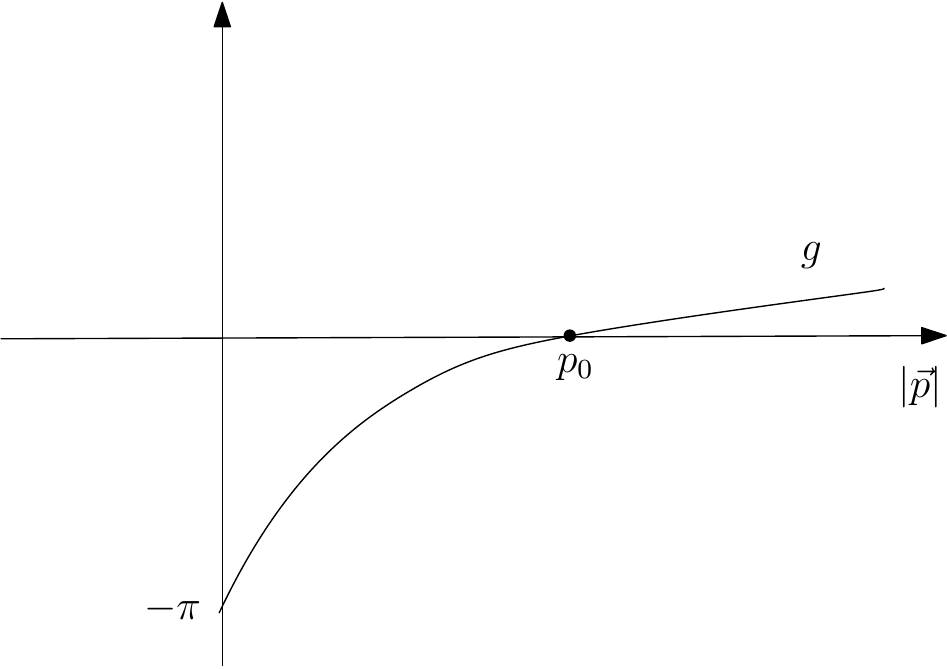}
    \caption{The function $g$ used in the definition of a $(-1)$-fold Dehn twist}
  \label{fig:Dehn}
\end{figure}\\ \par

We define the self diffeomorphism of $T^*S^{n-1}$ given by $$\tau(\vec{q},\vec{p})=\sigma_{g(|\vec{p}|)}(\vec{q},\vec{p})$$
This diffeomorphism is symplectic as $\tau^*d\lambda=d\lambda$ since we have $\tau^*\lambda=\lambda+|\vec{p}|d(g|\vec{p}|)$. We observe that $\lambda-\tau^*\lambda$ is exact and a primitive for this difference is $$h(|\vec{p}|)=1+\displaystyle\int_0^{|\vec{p}|}sg'(s)ds$$

We have the mapping torus for $T^*S^{n-1}$ defined as follows. First, consider the map

\begin{gather*}
\phi: T^*S^{n-1} \times \R \rightarrow T^*S^{n-1}\times \R\\
(\vec{q},\vec{p};\varphi)\mapsto (\tau(\vec{q},\vec{p});\varphi+h(|\vec{p}|))
\end{gather*}
This map preserves the contact form $\alpha=d\varphi+\vec{p}d\vec{q}$ on $T^*S^{n-1}\times \R$,  so we obtain a contact structure on the mapping torus

$$A:=T^*S^{n-1}\times \R/\varphi$$
All we have to do to complete the construction is to glue in the binding. In fact, as it is more relevant to our work, it is better to change point of view and focus on a neighborhood of the binding rather than the mapping torus just described. Another description using a different mapping torus clarifies the situation. \\ \par

Consider $$\widetilde{A}=((T^*S^{n-1}-0)\times \R)/(x,\varphi)\sim (x,\varphi+1)\simeq (T^*S^{n-1}-0)\times S^1$$

We have the map
\begin{gather*}
\psi:\widetilde{A}\rightarrow A\\
((\vec{q},\vec{p});\varphi)\mapsto (\sigma_{\varphi g}(\vec{q},\vec{p});h(|\vec{p}|)\varphi)
\end{gather*}

which is a diffeomorphism onto its image and allows us to think of the construction explicitly \say{standing on the binding}. In fact, we can choose a neighborhood of the binding so large that it covers the whole $\widetilde{A}$, so $B$ will describe the entire contact manifold except for a set of positive codimension which is precisely the set in the mapping torus corresponding to the zero section. Computations effectively only take place on $B$.

We have a contact form $\widetilde{\alpha}$ on $\widetilde{A}$ by pulling back $\alpha$ using $\psi$. Indeed, $$\widetilde{\alpha}=\psi^*\alpha=\widetilde{h}(|\vec{p}|)d\varphi+\vec{p}d\vec{q}$$
where $\widetilde{h}(|\vec{p}|)=1-\displaystyle\int_0^{|\vec{p}|}g(s)ds$.

Denote a neighborhood of the binding by $B:=ST^*S^{n-1} \times D^2$. As mentioned before, this neighborhood covers all of $S^{2n-1}$ except for the set in the mapping torus corresponding to the zero section of the pages. We think of $T^*S^{n-1}\subset \R^{2n}$ described by the equations.
$$\vec{p}\cdot \vec{p}=1, \quad \vec{q}\cdot \vec{q}=0, \quad \vec{p}\cdot \vec{q}=0$$
On this neighborhood of the binding we have the contact form $$\alpha=h_1(r)\lambda+h_2(r)d\varphi$$
where $\lambda$ is the restriction of the canonical form $\vec{p}d\vec{q}$ to $ST^*S^{n-1}$ and $(r,\varphi)$ the polar coordinates on $D^2$. The functions $h_1,h_2$ are chosen so that $\alpha$ is a contact form and matches $\widetilde{\alpha}$ in a collar neighborhood of the boundary, i.e. close to the set corresponding to the zero section of the pages. For $\alpha$ to be a contact form we need to assume that $h_1\neq 0$ and $\frac{h_1h_2'-h_1'h_2}{r}=\frac{D(r)}{r}\neq 0$. Such choice is depicted in the figure \ref{fig:Hdimh1h2}.  \\ \par

\begin{figure}
    \centering
    \includegraphics{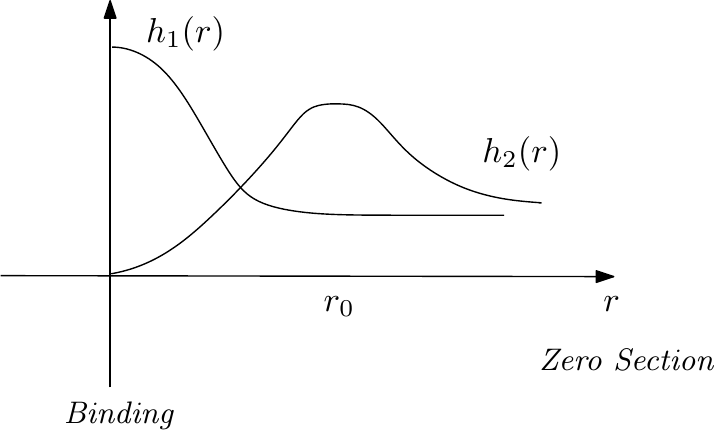}
    \caption{The functions $h_1(r)$ and $h_2(r)$}
    \label{fig:Hdimh1h2}
\end{figure}

It is both interesting and useful to digress a bit in order to write down explicitly how the Reeb orbits in $B$ look like. We have that the Reeb field on the binding is 
$$R_{\lambda}=\vec{p}\bd_{\vec{q}}-\vec{q} \bd_{\vec{p}}$$
and in a neighborhood of the binding is

$$R_{\alpha}=\frac{1}{D(r)}\Big( h_2'(r)R_\lambda-h_1'(r)\bd_\varphi \Big)$$

So the orbits are

$$x(t)=\Big(\theta_0+\frac{h_2'(r)}{D(r)}t, r, \varphi_0-\frac{h_1'(r)}{D(r)}t \Big)$$

where $\theta$ the coordinate corresponding to the geodesic flow. We observe that we have a closed orbit whenever

$$\frac{h_1'(r)}{2\pi h_2'(r)}=\frac{p}{q}\in \Q \cup \{\infty\}$$

and the actions are

$$T=q\frac{D(r)}{h_2'(r)}=2\pi p \frac{D(r)}{h_1'(r)}$$

where sign($p$)$=$sign($h_1'(r)$), sign($q$)$=$sign($h_2'(r)$) and whenever $h_1'(r)=p=0$ or $h_2'(r)=q=0$ pick the one that makes sense. \\ \par

Taking a step back from describing the dynamics of the form, thus far a cooriented contact structure on $S^{2n-1}$ given as the kernel of the 1-form $\alpha$ is constructed. This contact sphere is denoted by $(S^{2n-1},\alpha)$. As shown by Bourgeois and Van Koert, after perturbation using a Morse function, there exists a unique closed orbit $\gamma_0$ which has linking number with the binding equal to 1. The choice of functions $h_1,h_2$ and $g$ can be adjusted so that $\gamma_0$ is the lowest action orbit, thus $\gamma_0$ can only bound holomorphic planes. What follows are statements proved in \cite{MR2646902}
\begin{itemize}
\item $\gamma_0$ is the lowest action orbit with action $\mathcal{A}(\gamma_0)\simeq 2\pi h_2(r_0)$.
\item $\gamma_0$ bounds a unique holomorphic plane.
\item $\gamma_0$ has degree 1.
\end{itemize}

\begin{Rem}
The fact that the action is approximately and not precisely $2\pi h_2(r_0)$ has to do with the fact that orbits are degenerate so we have to use a perturbation as in \cite{Bourgeois}, section 2.2.
\end{Rem}

From the items above the next proposition quickly follows.
\begin{Prop}
The $l$-invariant of $\alpha$ is $\mathcal{A}(\gamma_0)\simeq 2\pi h_2(r_0)$.
\end{Prop}

\subsection{Modification of contact form} \text{} \\ \par
Before we start with any modification it is important to recall the general direction. Our goal is to begin with a contact manifold $(M,\xi)$, connect sum with a special modifying contact $S^{2n-1}$ which makes it overtwisted and then compensate for possible alteration of the homotopy type of the plane field by connected summing with the inverting $S^{2n-1}$. As explained earlier, since almost contact structures form a group under connected sum, we can compensate for the possible change of the homotopy type of $\ker(\eta)$. After this, we bi-Lipschitz embed the space of contact forms supporting this contact structure on $M\#S^{2n-1}\#S^{2n-1}$ inside $\R^2$. Since the modification only takes place away of the zero section of the pages of the open book for $S^{2n-1}$ and the Darboux ball we use to connect sum is close to the region corresponding to the zero section of the pages, it is enough to describe the modification purely on $S^{2n-1}$. Again, we have some restrictions to this modification. If we start with an algebraically overtwisted contact manifold we can make the $l$-invariant as large as the primitive for the unit of the contact homology algebra before connected summing. Moreover, the modified $l$-invariant can be as large as the lowest action orbit on $M$ and the modifying $S^{2n-1}$. We can arrange for the actions of the orbits of the compensating $S^{2n-1}$ to be arbitrarily large so as they do not affect our calculations. We remark that these restrictions do not affect our result since it is a large-scale geometry result. \\ \par
Similarly to the 3-dimensional case we modify $\alpha$ on $S^{2n-1}$ as follows

$$
\alpha_{k,l}=F((\ln(\sqrt[n]{k}),\ln(l)))=\sqrt[n]{k}(h_1(r)\lambda+h_{2,l}(r)d\phi)
$$
where $h_{2,l}(r)$ is a smooth function which is identical to $h_2(r)$ outside of a small neighborhood of $r_0$ and $h_{2,l}(r_0)=\frac{l}{2\pi(1+\delta \mu(\theta_-))}$, where $\mu(\theta_-)$ is the value of the Morse function used for the perturbation at the critical point $\theta_-$ which corresponds to the lowest action orbit $\gamma_0$.
Following the proof from \cite{MR2646902}, we have for the lowest action orbit $\gamma_0$ that
\begin{itemize}
    \item $\mathcal{A}(\gamma_0)=(1+\delta \mu(\theta_-)) 2\pi \sqrt[n]{k}\cdot h_{2,l}(r_0)$, for $0<\delta<<1$ which comes from the perturbation.
    \item The degree of $\gamma_0$ is equal to 1.
    \item $\gamma_0$ bounds a unique holomorphic plane.
\end{itemize}
In summary, we have that the $l$-invariant $l(\alpha_{k,l})=\mathcal{A}(\gamma_0)=(1+\delta)2\pi \sqrt[n]{k}\cdot h_{2,l}(r_0)$. Note that when modifying it we can normalize in order for $k$ to be equal to 1.\\ \par

Essentially the proof from subsection \ref{pfofqiemb} carries over to the higher dimensional case with little modification. We briefly describe the argument. The following two lemmata, as in 3 dimensions, are essential for the proof. First, a quick calculation which is almost identical to the proof of lemma \ref{volume} shows
\begin{Lemma}
If $\alpha \prec \beta$, then $Vol((S^{2n-1},\alpha))\leq Vol((S^{2n-1},\beta))$. Also, $Vol(S^{2n-1},C \cdot \alpha)=C^n \cdot Vol(S^{2n-1},\alpha)$.
\end{Lemma}
Moreover, since the proof of lemma \ref{linvineq} does not depend on the dimension, we have that it still holds. We recall it for clarity.
\begin{Lemma}
If $\alpha \prec \beta$, then $l(\alpha) \leq l(\beta)$. Also, $l(C\cdot \alpha)=C \cdot l(\alpha)$
\end{Lemma}

\subsubsection{Left inequality}
Letting $\alpha=\alpha_{k_1,l_1}$, $\beta=\alpha_{k_2,l_2}$ and $(\vec{x},\vec{y})=((\ln(\sqrt[n]{k_1}),\ln(l_1)),(\ln(\sqrt[n]{k_2}),\ln(l_2)))$ and following the same calculations as in subsection \ref{lineq} and using the previous 2 adapted lemmata, we obtain
$$|\ln(\sqrt[n]{k_2})-\ln(\sqrt[n]{k_1})|\leq d_{CBM}(\alpha,\beta)$$
and
$$|\ln(l_2)-\ln(l_1)|\leq d_{CBM}(\alpha,\beta)$$
Thus,
$$d_{\infty}(\vec{x},\vec{y})\leq d_{CBM}(\alpha,\beta)$$
The only difference comes from the dimension dependence of the volume so we have to work with the $n^{th}$ root instead of a square root.

\subsubsection{Right inequality}
The spirit of the proof of the right-hand inequality is identical to the 3-dimensional case. We still have to use Gray's stability theorem in order to obtain the proper bounds, yet since we are not performing a Lutz twist the analysis of the modification is slightly different. Again we let $\vec{x}=(\sqrt[n]{k_1},l_1)$ and $\vec{y}=(\sqrt[n]{k_2},l_2)$. We will use an analogous triangle inequality
\begin{gather}
    d_{CBM}(F(\Vec{x}),F(\Vec{y}))=d_{CBM}((\sqrt[n]{k_1},l_1),(\sqrt[n]{k_2},l_2)) \leq \nonumber \\ d_{CBM}\Big(\Big(\sqrt[n]{k_1},l_1\Big),\Big(\sqrt[n]{k_2},\sqrt[n]{\frac{k_2}{k_1}}l_1\Big)\Big)+d_{CBM}\Big(\Big(\sqrt[n]{k_2},\sqrt[n]{\frac{k_2}{k_1}}l_1\Big),\Big(\sqrt[n]{k_2},l_2\Big)\Big) 
    \label{keyinequalityh}
\end{gather}
Note that as in the 3-dimensional case we abuse the notation in order to avoid making it too heavy. This inequality again comes from the fact that modification of volume has a small effect on the $l$-invariant. As in the previous lower dimensional case, we can use the trick of the compensating function in order to modify the $l$-invariant of our contact form without modifying the volume. Similarly to the 3-dimensional case, letting
\begin{gather*}
    \alpha=F((\ln(\sqrt[n]{k_1}),\ln(l_1)))\\
    \beta=F((\ln(\sqrt[n]{k_2}),\ln(l_2)))\\
    \gamma=F\Big(\Big(\ln(\sqrt[n]{k_2}),\sqrt[n]{\frac{k_2}{k_1}}\ln(l_2)\Big)\Big)
\end{gather*}
we obtain by purely using the volume modification as in the 3-dimensional case

\begin{gather}
d_{CBM}(\alpha,\gamma)\leq \Big|\sqrt[n]{k_2}-\sqrt[n]{k_1}\Big|
\end{gather}

We now need to study how the $l$-invariant modification affects our embedding. Namely, obtain bounds when travelling from the form $\gamma$ to the form $\beta$. This is again done by using Gray's stability theorem. Similarly to 3 dimensions, the quantity we need to provide bounds for is

\begin{equation}\label{tobeboundedh}
||\ln{f_t}-\ln{f_s}||_{\infty} \leq \int_s^t \Big|\Big|\frac{d}{du}(h_{2,u}(r))\cdot\Big(\frac{-h_1'(r)}{D_u(r)}\Big)\Big|\Big|_{\infty}du
\end{equation}
where $f_t$ the conformal factor coming from Gray's theorem and $D_u(r)=h_1(r)h_{2,u}'(r)-h_1'(r)h_{2,u}(r)$.
We have that $\frac{d}{du}(h_{2,u}(r))$ is zero outside of a neighborhood of $r_0$ and attains its maximum at $r_0$. One can adjust the slope of $h_1(r)$ near $r_0$ in order to make $||D_u(r)||_{\infty}$ attain its minimum at $r_0$. This can be done by requiring $h_1'(r)\simeq 0$ and $h_2'(r)\simeq \frac{h_1'(r)h_2(r)}{h_1(r)}$ in a neighborhood of $r_0$. Then, we get that 
$$\Big|\Big|\frac{d}{du}(h_{2,u}(r))\cdot\Big(\frac{-h_1'(r)}{D_u(r)}\Big)\Big|\Big|_{\infty}=\Big|\frac{d}{du}(h_{2,u}(r_0))\cdot\Big(\frac{-1}{h_{2,u}(r_0)}\Big)\Big|$$
Hence integrating (\ref{tobeboundedh}) with respect to $u$ we obtain
$$||\ln{f_t}-\ln{f_s}||_{\infty} \leq |\ln(l_t)-\ln(l_s)|$$
Applying now this result to the forms $\gamma$ and $\beta$ we get more precisely that
$$d_{CBM}(\gamma,\beta)\leq ||\ln{f_t}-\ln{f_s}||_{\infty} \leq |\ln(l(\gamma))-\ln(l(\beta))|$$
Combining the triangle inequality with the two inequalities obtained above we get
\begin{gather*}
 d_{CBM}(\alpha,\beta)\leq d_{CBM}(\alpha,\gamma)+d_{CBM}(\gamma,\beta)\\ \leq
 |\sqrt[n]{k_2}-\sqrt[n]{k_1}|+|\ln(l(\beta))-\ln(l(\gamma))| \\
 \leq |\sqrt[n]{k_2}-\sqrt[n]{k_1}|+\Big|\ln\Big(\sqrt[n]{\frac{k_2}{k_1}}l_1\Big)-\ln(l_2)\Big| \\ \leq 
 2|\sqrt[n]{k_2}-\sqrt[n]{k_1}|+|\ln(l_2)-\ln(l_1)|\leq 2d_{\infty}(\vec{x},\vec{y})
\end{gather*}
\\ \par

As a concluding remark, let us summarize that what we achieved, combining the results of the last and this subsection, is to extend the result to all odd dimensions greater than 3 and thus prove the theorem in its full generality.

\section{Remark on the possibility of more degrees of freedom and use of other homology theories} \label{Remarkonhigherdegrees}

The first degree of freedom (volume) was quite natural to consider. Also, the class corresponding to the empty word is a very special element of contact homology theories and in the overtwisted case, it is expected to vanish. It is therefore natural to ask about the filtration level for which it vanishes. The empty word class (known in most homology theories as the contact invariant, e.g. in ECH) appears at level 0 and vanishes at some higher level. Working with the contact homology algebra of overtwisted structures, all classes have to vanish at some filtration level, thus there are no semi-infinite bars in the barcode of the persistence module $CH^{\leq t}_*(Y,\lambda_{ot})$. One natural generalization of the idea would be to check what are the vanishing levels or lengths for the other finite bars and thus bi-Lipschitz or quasi-isometrically embed part of $\mathbb{R}^n$, where $n-1$ is the cardinality of the set of finite bars, into the space of contact forms. The algebra structure yields that most of the information is encoded to what we already have done. Namely, as explained in the introduction, the algebra structure yields that the largest length of a finite bar in the barcode is the $l$-invariant. \\ \par

This shows that using contact homology at least, in the overtwisted case there is no other powerful enough (or large enough) bar in the barcode to help us distinguish between two contact forms and have more degrees of freedom. This is not the case with ECH as one can see. ECH, in contrast with SFT theories, does not have a natural algebra structure (i.e. a multiplication that behaves well under grading and $\partial$). This turns out to be a helpful thing, as in the absence of algebra structure, no finite bar in principle controls the length of any other finite bar. The question regarding the $l$-invariant was about the freedom to alter the right endpoint of the corresponding vanishing bar (in this case the bar corresponding to the empty Reeb orbit) in the barcode of contact homology. This directly translates to filtered ECH and questions about its finite bars. \\ \par

Another direction may be provided when looking at ECH capacities. ECH capacities capture the information emerging from semi-infinite bars and not finite bars, yet one might be able to show that some process analogous to the Lutz twist in our setting helps us control them. Looking at the full ECH spectrum (see \cite{2010arXiv1005.2260H}) will be enough since parts of the symplectization are always trivial cobordisms. This is a matter of a future work direction.

\bibliographystyle{alpha}
\bibliography{main.bib}

\end{document}